\documentclass{amsart}

\usepackage{amsmath}
\usepackage{amssymb}
\usepackage{amsthm}
\usepackage{graphicx}
\usepackage{hyperref}
\usepackage{enumerate}
\usepackage{subfig}

\theoremstyle{plain} %
\newtheorem{theorem}{Theorem}[section]
\newtheorem{proposition}[theorem]{Proposition}
\newtheorem{lemma}[theorem]{Lemma}
\newtheorem{corollary}[theorem]{Corollary}

\newtheorem{assumption}[theorem]{Assumption}

\title{How smooth can convex chaotic billiard tables be?}
\author[L.A. Bunimovich]{Leonid A. Bunimovich}
\address{School of Mathematics, Georgia Institute of Technology, Atlanta GA 30332, USA}
\email{bunimovh@math.gatech.edu}
\thanks{L.B. was partially supported by the NSF grant DMS-1600568.}
\author[A. Grigo]{Alexander Grigo}
\address{University of Oklahoma, Department of Mathematics, Norman OK 73019, USA}
\email{grigo@math.ou.edu}
\thanks{A.G. was partially supported by the NSF grant DMS-1413428.}

\begin{document}

\begin{abstract}
  We solve the longstanding problem of smoothing a stadium billiard.
  Besides our results demonstrate why there were no clear conjectures
  how much the stadium's boundary must be smoothened to destroy
  chaotic dynamics. To do that we needed to extend standard KAM theory
  to analyze stability of periodic orbits, because of the low smoothness
  of the system. In fact, the stadium has a $C^1$ boundary, and we show
  that $C^2$ smoothing results in appearance of elliptic periodic orbits.
\end{abstract}

\maketitle

\tableofcontents

\section{Introduction}
\label{sect_intro}

In 1973 Lazutkin \cite{MR0328219} showed that for
strictly convex billiard tables $Q$ with a
boundary $\partial Q$ of class $C^{553}$
there exists an uncountable family of caustics near the boundary.
The presence of these caustics prevents the billiard dynamics from
being ergodic
and gives rise to nearly integrable motion close to the boundary.

Shortly after, in 1974 it was shown in
\cite{MR0342677,MR0357736} (see also \cite{MR530154}) that there are convex
billiard tables on which the billiard dynamics is hyperbolic and ergodic.
In billiards with focusing boundary components the mechanism that creates
the hyperbolicity is the mechanism of defocusing.
The most famous and best studied convex billiard in this class is the
stadium \cite{MR0357736,MR530154} %
whose boundary consists of two semi-circles connected by two straight line
segments.
For any (nonzero) length of the line segments the resulting billiard is
hyperbolic and ergodic. Note that the boundary of the stadium
is (globally) $C^1$.

These two results are in sharp contrast to each other. The difference
is due to the smoothness of the boundary of the billiard table.
If the boundary
of a convex table is smooth enough
then the presence of caustics prevents ergodicity.
This was first proved by Lazutkin for $C^{553}$ smooth boundaries.
Later R.Douady \cite{douady} lowered the smoothness requirement
to $C^6$ boundaries.
On the other hand, if the boundary is assumed
to be only $C^1$ smooth, then there are (continuous families of)
convex billiard tables with hyperbolic and ergodic billiard dynamics.
Therefore, the natural question is: which class of smoothness
of the boundary of the billiard table separates convex billiards
with completely chaotic dynamics from non-ergodic
dynamics with elliptic islands?
This question was raised immediately after the appearance of the stadium
billiard in 1974 by
Alekseev, Anosov, Arnold, Katok, Moser and others.
Nevertheless, this question remained open without clear conjectures.
The results of our paper answer this long standing problem.
Moreover, we clarify why there were no clearly stated conjectures.
Namely, we show that there are essentially two different mechanisms, which
create elliptic periodic orbits for short and long separations of
the two curved parts of the boundary, respectively.

Various classes of hyperbolic (chaotic) billiards with boundary containing
focusing components were constructed through the years
(see e.g. \cite{MR3607466} and references therein).
The main property/restriction though is that all
focusing components in such billiards must be absolutely focusing. Recall
that a focusing components $\Gamma$ is absolutely focusing if
any infinitesimal beam of parallel rays falling onto $\Gamma$
will leave $\Gamma$, after a series of consecutive
reflections, as a focusing beam \cite{MR1179172,MR1963965,MR1133266}.

In \cite{MR2563800} we showed that
as soon as non-absolutely focusing boundary components are
present the mechanism of defocusing fails, and hence
stable periodic orbits can appear.
Observe that focusing boundary components with vanishing curvature,
which appear when making the boundary of the stadium $C^2$ smooth, are
never absolutely focusing.
Therefore it seems natural to
attack the above question from this point of view, especially because
the hyperbolicity in the stadium billiard is generated
solely by the defocusing mechanism.
Note that since the curvature of each boundary component is assumed
to be continuous it follows that the contact between the straight line
segments and the curved boundary component is either $C^2$ or $C^1$,
depending on whether the curvature of the curved boundary component
vanishes or not. An intermediate regularity of type
$C^{1+\alpha}$ is not possible.

However, the problem addressed in this paper is much more challenging
than the one in \cite{MR2563800}. This is because there
one was allowed to design the boundary of the billiard table.
Here the problem is that we only allow for small perturbation of
a given boundary.
Moreover, due to the (insufficient) $C^2$ smoothness standard methods of
perturbation theory cannot be applied to establish
stability of periodic orbits in case of
small separations and new (low smoothness) results must
be obtained.
Among new technical tools and results we used to address
the low smoothness is a convenient system of coordinates which
could be used for analytic and numerical studies of more general
not sufficiently smooth Hamiltonian systems.

The structure of the paper is as follows. In
Section~\ref{sect_main_results} we formulate the problem
more precisely and state the main results.
In Section~\ref{sect_smallSeparations} we study the
dynamics of the billiard in $C^2$-stadia for small separations of
the two curved boundary segments. The corresponding analysis
of the dynamical properties for large separations is carried out
in Section~\ref{sect_large_separations}. The last
Section~\ref{sect_conclusions} contains some concluding remarks.

\section{Setup and Statement of the Main Results}
\label{sect_main_results}

The billiard tables we consider are obtained by
the following construction.
Take a semi-circle and replace some segment near its
endpoints by a convex curve whose curvature continuously
tends to zero near its endpoint, and is such that
the resulting focusing curve $\Gamma$
remains symmetric about the central axis just as the semi-circle was.
Consider two identical copies
of the focusing boundary component $\Gamma$.
By connecting their endpoints using straight line segments of length $L$
we obtain a continuous family (parametrized by $L$)
of $C^2$-smooth stadium-like billiard tables, or simply $C^2$-stadia.
See
\begin{figure}[ht!]
  \centering
  \includegraphics[width=8cm]{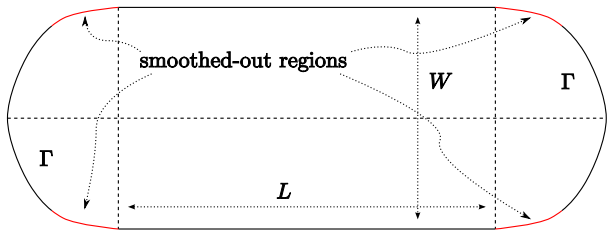}
  \caption{Smooth stadium like billiard tables.}
  \label{fig_general_table}
\end{figure}
Fig.~\ref{fig_general_table} for the illustration of this construction.

Our strategy of proof is the following.
First, we need to carefully construct a family of periodic orbits, which
have reflections off the smoothened part of the boundary. It is worthwhile
to also mention that not all such periodic orbits are stable. Then we
establish linear stability of the constructed periodic orbits, and
then their nonlinear stability using standard methods from normal
form theory.

The length $L$ of the straight line segments will be referred to as
the separation distance between the curved boundary components.
The part of the resulting billiard table, which is obtained by completing
the two parallel line segments to a rectangle will be referred to
as the rectangular channel.

In this paper we prove two main results about the dynamics of the billiard
in $C^2$-stadia.
The first result concerns small separation distances $L$.
Due to the assumed symmetry of $\Gamma$ the billiard table with
$L=0$ has a two-periodic orbit reflecting off the endpoints of
$\Gamma$. Since the curvature at the endpoints of
$\Gamma$ vanishes this two-periodic orbit is parabolic.
If this orbit was non-linearly stable one would hope that for small enough
$L$ there will be a stable periodic orbit on the resulting $C^2$-stadium.
However, the billiard map of the $C^2$-stadium is globally only of class
$C^1$, which is not regular enough for KAM-type of arguments to apply directly.
Since we are interested in boundary components $\Gamma$
that differ from semi-circles only on short segments near the endpoints
we introduce in
Lemma~\ref{lem_smoothening_characterizationCurvature} and
Corollary~\ref{cor_smoothening_explicitGamma}
a scaling parameter $\alpha$ and a ``shape function'' $h$ such that
the arc length $s_\alpha$ of the smoothened segment
of $\Gamma$
is proportional to $\alpha$, and the curvature ${\mathcal K}$
on the smoothened segment is given by scaling $h$, i.e.
${\mathcal K}(s)$ is obtained by evaluating
$ -h(\frac{s}{s_\alpha}) $.
In Section~\ref{sect_smallSeparations} we prove that for fixed $h$
the billiard in the $C^2$-stadium corresponding to small enough
$\alpha$ and small enough enough separations $L$
will have elliptic periodic orbits. More precisely, we prove the following

\begin{theorem}[Short Separations]
  \label{thm_shortEllipticOrbits}
  Suppose $h$ is such that the curvature of the smoothened boundary component
  decreases monotonically to zero so that its arc length is
  less than twice the arc length of the circular segment it replaces.
  Suppose further that $h$ satisfies the non-degeneracy condition
  \eqref{eqn_tag_condition_h} on page \pageref{eqn_tag_condition_h}.
  Then for all sufficiently small values of $\alpha$
  there exists $L_{h,\alpha}>0$ such that
  for any separation distance $L$ with $0 \leq L < L_{h,\alpha}$
  the resulting $C^2$-stadium has an elliptic periodic orbit.
\end{theorem}

Since the boundary of the smoothened stadium is $C^2$,
changing $L$ from zero to a nonzero value represents a $C^1$
perturbation of the billiard map.
Recall that for the usual stadium the corresponding perturbation
is only $C^0$. In view of perturbation theory the existence of elliptic
periodic orbits due to the increased smoothness of perturbation
not too surprising. However, as we already pointed out,
it seems not possible to obtain a proof of Theorem~\ref{thm_shortEllipticOrbits}
by a direct application of standard results of perturbation theory.

For the usual stadium billiard the billiard dynamics is
hyperbolic and ergodic
not only for arbitrarily short, but also for arbitrarily large
separations of the two semi-circles. This is due to the mechanism
of defocusing \cite{MR0357736,MR0342677,MR530154}, which actually
becomes more efficient the larger the separation distance is.
Therefore, our second main result Theorem~\ref{thm_large_separations}
is less expected than Theorem~\ref{thm_shortEllipticOrbits}.

\begin{theorem}[Large Separations]
  \label{thm_large_separations}
  $ $
  \begin{enumerate}[(i)]
    \item
      \label{thm_large_separations_intervals}
      Let $\Gamma$ be obtained
      from a semi-circle by smoothening
      a sufficiently small neighborhood of its endpoints.
      Then there exist constants $\delta, a, b>0$ and $N \in {\mathbb N}$
      such that for any separation distance $L$ with
      \begin{equation*}
        L \in \bigcup_{n\geq N} [a+ n\,b, a+n\,b + \delta]
      \end{equation*}
      the resulting $C^2$-smooth stadium-like billiard table has
      elliptic periodic orbits.

    \item
      \label{thm_large_separations_nonlinearStability}
      There exists an open (in the $C^5$ topology) set $G$
      of smoothenings of a semi-circle, that includes arbitrarily
      short smoothenings of that semi-circle, such that for any
      $\Gamma \in G$
      the corresponding $C^2$-smooth stadium-like billiard table
      has a nonlinearly stable periodic orbit
      for all separation distances $L\geq L_{\Gamma}$.
  \end{enumerate}
\end{theorem}
Unlike the elliptic periodic orbits for small separations,
the elliptic periodic orbits of
Theorem~\ref{thm_large_separations} are in no sense obtained by a
perturbation for the billiard table corresponding to $L=0$.
In fact, even the existence of elliptic periodic orbits
for some (arbitrary) large separation distance is not obvious.
So it is quite surprising that elliptic periodic orbits actually
exist for a large and very regular set of separation distances, which for
an open set of smoothenings is even the set of all
sufficiently large separation distances!

The key observation behind the construction of these orbits is
the fact that the curved boundary component $\Gamma$
is non-absolutely focusing. The elliptic
periodic orbits of Theorem~\ref{thm_large_separations}
naturally correspond to those constructed in \cite{MR2563800}.
Recall that the general design principle for hyperbolic billiards
\cite{MR0342677,MR0357736,MR848647,MR954676,MR1963965,MR1133266}
requires the focusing boundary components to be absolutely focusing.
This is because the only general principle known to ensure
hyperbolicity with focusing boundary components is the
defocusing mechanism.
Since Theorem~\ref{thm_large_separations} shows that
violating the absolutely focusing property of $\Gamma$
cannot be compensated for by making the separation distance arbitrarily large
we see that the absolute focusing property is not only necessary
for the general design principle of hyperbolic billiards with focusing
components \cite{MR2563800},
but also in such rigid
settings as in the stadium billiard family, which has only the
separation distance as a free parameter.

Finally, combining our results for large separation
with a special construction of stable periodic orbits for small
separation distances we show the following:
\begin{theorem}[All Separations]
  \label{thm_all_separations}
  There exists an open (in the $C^5$ topology) set $G$
  of smoothenings of a semi-circle, that includes arbitrarily
  short smoothenings of that semi-circle, such that for any
  $\Gamma \in G$
  the corresponding $C^2$-smooth stadium-like billiard table
  has a nonlinearly stable periodic orbit
  for all separation distances $L\geq 0$.
\end{theorem}

\section{Elliptic Periodic Orbits for Small Separations}
\label{sect_smallSeparations}

In this section we will prove Theorem~\ref{thm_shortEllipticOrbits}.
To explain the underlying idea,
consider the billiard table where the two curved boundary components
are not separated. In the case of the usual stadium this would be a circle.
Suppose that there exists an elliptic periodic orbit, which has no reflection
off either of the two points where the two curved boundary components
are joined. Then separating the two curved boundary components by adding
straight line segments to the boundary of the billiard table represents
a smooth perturbation of the billiard map in a neighborhood of the
periodic orbit. Therefore, the ellipticity of the periodic orbit implies that
it persists as an elliptic periodic orbit for all sufficiently small
separations of the two curved boundary components, and thus proves
Theorem~\ref{thm_shortEllipticOrbits}.

Therefore, the strategy of proof we adopt is the following.
In all of this section we assume that the two curved boundary components
are joined at their endpoints, where the curvature vanishes. This results
in a billiard table with a globally $C^2$ smooth boundary, which is
symmetric about the vertical and horizontal axis,
recall Fig.~\ref{fig_general_table} and see also
Fig.~\ref{fig_four-periodic} below.
In a first step we establish the existence
of a certain class of periodic orbits. In a second step we derive a criterion
on the smoothening that guarantees that these periodic orbits are elliptic.
From this we then obtain, as described above, the existence of
elliptic periodic orbits for all sufficiently small separation distances.

\subsection{Existence of certain periodic orbits}

First we explain the special periodic orbits we consider for the
analysis of the dynamics of the billiard for small separations.
As already mentioned before, we assume here as we will continue throughout
the entire section, that the two smoothened segments are
not separated.
Let $n\geq 0$ be some integer, and choose $n+2$ points
on $\Gamma$ in counterclockwise orientation.
Denote the corresponding arc length parameters by
$s_0, \ldots, s_{n+1}$.
Let the first and the last points in this sequence coincide
with the midpoints and the endpoint of $\Gamma$,
respectively. By changing $s_1, \ldots, s_n$,
while preserving their order, we can maximize the length of the
broken line connecting these $n+2$ points. And because the length
of such a broken line is a generating function for the billiard dynamics
\cite{MR1326374,MR2168892,MR2076302}
any maximizing configuration represents a segment
of a billiard trajectory $\gamma_s$.
Since the billiard table is symmetric about
the horizontal and the vertical axis, and since $\gamma_s$ has
the midpoint of $\Gamma$ as its first point and the
endpoint of $\Gamma$ as its last point it then
follows by symmetry that the $\gamma$ is part of a periodic
billiard trajectory of period $4\,(n+1)$. And this periodic
orbit is also symmetric about the vertical and horizontal symmetry axis
of the billiard table.

\begin{lemma}[Existence of certain periodic orbits -- qualitative version]
  \label{lem_existence_qualitative}
  For every fixed smoothening of the circular segments
  there exists an integer $n_\Gamma \geq 0$ such that
  for every $0 \leq n \leq n_\Gamma$
  there exists a $4\,(n+1)$-periodic orbit that is symmetric
  about the horizontal symmetry axis of the billiard table,
  but not the vertical symmetric axis of the billiard table. Furthermore,
  these orbits have reflections off the midpoints of the curved boundary
  components, but not off their endpoints.
  An illustration is such a periodic orbit is given in
  Fig.~\ref{fig_four-periodic} below.
\end{lemma}
\begin{proof}
  In the preceding discussion we noted that for any $n\geq 0$
  a periodic orbit (corresponding to $\gamma_s$) exists,
  which is symmetric about
  the horizontal and vertical symmetry axis, and with reflections
  off the midpoint and the endpoints of $\Gamma$.
  Clearly, for any smoothening of a semi-circle near its endpoints
  there exists an integer $n_\Gamma \geq 0$, such that
  for every $0 \leq n \leq n_\Gamma$ only
  the point corresponding to $s_{n+1}$ in $\gamma_s$
  is off the smoothened part of $\Gamma$. More precisely,
  $s_0, \ldots, s_n$ correspond to points
  off the interior of the circular part of $\Gamma$.

  Let $\psi_n$ denote the angle between the horizontal axis and the
  line segment connecting the point $\Gamma(s_n)$
  with the center point of the circle, see
  Fig.~\ref{fig_four-periodic}.
  And since $\Gamma(s_{n+1})$ corresponds
  to the endpoint of $\Gamma$, at which the curvature
  vanishes by assumption on the smoothening, it follows that
  increasing the angle $\psi_n$
  while keeping $\Gamma(s_0)$ fixed,
  will move the point
  $\Gamma(s_{n+2})$
  towards the vertical symmetry axis, and gives rise to
  a perturbed billiard trajectory piece $\gamma$.
  Of course, $\gamma$ is no longer part of a $4\,(n+1)$-periodic orbit anymore.
  And since the new point of reflection
  $\Gamma(s_{n+2})$
  is (for small increments of $\psi_n$) still off the circular part
  of the boundary we obtain that
  $\Gamma(s_{2\,n+2})$ moves away from the
  midpoint of the curved segment in clockwise direction (i.e. the
  points ``moves upwards''; cf Fig.~\ref{fig_four-periodic}).

  However, increasing $\psi_n$ even further will eventually
  force $\Gamma(s_{2\,n+2})$ to move in
  the opposite direction.
  By continuity it eventually must cross again the midpoint of the
  curved boundary component.
  And due to the symmetry of the billiard table about the horizontal
  axis this will result in a $(4\,(n+1)$-periodic orbit, which is
  symmetric about the horizontal symmetry axis of the billiard table,
  but which is not symmetric about the vertical.
  Such a periodic orbit is illustrated in
  Fig.~\ref{fig_four-periodic}).
\end{proof}

\begin{figure}[ht!]
  \includegraphics[width=6cm]{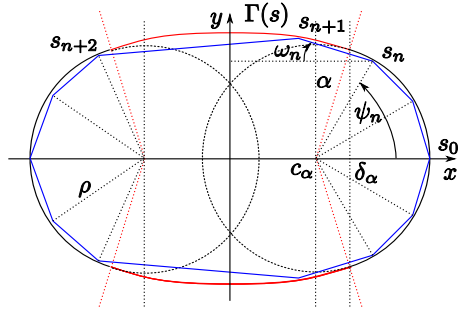}
  \caption{In blue is an illustration of a $4\,(n+1)$-periodic
  orbit, which is symmetric about the horizontal symmetry axis, but
  not the vertical one, and with reflections off the midpoints
  of the curved boundary components, but not their endpoints.
  The existence of such periodic orbits for any $n \geq 0$
  is shown by Lemma~\ref{lem_existence_qualitative}.}
  \label{fig_four-periodic}
\end{figure}

Note, however, that Lemma~\ref{lem_existence_qualitative}
does not make a statement about the location of
the $n+1$-st reflection (see Fig.~\ref{fig_four-periodic}).
In particular, there is no guarantee that a periodic orbit given
by Lemma~\ref{lem_existence_qualitative} has a reflection
off the smoothened part of the boundary.
In particular, the stability type of these periodic orbits
is not clear at all. In fact, it is not too difficult to
see that these periodic orbits could be hyperbolic, if the
part of the boundary corresponding to the smoothened region
is sufficiently long compared to the part of the circle it replaces.
In particular, if all reflections are
off the circular parts any such periodic orbit must be hyperbolic.

Therefore, in order to prove the existence of elliptic periodic
orbits for small separations as stated in
Theorem~\ref{thm_shortEllipticOrbits}
we need to quantify more precisely the smoothening of the boundary.
This is the aim of the next part.

\subsection{Definition of the smoothening and quantitative analysis of
certain periodic orbits}

Let $\Gamma$ be obtained by smoothening a
semi-circle of radius ${\rho}$
near its two endpoints. Let $\alpha$ denote
the angle corresponding to the part near the endpoint of the
semi-circle that got replaced by the smoothed out part, whose
arc length will be denoted by $s_\alpha$. Recall
Fig.~\ref{fig_four-periodic}.
The arc length parameter and the curvature along
$\Gamma$ will be
denoted by $s$ and ${\mathcal K}(s)$,
respectively. We count $s$ starting at zero at the
point where the circular part of $\Gamma$ ends
and the smoothened part starts, and $s$ increases in value
towards the endpoint of $\Gamma$.
Then the following Lemma~\ref{lem_smoothening_characterizationCurvature}
provides a characterization of the smoothened part of $\Gamma$.

\begin{lemma}[Normal form of the smoothened segment]
  \label{lem_smoothening_characterizationCurvature}
  There exists a continuous function
  \begin{align*}
    h &\colon [0,1] \to [0,\infty)
    \;,\qquad
    h(0) = 1
    \;,\quad
    h(1) = 0
  \end{align*}
  such that
  \begin{align*}
    {\mathcal K}(s)
    &=
    -\frac{1}{{\rho}}
    \,h\Big( \frac{s}{{\rho}\,\alpha}
    \int_0^1 h(\xi)\,d\xi \Big)
    \qquad\text{for all}\qquad
    0 \leq s \leq
    \frac{{\rho}\,\alpha}{ \int_0^1 h(\xi)\,d\xi }
    =
    s_\alpha
  \end{align*}
  holds for the curvature ${\mathcal K}$
  and the arc length $s_\alpha$
  of the smoothened part of $\Gamma$, respectively.
\end{lemma}
\begin{proof}
  Since every plane curve is, up to rigid motion, uniquely
  specified by its curvature, it is ${\mathcal K}$ that
  needs to be characterized.
  Since the smoothened part of $\Gamma$ of arc length
  $s_\alpha$ must (by the very definition of $\alpha$)
  enclose the angle $\alpha$ it follows that the curvature must satisfy
  \begin{align*}
    {\mathcal K}(0) &= -\frac{1}{{\rho}}
    \;,\quad
    {\mathcal K}(s_\alpha) =0
    \;,\quad
    \alpha = \int_0^{s_\alpha}
    -{\mathcal K}(s) \,ds
    \;,\quad
    {\mathcal K}(s) \leq 0
    \quad\forall\;0\leq s \leq s_\alpha
    \;.
  \end{align*}
  These conditions on ${\mathcal K}$ are equivalent
  to the existence of a continuous function
  $h \colon [0,1] \to [0,\infty)$ such that
  \begin{align*}
    h(0) = 1
    \;,\quad
    h(1) = 0
    \;,\qquad
    {\mathcal K}(s)
    =
    -\frac{1}{{\rho}}
    \,h\Big(\frac{s}{s_\alpha}\Big)
    \qquad 0 \leq s \leq s_\alpha
  \end{align*}
  in which case it follows that
  \begin{align*}
    s_\alpha
    =
    \frac{{\rho}\,\alpha}{ \int_0^1 h(\xi)\,d\xi }
    \;.
  \end{align*}
\end{proof}

\begin{corollary}[Explicit form of the smoothened segment]
  \label{cor_smoothening_explicitGamma}
  Let $h$ be as in Lemma~\ref{lem_smoothening_characterizationCurvature}.
  Then smoothened part of the boundary component $\Gamma$
  can be written as
  $
  \Gamma(s_\alpha\,\zeta)
  =
  \Gamma^\alpha(\zeta)
  $ for $0 \leq \zeta \leq 1$, where
  \begin{align*}
    \frac{1}{{\rho}}\,
    \Gamma_x^\alpha(\zeta)
    &=
    \alpha
    \,\frac{ \int_\zeta^1 \cos[\alpha\,\Theta(\xi)]\, d\xi
    }{ \int_0^1 h(\xi)\,d\xi }
    \;,\qquad
    \frac{1}{{\rho}}\,
    \Gamma_y^\alpha(\zeta)
    =
    \cos\alpha
    +
    \alpha
    \, \frac{ \int_0^\zeta \sin[\alpha\,\Theta(\xi)] \,d\xi
    }{ \int_0^1 h(\xi)\,d\xi }
    \\
    &\qquad\qquad\text{where}\qquad
    \Theta(\zeta)
    =
    \frac{ \int_\zeta^1 h(\xi)\,d\xi }{ \int_0^1 h(\xi)\,d\xi }
  \end{align*}
  for $0 \leq \zeta \leq 1$.
\end{corollary}
\begin{proof}
  By Lemma~\ref{lem_smoothening_characterizationCurvature}
  we obtain the expression for the curvature ${\mathcal K}$
  of $\Gamma$ in terms of the function $h$.
  And since the smoothened part of the boundary component
  $\Gamma$ can be written in terms of its curvature as
  \begin{align*}
    \Gamma(s)
    &=
    \begin{pmatrix}
      \delta_\alpha \\
      {\rho}\,\cos\alpha
    \end{pmatrix}
    +
    \int_0^s
    \begin{pmatrix}
      -\cos\theta(r) \\
      \sin\theta(r)
    \end{pmatrix}
    dr
    \;,\qquad
    \theta(s)
    =
    \alpha -
    \int_0^s -{\mathcal K}(r)\,dr
  \end{align*}
  for all $0 \leq s \leq s_\alpha$,
  it follows that
  \begin{align*}
    \begin{split}
      \Gamma^\alpha(\zeta)
      &=
      \begin{pmatrix}
        \delta_\alpha \\
        {\rho}\,\cos\alpha
      \end{pmatrix}
      +
      \frac{{\rho}\,\alpha}{ \int_0^1 h(\xi)\,d\xi }
      \int_0^\zeta
      \begin{pmatrix}
        -\cos[\alpha\,\Theta(\xi)] \\
        \sin[\alpha\,\Theta(\xi)]
      \end{pmatrix}
      d\xi
      \\
      \Theta(\zeta)
      &=
      1
      -
      \frac{ \int_0^\zeta h(\xi)\,d\xi }{ \int_0^1 h(\xi)\,d\xi }
      \equiv
      \frac{ \int_\zeta^1 h(\xi)\,d\xi }{ \int_0^1 h(\xi)\,d\xi }
    \end{split}
    \qquad\text{for all}\qquad
    0 \leq \zeta \leq 1
    \;.
  \end{align*}
  The expression for $\delta_\alpha$ is given
  by the condition $\Gamma_x(s_\alpha)=0$,
  see Fig.~\ref{fig_four-periodic}, i.e.
  \begin{align*}
    \frac{\delta_\alpha}{ {\rho}\,\alpha }
    &=
    \frac{ \int_0^1 \cos[\alpha\,\Theta(\xi)]\, d\xi
    }{ \int_0^1 h(\xi)\,d\xi }
    \;,\qquad
    \Theta(\zeta)
    =
    \frac{ \int_\zeta^1 h(\xi)\,d\xi }{ \int_0^1 h(\xi)\,d\xi }
  \end{align*}
  and hence shows the claimed expression for $\Gamma$
  after substituting this expression for $\delta_\alpha$ in the above
  form of $\Gamma^\alpha$.
\end{proof}

Furthermore, upon inspection of Fig.~\ref{fig_four-periodic}
we see that the location $c_\alpha$ of the center point of the
circular part of $\Gamma$ satisfies
$
\frac{c_\alpha}{{\rho}}
=
\frac{1}{{\rho}}\, \Gamma_x(0) - \sin\alpha
$, so that
the explicit parametrization of $\Gamma$
as given in Corollary~\ref{cor_smoothening_explicitGamma} yields
\begin{equation}
  \label{eqn_explictForm_c_alpha}
  \frac{c_\alpha}{{\rho}}
  =
  \alpha
  \,\frac{ \int_0^1 \cos[\alpha\,\Theta(\xi)]\, d\xi
  }{ \int_0^1 h(\xi)\,d\xi }
  -
  \sin\alpha
  \qquad\text{for}\qquad
  \alpha>0
  \;.
\end{equation}

With this more quantitative formulation of the smoothening of
the circular boundary components we arrive at a quantitative
version of Lemma~\ref{lem_existence_qualitative}.
In light of the discussion after Lemma~\ref{lem_existence_qualitative}
of the stability of such periodic orbits we are mainly interested
in the case where out of the point of reflections corresponding to
$s_0, \ldots, s_{n+2}$
only the $n+1$-st reflection is off the smoothened
part of $\Gamma$,
whereas the remaining reflections are off the interior of the circular part
of $\Gamma$.
\begin{proposition}[Existence of certain periodic orbits
  -- quantitative version]
  \label{prop_small_perturbations_lowPeriod_existence}
  Let $n\geq 0$ be some fixed integer. Then a $4\,(n+1)$-periodic orbit as
  in Lemma~\ref{lem_smoothening_characterizationCurvature}
  (and illustrated in Fig.~\ref{fig_four-periodic})
  such that only the $n+1$-st reflection in the sequence of reflections
  $s_0, \ldots, s_{n+2}$
  is off the smoothened
  boundary component $\Gamma$
  exists if and only if
  \begin{align*}
    \frac{1}{{\rho}}
    \,\Gamma_x^\alpha(\zeta)
    &=
    \frac{ \cos[ \frac{ \varphi_{n+1} }{2\,n+1} ]
    }{\sin\varphi_{n+1}}
    \,\sin[ \alpha\,\Theta(\zeta) ]
    \, \cos\Big[ \frac{ \alpha\,\Theta(\zeta) }{2\,n+1} \Big]
    \\
    &\qquad
    +
    \frac{ \sin[ \frac{ \varphi_{n+1} }{2\,n+1} ]
    }{\cos\varphi_{n+1}}
    \,\sin\Big[ \frac{ \alpha\,\Theta(\zeta) }{2\,n+1} \Big]
    \,\cos[ \alpha\,\Theta(\zeta) ]
    \\
    &\qquad
    +
    \frac{c_\alpha}{{\rho}}
    \,\sin[ \alpha\,\Theta(\zeta) ]\,\cos[ \alpha\,\Theta(\zeta) ]
    \,\Big[ \tan\varphi_{n+1} + \frac{1}{\tan\varphi_{n+1}} \Big]
    \\
    \frac{1}{{\rho}}
    \,\Gamma_y^\alpha(\zeta)
    &=
    \frac{ \cos[ \frac{ \varphi_{n+1} }{2\,n+1} ]
    }{\sin\varphi_{n+1}}
    \, \cos\Big[ \frac{ \alpha\,\Theta(\zeta) }{2\,n+1} \Big]
    \,\cos[ \alpha\,\Theta(\zeta) ]
    \\
    &\qquad
    -
    \frac{ \sin[ \frac{ \varphi_{n+1} }{2\,n+1} ]
    }{\cos\varphi_{n+1}}
    \,\sin\Big[ \frac{ \alpha\,\Theta(\zeta) }{2\,n+1} \Big]
    \,\sin[ \alpha\,\Theta(\zeta) ]
    \\
    &\qquad
    +
    \frac{c_\alpha}{{\rho}}
    \,\Big[
    \frac{ \cos^2[ \alpha\,\Theta(\zeta) ] }{\tan\varphi_{n+1}}
    -
    \sin^2[ \alpha\,\Theta(\zeta) ] \,\tan\varphi_{n+1}
    \Big]
  \end{align*}
  have a solution for $0 \leq \zeta \leq 1$ and
  $0 < \varphi_{n+1} < \frac{\pi}{2}$
  subject to the constraints
  \begin{align*}
    0 < \psi_n, \psi_{n+2} < \frac{\pi}{2} - \alpha
    \;,\qquad
    0 < \varphi_n-\psi_n, \varphi_{n+2} - \psi_{n+2}
  \end{align*}
  where
  \begin{align*}
    \varphi_n
    &=
    \frac{\pi}{2}
    - \frac{ \varphi_{n+1} - \alpha\,\Theta(\zeta) }{2\,n+1}
    \;,\qquad
    \varphi_{n+2}
    =
    \frac{\pi}{2}
    - \frac{ \varphi_{n+1} + \alpha\,\Theta(\zeta) }{2\,n+1}
    \\
    \psi_n &= n\,(\pi - 2\,\varphi_n)
    \;,\qquad
    \psi_{n+2} = n\,(\pi - 2\,\varphi_{n+2})
    \\
    \frac{\tau_{n+1,n+2}}{{\rho}}
    &=
    \frac{1}{{\rho}}
    \begin{pmatrix}
      \Gamma_x^\alpha(\zeta) + c_\alpha \\
      \Gamma_y^\alpha(\zeta)
    \end{pmatrix}
    \cdot
    \begin{pmatrix}
      \sin[ \varphi_{n+1} + \alpha\,\Theta(\zeta) ] \\
      \cos[ \varphi_{n+1} + \alpha\,\Theta(\zeta) ]
    \end{pmatrix}
    +
    \cos\varphi_{n+2}
    \\
    \frac{\tau_{n,n+1}}{{\rho}}
    &=
    \frac{1}{{\rho}}
    \begin{pmatrix}
      \Gamma_x^\alpha(\zeta) - c_\alpha \\
      \Gamma_y^\alpha(\zeta)
    \end{pmatrix}
    \cdot
    \begin{pmatrix}
      -\sin[ \varphi_{n+1} -\alpha\,\Theta(\zeta) ] \\
      \cos[ \varphi_{n+1} - \alpha\,\Theta(\zeta)]
    \end{pmatrix}
    +
    \cos\varphi_n
  \end{align*}
  and $\alpha>0$.
\end{proposition}
\begin{proof}
  For such a periodic orbit the number of reflections off each of the two
  circular parts of the boundary of the billiard table arcs is
  $2\,n+1$, recall Fig.~\ref{fig_four-periodic}.
  \begin{figure}[ht!]
    \includegraphics[width=7cm]{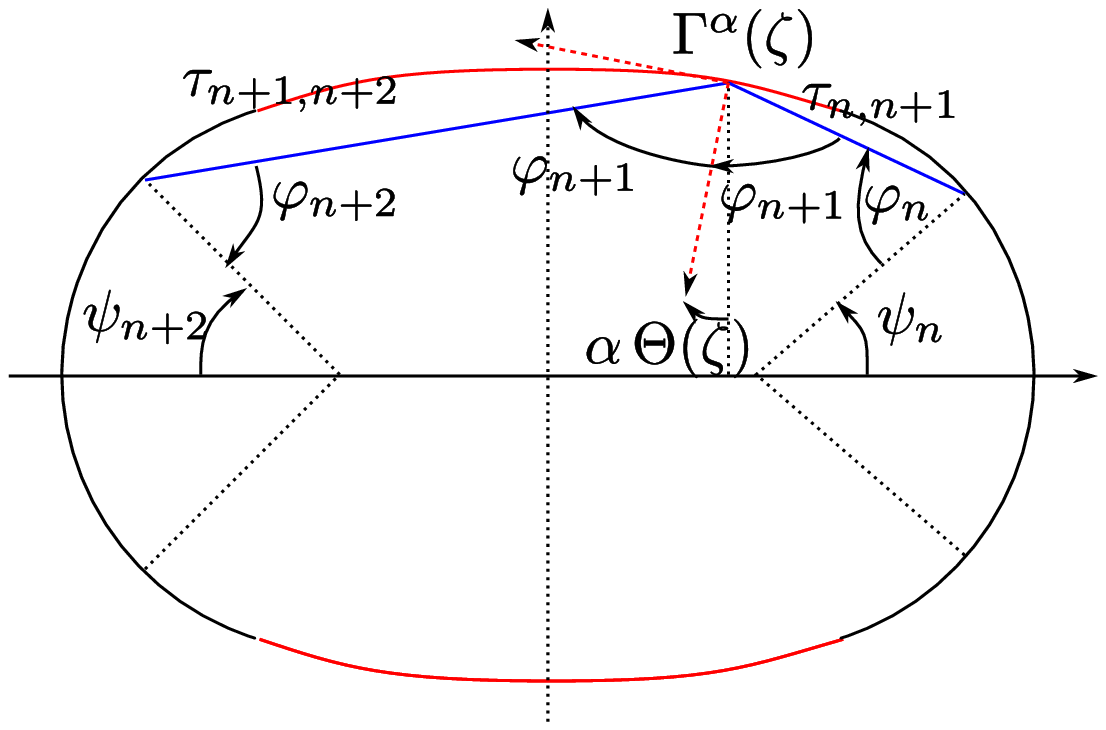}
    \caption{Construction of periodic orbit as in
    Lemma~\ref{lem_existence_qualitative} where out of the sequence
    of reflection points at
    $s_0, \ldots, s_{n+2}$ only
    $s_{n+1}$ is on the smoothened part of the boundary.
    }
    \label{fig_four-periodic_construction}
  \end{figure}
  The notation we will use below is illustrated in
  Fig.~\ref{fig_four-periodic_construction}, which is
  Fig.~\ref{fig_four-periodic} adapted to our setting
  where only $s_{n+1}$ is off the smoothened part.

  By inspecting Fig.~\ref{fig_four-periodic_construction}
  we see that at the $n+1$-st reflection a necessary and sufficient
  condition for a specular reflection is that
  \begin{align*}
    \pi-2\,\varphi_{n+1}
    &=
    \varphi_{n+2} - \psi_{n+2}
    +
    \varphi_n - \psi_n
    \\
    0 < \varphi_{n+1} < \frac{\pi}{2}
    &\;,\qquad
    0 < \psi_n, \psi_{n+2} < \frac{\pi}{2} - \alpha
    \;,\qquad
    0 < \varphi_n-\psi_n, \varphi_{n+2} - \psi_{n+2}
    \\
    \Gamma^\alpha(\zeta)
    &=
    \begin{pmatrix}
      -c_\alpha - {\rho}\,\cos\psi_{n+2} \\
      {\rho}\,\sin\psi_{n+2}
    \end{pmatrix}
    +
    \tau_{n+1,n+2}
    \begin{pmatrix}
      \cos(\varphi_{n+2} - \psi_{n+2}) \\
      \sin(\varphi_{n+2} - \psi_{n+2})
    \end{pmatrix}
    \\
    \Gamma^\alpha(\zeta)
    &=
    \begin{pmatrix}
      c_\alpha + {\rho}\,\cos\psi_n \\
      {\rho}\,\sin\psi_n
    \end{pmatrix}
    +
    \tau_{n,n+1}
    \begin{pmatrix}
      -\cos(\varphi_n - \psi_n) \\
      \sin(\varphi_n - \psi_n)
    \end{pmatrix}
  \end{align*}
  hold. The conditions
  $0 < \psi_n, \psi_{n+2} < \frac{\pi}{2} - \alpha$
  and
  $0 < \varphi_n-\psi_n, \varphi_{n+2} - \psi_{n+2}$
  are imposed to guarantee that there is only one reflection off
  the smoothened boundary component $\Gamma$.

  Once the trajectory forms a valid billiard trajectory
  a necessary and sufficient condition
  for it to form a periodic orbit as shown in
  Fig.~\ref{fig_four-periodic}
  is that the angles $\psi_n$ and $\psi_{n+2}$
  must be chosen so that
  \begin{align*}
    \psi_n = n\,(\pi - 2\,\varphi_n)
    \;,\qquad
    \psi_{n+2} = n\,(\pi - 2\,\varphi_{n+2})
  \end{align*}
  hold. This is due to the explicit solution of the billiard
  map along a circular arc.

  Using the parametrization of $\Gamma$
  as given in Corollary~\ref{cor_smoothening_explicitGamma} shows that
  \begin{align*}
    \varphi_{n+1}
    &=
    \frac{\pi}{2} - \alpha\,\Theta(\zeta)
    -
    (\varphi_{n+2} - \psi_{n+2})
    =
    \frac{\pi}{2} + \alpha\,\Theta(\zeta) - (\varphi_n - \psi_n)
  \end{align*}
  for the expression for $\varphi_{n+1}$ in terms of $\zeta$,
  where the equality of the two expressions on the right is due
  to the above conditions on the angles in order to form
  a specular reflection.

  Finally, note that by taking projections of the four equations
  \begin{align*}
    \frac{1}{{\rho}}\,\Gamma^\alpha(\zeta)
    &=
    \begin{pmatrix}
      -\frac{c_\alpha}{{\rho}} - \cos\psi_{n+2} \\
      \sin\psi_{n+2}
    \end{pmatrix}
    +
    \frac{\tau_{n+1,n+2}}{{\rho}}
    \begin{pmatrix}
      \cos(\varphi_{n+2} - \psi_{n+2}) \\
      \sin(\varphi_{n+2} - \psi_{n+2})
    \end{pmatrix}
    \\
    \frac{1}{{\rho}}\,\Gamma^\alpha(\zeta)
    &=
    \begin{pmatrix}
      \frac{c_\alpha}{{\rho}} + \cos\psi_n \\
      \sin\psi_n
    \end{pmatrix}
    +
    \frac{\tau_{n,n+1}}{{\rho}}
    \begin{pmatrix}
      -\cos(\varphi_n - \psi_n) \\
      \sin(\varphi_n - \psi_n)
    \end{pmatrix}
  \end{align*}
  we can separate the parts depending on
  $\tau_{n,n+1}$ and $\tau_{n+1,n+2}$
  from the parts depending only on $\zeta$ and the various angles.
  This proves the two equations for the free paths
  \begin{align*}
    \frac{\tau_{n+1,n+2}}{{\rho}}
    &=
    \frac{1}{{\rho}}
    \begin{pmatrix}
      \Gamma_x^\alpha(\zeta) + c_\alpha \\
      \Gamma_y^\alpha(\zeta)
    \end{pmatrix}
    \cdot
    \begin{pmatrix}
      \cos(\varphi_{n+2} - \psi_{n+2}) \\
      \sin(\varphi_{n+2} - \psi_{n+2})
    \end{pmatrix}
    +
    \cos\varphi_{n+2}
    \\
    \frac{\tau_{n,n+1}}{{\rho}}
    &=
    \frac{1}{{\rho}}
    \begin{pmatrix}
      \Gamma_x^\alpha(\zeta) - c_\alpha \\
      \Gamma_y^\alpha(\zeta)
    \end{pmatrix}
    \cdot
    \begin{pmatrix}
      -\cos(\varphi_n - \psi_n) \\
      \sin(\varphi_n - \psi_n)
    \end{pmatrix}
    +
    \cos\varphi_n
  \end{align*}
  and the two equations
  \begin{align*}
    \frac{1}{{\rho}}
    \begin{pmatrix}
      \Gamma_x^\alpha(\zeta) + c_\alpha \\
      \Gamma_y^\alpha(\zeta)
    \end{pmatrix}
    \cdot
    \begin{pmatrix}
      -\sin(\varphi_{n+2} - \psi_{n+2}) \\
      \cos(\varphi_{n+2} - \psi_{n+2})
    \end{pmatrix}
    &=
    \sin\varphi_{n+2}
    \\
    \frac{1}{{\rho}}
    \begin{pmatrix}
      \Gamma_x^\alpha(\zeta) - c_\alpha \\
      \Gamma_y^\alpha(\zeta)
    \end{pmatrix}
    \cdot
    \begin{pmatrix}
      \sin(\varphi_n - \psi_n) \\
      \cos(\varphi_n - \psi_n)
    \end{pmatrix}
    &=
    \sin\varphi_n
  \end{align*}
  for the angles and $\zeta$.
  Since we have already shown that
  \begin{align*}
    \varphi_{n+2} - \psi_{n+2}
    =
    \frac{\pi}{2} - \alpha\,\Theta(\zeta) - \varphi_{n+1}
    \qquad\text{and}\qquad
    \varphi_n - \psi_n
    =
    \frac{\pi}{2} + \alpha\,\Theta(\zeta) - \varphi_{n+1}
  \end{align*}
  we can eliminate $\psi_n$ and $\psi_{n+2}$ and obtain the
  two claimed equations for the free paths
  and the following two equations for the angles and $\zeta$
  \begin{align*}
    \frac{1}{{\rho}}
    \begin{pmatrix}
      \Gamma_x^\alpha(\zeta) + c_\alpha \\
      \Gamma_y^\alpha(\zeta)
    \end{pmatrix}
    \cdot
    \begin{pmatrix}
      -\cos[ \varphi_{n+1} + \alpha\,\Theta(\zeta) ] \\
      \sin[ \varphi_{n+1} + \alpha\,\Theta(\zeta) ]
    \end{pmatrix}
    &=
    \sin\varphi_{n+2}
    \\
    \frac{1}{{\rho}}
    \begin{pmatrix}
      \Gamma_x^\alpha(\zeta) - c_\alpha \\
      \Gamma_y^\alpha(\zeta)
    \end{pmatrix}
    \cdot
    \begin{pmatrix}
      \cos[ \varphi_{n+1} - \alpha\,\Theta(\zeta) ] \\
      \sin[ \varphi_{n+1} - \alpha\,\Theta(\zeta) ]
    \end{pmatrix}
    &=
    \sin\varphi_n
    \;.
  \end{align*}

  Using the already obtained relations
  $\psi_n = n\,(\pi - 2\,\varphi_n)$
  and
  $\psi_{n+2} = n\,(\pi - 2\,\varphi_{n+2})$
  one gets
  \begin{align*}
    (2\,n+1)\,\varphi_{n+2}
    &=
    n\,\pi + \frac{\pi}{2} - \alpha\,\Theta(\zeta) - \varphi_{n+1}
    \\
    (2\,n+1)\,\varphi_n
    &=
    n\,\pi + \frac{\pi}{2} + \alpha\,\Theta(\zeta) - \varphi_{n+1}
  \end{align*}
  and hence
  \begin{align*}
    \varphi_{n+2}
    =
    \frac{\pi}{2}
    - \frac{ \varphi_{n+1} + \alpha\,\Theta(\zeta) }{2\,n+1}
    \qquad\text{and}\qquad
    \varphi_n
    =
    \frac{\pi}{2}
    - \frac{ \varphi_{n+1} - \alpha\,\Theta(\zeta) }{2\,n+1}
  \end{align*}
  which we substitute in the two equations for $\zeta$ and the angles
  and obtain
  \begin{align*}
    \frac{1}{{\rho}}
    \begin{pmatrix}
      \Gamma_x^\alpha(\zeta) + c_\alpha \\
      \Gamma_y^\alpha(\zeta)
    \end{pmatrix}
    \cdot
    \begin{pmatrix}
      -\cos[ \varphi_{n+1} + \alpha\,\Theta(\zeta) ] \\
      \sin[ \varphi_{n+1} + \alpha\,\Theta(\zeta) ]
    \end{pmatrix}
    &=
    \cos\Big[ \frac{ \varphi_{n+1} + \alpha\,\Theta(\zeta) }{2\,n+1} \Big]
    \\
    \frac{1}{{\rho}}
    \begin{pmatrix}
      \Gamma_x^\alpha(\zeta) - c_\alpha \\
      \Gamma_y^\alpha(\zeta)
    \end{pmatrix}
    \cdot
    \begin{pmatrix}
      \cos[ \varphi_{n+1} - \alpha\,\Theta(\zeta) ] \\
      \sin[ \varphi_{n+1} - \alpha\,\Theta(\zeta) ]
    \end{pmatrix}
    &=
    \cos\Big[ \frac{ \varphi_{n+1} - \alpha\,\Theta(\zeta) }{2\,n+1} \Big]
  \end{align*}
  These two equations determine $\zeta$ and $\varphi_{n+1}$ as
  functions of $\alpha$. All other angles and the free paths are then
  determined through $\zeta$ and $\varphi_{n+1}$.

  Adding and subtracting these two equations yields
  \begin{align*}
    \frac{1}{{\rho}}
    \begin{pmatrix}
      \Gamma_x^\alpha(\zeta) \\
      \Gamma_y^\alpha(\zeta)
    \end{pmatrix}
    \cdot
    \begin{pmatrix}
      \sin\varphi_{n+1}\,\sin[ \alpha\,\Theta(\zeta) ]
      \\
      \sin\varphi_{n+1}\,\cos[ \alpha\,\Theta(\zeta) ]
    \end{pmatrix}
    &=
    \cos\Big[ \frac{ \varphi_{n+1} }{2\,n+1} \Big]
    \, \cos\Big[ \frac{ \alpha\,\Theta(\zeta) }{2\,n+1} \Big]
    \\
    &\qquad
    +
    \frac{c_\alpha}{{\rho}}
    \,\cos\varphi_{n+1}\,\cos[ \alpha\,\Theta(\zeta) ]
    \\
    \frac{1}{{\rho}}
    \begin{pmatrix}
      \Gamma_x^\alpha(\zeta) \\
      \Gamma_y^\alpha(\zeta)
    \end{pmatrix}
    \cdot
    \begin{pmatrix}
      -\cos\varphi_{n+1}\,\cos[ \alpha\,\Theta(\zeta) ]
      \\
      \cos\varphi_{n+1}\,\sin[ \alpha\,\Theta(\zeta) ]
    \end{pmatrix}
    &=
    -\sin\Big[ \frac{ \varphi_{n+1} }{2\,n+1} \Big]
    \,\sin\Big[ \frac{ \alpha\,\Theta(\zeta) }{2\,n+1} \Big]
    \\
    &\qquad
    -
    \frac{c_\alpha}{{\rho}}
    \,\sin\varphi_{n+1}\,\sin[ \alpha\,\Theta(\zeta) ]
  \end{align*}
  which we now can easily solve for
  $\Gamma^\alpha_x$
  and
  $\Gamma^\alpha_y$
  and obtain the claimed equations for $\zeta$ and $\varphi_{n+1}$.
\end{proof}

\subsection{Stability of the periodic orbits}

Proposition~\ref{prop_small_perturbations_lowPeriod_existence}
provides a quantitative description of the periodic orbits
like the one depicted in Fig.~\ref{fig_four-periodic}.
And since for these periodic orbits
all but two reflections are off circular boundary components
the analysis of their stability becomes manageable.
The first step is to compute the monodromy matrix corresponding
to such a periodic orbit.

\begin{proposition}[Monodromy matrix]
  \label{prop_small_perturbations_lowPeriod_monodromy}
  For any $n\geq 0$ the monodromy matrix $M_\alpha$
  of the $4\,(n+1)$-periodic orbit of
  Proposition~\ref{prop_small_perturbations_lowPeriod_existence}
  is given by
  \begin{align*}
    M_\alpha
    &=
    \begin{pmatrix}
      1 & \tau_{n,n+1} \\
      0 & 1
    \end{pmatrix}
    \begin{pmatrix}
      -1 & 0 \\
      \frac{2\,h(\zeta) }{{\rho}\,\cos\varphi_{n+1}} & -1
    \end{pmatrix}
    \begin{pmatrix}
      1 & \tau_{n+1,n+2} \\
      0 & 1
    \end{pmatrix}
    J_c(\varphi_{n+2})
    \cdot
    \\
    &\qquad
    \cdot
    \begin{pmatrix}
      1 & \tau_{n+1,n+2} \\
      0 & 1
    \end{pmatrix}
    \begin{pmatrix}
      -1 & 0 \\
      \frac{2\,h(\zeta) }{{\rho}\,\cos\varphi_{n+1}} & -1
    \end{pmatrix}
    \begin{pmatrix}
      1 & \tau_{n,n+1} \\
      0 & 1
    \end{pmatrix}
    J_c(\varphi_n)
  \end{align*}
  where $J_c(\varphi)$ is given by
  \begin{align*}
    J_c(\varphi) =
    \begin{pmatrix}
      -1 - 4\,n & 4\,n\,{\rho}\,\cos\varphi \\
      2\,\frac{2\,n+1}{{\rho}\,\cos\varphi} & -1 - 4\,n
    \end{pmatrix}
  \end{align*}
  and $\zeta$, $\varphi_n$, $\varphi_{n+1}$, $\varphi_{n+2}$,
  $\tau_{n,n+1}$, $\tau_{n+1,n+2}$
  are as in Proposition~\ref{prop_small_perturbations_lowPeriod_existence}.
\end{proposition}
\begin{proof}
  Along each of the two circular boundary components there are $2\,n+1$
  reflections.

  If $\varphi$ denotes the angle of reflection on one of the two
  circular boundary components, then
  \begin{align*}
    J_c(\varphi)
    &=
    \begin{pmatrix}
      -1 & 0 \\
      \frac{2}{{\rho}\,\cos\varphi} & -1
    \end{pmatrix}
    \Big[
    \begin{pmatrix}
      1 & 2\,{\rho}\,\cos\varphi \\
      0 & 1
    \end{pmatrix}
    \begin{pmatrix}
      -1 & 0 \\
      \frac{2}{{\rho}\,\cos\varphi} & -1
    \end{pmatrix}
    \Big]^{2\,n}
    \\
    &=
    \begin{pmatrix}
      -1 & 0 \\
      \frac{2}{{\rho}\,\cos\varphi} & -1
    \end{pmatrix}
    \begin{pmatrix}
      3 & -2\,{\rho}\,\cos\varphi \\
      \frac{2}{{\rho}\,\cos\varphi} & -1
    \end{pmatrix}^{2\,n}
    \\
    &=
    \begin{pmatrix}
      -1 & 0 \\
      \frac{2}{{\rho}\,\cos\varphi} & -1
    \end{pmatrix}
    \begin{pmatrix}
      1 + 4\,n & -4\,n\,{\rho}\,\cos\varphi \\
      \frac{4\,n}{{\rho}\,\cos\varphi} & 1 - 4\,n
    \end{pmatrix}
    =
    \begin{pmatrix}
      -1 - 4\,n & 4\,n\,{\rho}\,\cos\varphi \\
      2\,\frac{2\,n+1}{{\rho}\,\cos\varphi} & -1 - 4\,n
    \end{pmatrix}
  \end{align*}
  is the corresponding linearization of the billiard flow
  starting right before the first reflection and ending right
  after the last reflection.

  By inspection of
  Fig.~\ref{fig_four-periodic}
  and
  Fig.~\ref{fig_four-periodic_construction}
  we see that the monodromy matrix of the periodic orbit is given by
  \begin{align*}
    M_\alpha
    &=
    \begin{pmatrix}
      1 & \tau_{n,n+1} \\
      0 & 1
    \end{pmatrix}
    \begin{pmatrix}
      -1 & 0 \\
      -\frac{2\,{\mathcal K}(s_{n+1})
      }{\cos\varphi_{n+1}} & -1
    \end{pmatrix}
    \begin{pmatrix}
      1 & \tau_{n+1,n+2} \\
      0 & 1
    \end{pmatrix}
    J_c(\varphi_{n+2})
    \cdot
    \\
    &\qquad
    \cdot
    \begin{pmatrix}
      1 & \tau_{n+1,n+2} \\
      0 & 1
    \end{pmatrix}
    \begin{pmatrix}
      -1 & 0 \\
      -\frac{2\,{\mathcal K}(s_{n+1})
      }{\cos\varphi_{n+1}} & -1
    \end{pmatrix}
    \begin{pmatrix}
      1 & \tau_{n,n+1} \\
      0 & 1
    \end{pmatrix}
    J_c(\varphi_n)
  \end{align*}
  and since ${\mathcal K}(s_{n+1})
  = -\frac{1}{{\rho}}\,h(\zeta)$ we obtain the claimed
  expression for $M_\alpha$.
\end{proof}

Since the billiard flow is area preserving the eigenvalues
of the monodromy matrix are related to its trace.
Using the result of
Proposition~\ref{prop_small_perturbations_lowPeriod_monodromy}
for the explicit form of the monodromy matrix of the
periodic orbit
we arrive at the following stability criterion.

\begin{corollary}[Criterion for stability]
  \label{cor_small_perturbations_lowPeriod_stability_limitCircle}
  In the special case
  \begin{align*}
    \varphi_n &= \varphi_{n+1} = \varphi_{n+2}
    = \frac{\pi}{4}\,\frac{2\,n + 1}{n+1}
    \;,\qquad
    \tau_{n,n+1}
    = \tau_{n+1,n+2}
    = 2\,{\rho}\,\cos\varphi_{n+1}
  \end{align*}
  the result of
  Proposition~\ref{prop_small_perturbations_lowPeriod_monodromy}
  for the monodromy matrix $M_\alpha$ implies that
  \begin{align*}
    \Big| \frac{1}{2}\,\operatorname{tr} M_\alpha \Big| < 1
    \iff
    \frac{n + \frac{1}{2}}{ n + 1 } < h(\zeta) < 1
    \quad\text{and}\quad
    h(\zeta) \neq \frac{n + \frac{3}{4}}{n+1}
  \end{align*}
  hold for all $n\geq 0$.
\end{corollary}
\begin{proof}
  Denote $2\,{\rho}\,\cos\varphi_{n+1}$ by $\tau$.
  Then, under the stated assumptions on the angles and the free paths, the
  result of
  Proposition~\ref{prop_small_perturbations_lowPeriod_monodromy}
  for the monodromy matrix becomes
  \begin{align*}
    M_\alpha
    &=
    \Big[
    \begin{pmatrix}
      1 & \tau \\
      0 & 1
    \end{pmatrix}
    \begin{pmatrix}
      -1 & 0 \\
      \frac{2\,h(\zeta) }{{\rho}\,\cos\varphi} & -1
    \end{pmatrix}
    \begin{pmatrix}
      1 & \tau \\
      0 & 1
    \end{pmatrix}
    J_c(\varphi)
    \Big]^2
    \\
    &=
    \Big[
    \begin{pmatrix}
      1 & \tau \\
      0 & 1
    \end{pmatrix}
    \begin{pmatrix}
      -1 & 0 \\
      \frac{2\,h(\zeta) }{{\rho}\,\cos\varphi} & -1
    \end{pmatrix}
    \begin{pmatrix}
      1 & \tau \\
      0 & 1
    \end{pmatrix}
    \begin{pmatrix}
      -1 - 4\,n & 4\,n\,{\rho}\,\cos\varphi \\
      2\,\frac{2\,n+1}{{\rho}\,\cos\varphi} & -1 - 4\,n
    \end{pmatrix}
    \Big]^2
    \\
    &=
    \Big[
    \begin{pmatrix}
      1 & \tau \\
      0 & 1
    \end{pmatrix}
    \begin{pmatrix}
      -1 & 0 \\
      \frac{4\,h(\zeta) }{\tau} & -1
    \end{pmatrix}
    \begin{pmatrix}
      1 & \tau \\
      0 & 1
    \end{pmatrix}
    \begin{pmatrix}
      -1 - 4\,n & 2\,n\,\tau \\
      4\,\frac{2\,n+1}{\tau} & -1 - 4\,n
    \end{pmatrix}
    \Big]^2
    \\
    &=
    \begin{pmatrix}
      -7 - 12\,n + 4\,(3 + 4\,n)\,h
      &
      2\,\tau\,[1 + 3\,n - 2\,(1 + 2\,n)\,h]
      \\
      -\frac{4}{\tau}\,[1 + 2\,n - (3 + 4\,n)\,h]
      &
      1 + 4\,n - 4\,(1 + 2\,n)\,h
    \end{pmatrix}^2
  \end{align*}
  where $h$ is short hand notation for $h(\zeta)$.
  In particular,
  \begin{align*}
    \operatorname{tr} M_\alpha
    &=
    \Big[
    -7 - 12\,n + 4\,(3 + 4\,n)\,h
    + 1 + 4\,n - 4\,(1 + 2\,n)\,h
    \Big]^2 -2
    \\
    &=
    4\,[ 3 + 4\,n - 4\,(1 + n)\,h ]^2 -2
  \end{align*}
  follows for the trace of $M_\alpha$.
  Hence
  \begin{align*}
    \Big|\frac{1}{2}\,\operatorname{tr} M_\alpha\Big| < 1
    &\iff
    -1 < 2\,[ -3 - 4\,n + 4\,(1 + n)\,h ]^2 -1 < 1
    \\
    &\iff
    0 < | -3 - 4\,n + 4\,(1 + n)\,h | < 1
    \\
    &\iff
    -1 < -3 - 4\,n + 4\,(1 + n)\,h  < 1
    \;,\quad
    h \neq \frac{n + \frac{3}{4}}{n+1}
    \\
    &\iff
    \frac{n + \frac{1}{2}}{ n + 1 } < h < 1
    \;,\quad
    h \neq \frac{n + \frac{3}{4}}{n+1}
    \;,
  \end{align*}
  which finishes the proof.
\end{proof}

\subsection{Existence and stability of certain periodic orbits for small
smoothening regions}

The result of
Lemma~\ref{lem_smoothening_characterizationCurvature}
showed that the arc length $s_\alpha$
of the smoothened boundary component is proportional
to $\alpha$, provided that $h$ is kept fixed.
Therefore, the asymptotics as $\alpha$ tends to zero
correspond to $C^2$-stadia, which are $C^1$-close to
the usual stadium.

In the usual stadium all periodic orbits are hyperbolic.
The aim of this part of the paper is to show that for
arbitrarily small smoothening of the stadium elliptic
periodic orbits exist for sufficiently small separations $L$
of the two curved boundary components. And as outlined
at the beginning of this section we will actually show the
existence of elliptic periodic orbits for
the smoothed stadium where $L=0$, i.e. where the two
curved boundary component are joined at their endpoints without
additional straight line segments.
A criterion for the existence of certain periodic orbits was established
in Proposition~\ref{prop_small_perturbations_lowPeriod_existence}.

Since we are interested in a small smoothening, i.e. the smoothened
boundary component $\Gamma$ should differ from a semi-circle
only on small segments near the endpoints we need to consider
small values of $\alpha$.
For sufficiently small $\alpha$ the following
Proposition~\ref{prop_asymptotic_zeta} provides an asymptotic
description of the periodic orbits
considered in Proposition~\ref{prop_small_perturbations_lowPeriod_existence}.

\begin{proposition}[Existence of the orbit for small $\alpha$]
  \label{prop_asymptotic_zeta}
  Let $n\geq 0$ and a continuous function
  $h \colon [0,1] \to [0,\infty)$ with
  $h(0) = 1$ and $h(1) = 0$ be fixed.
  Then a $4\,(n+1)$-periodic orbit as
  in Proposition~\ref{prop_small_perturbations_lowPeriod_existence}
  exists for sufficiently small
  values of $\alpha$ if and only if
  \begin{equation*}
    \frac{2\,n+1}{2\,n+2}\,(1-\zeta)
    =
    \int_\zeta^1 h(\xi)\,d\xi
  \end{equation*}
  has a solution $0 < \zeta_* < 1$. In that case
  \begin{align*}
    \zeta(\alpha)
    =
    \zeta_* + \operatorname{\mathcal{O}}(\alpha)
    \qquad\text{and}\qquad
    \varphi_{n+1}(\alpha)
    =
    \frac{\pi}{4}\,\frac{2\,n+1}{n+1}
    +
    \operatorname{\mathcal{O}}(\alpha)
  \end{align*}
  holds uniformly for any solution and all sufficiently small values of
  $\alpha$.
\end{proposition}
\begin{proof}
  Combining the explicit form of $\Gamma^\alpha$
  as given in
  Corollary~\ref{cor_smoothening_explicitGamma}
  with the result of
  Proposition~\ref{prop_small_perturbations_lowPeriod_existence}
  for the existence of the periodic orbit we see that
  a periodic orbit as
  shown in Fig.~\ref{fig_four-periodic}
  exists with the $n+1$-st reflection off the interior of the smoothened
  part of $\Gamma$ if and only if
  the system of two equations
  \begin{align*}
    &\alpha
    \,\frac{ \int_\zeta^1 \cos[\alpha\,\Theta(\xi)]\, d\xi
    }{ \int_0^1 h(\xi)\,d\xi }
    =
    \frac{ \cos[ \frac{ \varphi_{n+1} }{2\,n+1} ]
    }{\sin\varphi_{n+1}}
    \,\sin[ \alpha\,\Theta(\zeta) ]
    \, \cos\Big[ \frac{ \alpha\,\Theta(\zeta) }{2\,n+1} \Big]
    \\
    &\quad
    +
    \Big[
    \alpha \,\frac{ \int_0^1 \cos[\alpha\,\Theta(\xi)]\, d\xi
    }{ \int_0^1 h(\xi)\,d\xi }
    -
    \sin\alpha
    \Big]
    \,\sin[ \alpha\,\Theta(\zeta) ]\,\cos[ \alpha\,\Theta(\zeta) ]
    \,\Big[ \tan\varphi_{n+1} + \frac{1}{\tan\varphi_{n+1}} \Big]
    \\
    &\quad
    +
    \frac{ \sin[ \frac{ \varphi_{n+1} }{2\,n+1} ]
    }{\cos\varphi_{n+1}}
    \,\sin\Big[ \frac{ \alpha\,\Theta(\zeta) }{2\,n+1} \Big]
    \,\cos[ \alpha\,\Theta(\zeta) ]
  \end{align*}
  and
  \begin{align*}
    &\cos\alpha
    +
    \alpha
    \, \frac{ \int_0^\zeta \sin[\alpha\,\Theta(\xi)] \,d\xi
    }{ \int_0^1 h(\xi)\,d\xi }
    =
    \frac{ \cos[ \frac{ \varphi_{n+1} }{2\,n+1} ]
    }{\sin\varphi_{n+1}}
    \, \cos\Big[ \frac{ \alpha\,\Theta(\zeta) }{2\,n+1} \Big]
    \,\cos[ \alpha\,\Theta(\zeta) ]
    \\
    &\quad
    -
    \frac{ \sin[ \frac{ \varphi_{n+1} }{2\,n+1} ]
    }{\cos\varphi_{n+1}}
    \,\sin\Big[ \frac{ \alpha\,\Theta(\zeta) }{2\,n+1} \Big]
    \,\sin[ \alpha\,\Theta(\zeta) ]
    \\
    &\quad
    +
    \Big[
    \alpha \,\frac{ \int_0^1 \cos[\alpha\,\Theta(\xi)]\, d\xi
    }{ \int_0^1 h(\xi)\,d\xi }
    -
    \sin\alpha
    \Big]
    \,\Big[
    \frac{ \cos^2[ \alpha\,\Theta(\zeta) ] }{\tan\varphi_{n+1}}
    -
    \sin^2[ \alpha\,\Theta(\zeta) ] \,\tan\varphi_{n+1}
    \Big]
  \end{align*}
  has a solution for $0 < \zeta < 1$ and
  $0 < \varphi_{n+1} < \frac{\pi}{2}$, where we used
  \eqref{eqn_explictForm_c_alpha} to eliminate $c_\alpha$ in
  terms of $h$.

  Since $h$ and $n$ are fixed these two equations become
  \begin{equation}
    \label{eqn_leadingOrder_zetaPhi}
    \begin{split}
      \frac{ 1-\zeta }{ \int_0^1 h(\xi)\,d\xi }
      &=
      \frac{ \cos[ \frac{ \varphi_{n+1} }{2\,n+1} ]
      }{\sin\varphi_{n+1}}
      \,\Theta(\zeta)
      +
      \frac{ \sin[ \frac{ \varphi_{n+1} }{2\,n+1} ]
      }{\cos\varphi_{n+1}}
      \,\frac{\Theta(\zeta) }{2\,n+1}
      \\
      &\qquad
      +
      \alpha\,\Theta(\zeta)\,
      \frac{ \int_0^1 [1-h(\xi)]\,d\xi }{ \int_0^1 h(\xi)\,d\xi }
      \,\Big[ \tan\varphi_{n+1} + \frac{1}{\tan\varphi_{n+1}} \Big]
      +
      \operatorname{\mathcal{O}}(\alpha^2)
      \\
      1
      &=
      \frac{ \cos[ \frac{ \varphi_{n+1} }{2\,n+1} ]
      }{\sin\varphi_{n+1}}
      +
      \alpha\,
      \frac{ \int_0^1 [1-h(\xi)]\,d\xi }{ \int_0^1 h(\xi)\,d\xi }
      \,\frac{ 1 }{\tan\varphi_{n+1}}
      +
      \operatorname{\mathcal{O}}(\alpha^2)
    \end{split}
  \end{equation}
  where the terms $\operatorname{\mathcal{O}}(\alpha^2)$ are uniform in $\zeta$
  as $\alpha$ tends to zero.

  In the limit $\alpha \to 0$ these take on the form
  \begin{align*}
    \frac{ 1-\zeta }{ \int_0^1 h(\xi)\,d\xi }
    &=
    \frac{ \cos[ \frac{ \varphi_{n+1} }{2\,n+1} ]
    }{\sin\varphi_{n+1}}
    \,\Theta(\zeta)
    +
    \frac{ \sin[ \frac{ \varphi_{n+1} }{2\,n+1} ]
    }{\cos\varphi_{n+1}}
    \,\frac{\Theta(\zeta) }{2\,n+1}
    \;,\qquad
    1
    =
    \frac{ \cos[ \frac{ \varphi_{n+1} }{2\,n+1} ]
    }{\sin\varphi_{n+1}}
  \end{align*}
  so that the second equation for $\varphi_{n+1}$ yields
  \begin{align*}
    \varphi_{n+1}|_{\alpha=0}
    =
    \frac{\pi}{4}\,\frac{2\,n+1}{n+1}
  \end{align*}
  as the only solution.
  This implies that
  \begin{align*}
    \sin\Big[ \frac{ \varphi_{n+1} }{2\,n+1} \Big]
    =
    \sin\Big[ \frac{\pi}{4}\,\frac{1}{n+1} \Big]
    =
    \cos\Big[ \frac{\pi}{2} - \frac{\pi}{4}\,\frac{1}{n+1} \Big]
    =
    \cos\Big[ \frac{\pi}{4}\,\frac{2\,n+1}{n+1} \Big]
    =
    \cos\varphi_{n+1}
  \end{align*}
  holds at $\alpha =0$. Hence the first of the above two equations
  becomes
  \begin{equation}
    \label{eqn_existence_zeta_alphaO}
    \frac{ 1-\zeta }{ \int_0^1 h(\xi)\,d\xi }
    =
    \frac{2\,n+2}{2\,n+1}\,\Theta(\zeta)
    \qquad\text{or simply}\qquad
    \frac{2\,n+1}{2\,n+2}\,(1-\zeta)
    =
    \int_\zeta^1 h(\xi)\,d\xi
  \end{equation}
  where we used the definition of $\Theta$ in the last step.

  Because the error terms are uniform in $\zeta$, and
  because the implicit function theorem applies to
  \eqref{eqn_leadingOrder_zetaPhi} we see that
  \eqref{eqn_leadingOrder_zetaPhi}
  has a solution for sufficiently small values of $\alpha$
  if and only if
  \eqref{eqn_existence_zeta_alphaO}
  has a solution for $0 < \zeta < 1$.
\end{proof}

The asymptotic description of the periodic orbits
established in Proposition~\ref{prop_asymptotic_zeta}
corresponds precisely to the special case of the
stability analysis presented in
Corollary~\ref{cor_small_perturbations_lowPeriod_stability_limitCircle}.
Therefore we obtain the following existence and stability
result for periodic orbits as in
Fig.~\ref{fig_four-periodic}
for all sufficiently small smoothening regions.

\begin{theorem}[Asymptotic existence and stability]
  \label{thm_asymptotic_stability}
  Let $n\geq 0$ and $h\colon [0,1] \to [0,\infty)$ with
  $h(0)=1$ and $h(1)=0$ be fixed.
  For all sufficiently small values of $\alpha$
  there exists a $4\,(n+1)$ periodic orbit as in
  Proposition~\ref{prop_small_perturbations_lowPeriod_existence}
  if
  \begin{align*}
    \frac{2\,n+1}{2\,n+2}\,( 1-\zeta )
    =
    \int_{\zeta}^1 h(\xi)\,d\xi
  \end{align*}
  has a solution $0 < \zeta_* < 1$.
  Furthermore, this orbit is linearly stable if
  \begin{align*}
    \frac{n + \frac{1}{2}}{ n + 1 } < h(\zeta_*) < 1
    \quad\text{and}\quad
    h(\zeta_*) \neq \frac{n + \frac{3}{4}}{n+1}
  \end{align*}
  holds.
\end{theorem}
\begin{proof}
  Proposition~\ref{prop_asymptotic_zeta} shows that a
  periodic orbit corresponding to some
  $\zeta(\alpha)$ and $\varphi_{n+1}(\alpha)$
  is equivalent to the existence of a solution $0 < \zeta_* < 1$
  to
  \begin{align*}
    \frac{2\,n+1}{2\,n+2}\,( 1-\zeta )
    =
    \int_{\zeta}^1 h(\xi)\,d\xi
  \end{align*}
  and in this case such a solution must satisfy
  \begin{align*}
    \zeta(\alpha) = \zeta_* + \operatorname{\mathcal{O}}(\alpha)
    \;,\qquad
    \varphi_{n+1}(\alpha)
    =
    \frac{\pi}{4}\,\frac{2\,n+1}{n+1} +\operatorname{\mathcal{O}}(\alpha)
  \end{align*}
  uniformly in the choice of the solution $\zeta(\alpha)$ and
  $\varphi_{n+1}(\alpha)$.

  This together with (the second part of)
  Proposition~\ref{prop_small_perturbations_lowPeriod_existence}
  implies that
  \begin{align*}
    {\mathcal K}(s_{n+1}(\alpha))
    &\equiv
    -\frac{1}{{\rho}}\,h(\zeta(\alpha))
    =
    -\frac{1}{{\rho}}\,h(\zeta_*) + \operatorname{\mathcal{O}}(\alpha)
    \\
    \varphi_n
    &=
    \frac{\pi}{2} - \frac{ \varphi_{n+1} }{2\,n+1}
    +
    \operatorname{\mathcal{O}}(\alpha)
    =
    \frac{\pi}{4}\,\frac{2\,n+1}{n+1}
    +
    \operatorname{\mathcal{O}}(\alpha)
    \\
    \varphi_{n+2}
    &=
    \frac{\pi}{2} - \frac{ \varphi_{n+1} }{2\,n+1}
    +
    \operatorname{\mathcal{O}}(\alpha)
    =
    \frac{\pi}{4}\,\frac{2\,n+1}{n+1}
    +
    \operatorname{\mathcal{O}}(\alpha)
    \\
    \frac{\tau_{n+1,n+2}}{{\rho}}
    &=
    \cos\varphi_{n+1}
    +
    \cos\varphi_{n+2}
    +
    \operatorname{\mathcal{O}}(\alpha)
    =
    2\,\cos\Big[\frac{\pi}{4}\,\frac{2\,n+1}{n+1}\Big]
    +
    \operatorname{\mathcal{O}}(\alpha)
    \\
    \frac{\tau_{n,n+1}}{{\rho}}
    &=
    \cos\varphi_{n+1}
    +
    \cos\varphi_n
    +
    \operatorname{\mathcal{O}}(\alpha)
    =
    2\,\cos\Big[\frac{\pi}{4}\,\frac{2\,n+1}{n+1}\Big]
    +
    \operatorname{\mathcal{O}}(\alpha)
  \end{align*}
  hold uniformly.

  In particular, the periodic orbit is linearly stable if and only if
  $|\frac{1}{2}\, \operatorname{tr} M_\alpha| < 1$.
  Due to the uniformity of the error terms we can apply
  Corollary~\ref{cor_small_perturbations_lowPeriod_stability_limitCircle}
  and obtain
  \begin{align*}
    \Big| \frac{1}{2}\,\operatorname{tr} M_\alpha \Big| < 1
    \iff
    \frac{n + \frac{1}{2}}{ n + 1 } < h(\zeta_*) < 1
    \quad\text{and}\quad
    h(\zeta_*) \neq \frac{n + \frac{3}{4}}{n+1}
  \end{align*}
  for the trace of the monodromy matrix $M_\alpha$ for
  all sufficiently small values of $\alpha$.
\end{proof}

\subsection{Sufficient conditions for the existence of elliptic orbits.}

While the result of Theorem~\ref{thm_asymptotic_stability}
characterizes the existence of certain periodic orbits
and establishes their
stability, it does not say anything about how to choose the
smoothening, i.e. the function $h$, in order to obtain elliptic
periodic orbits.
Therefore, to prove our main result
Theorem~\ref{thm_shortEllipticOrbits}
we still need to describe the smoothenings,
that result in billiard tables which have elliptic orbits.
From the result of Theorem~\ref{thm_asymptotic_stability}
we already know that we need to study
equations of the form
$\frac{2\,n+1}{2\,n+2}\,( 1-\zeta )
=
\int_{\zeta}^1 h(\xi)\,d\xi$.
The following
Lemma~\ref{lem_auxMonotonicity},
Corollary~\ref{cor_auxSolutions}, and
Corollary~\ref{cor_auxSolutions_strict}
describe solutions to this equation. To simplify notation, for
a given continuous function $h \colon [0,1] \to {\mathbb R}$ we denote by
$H(\zeta)$ the function
\begin{equation*}
  H(\zeta) = \int_0^1 h(t\,\zeta + 1-t)\,dt
  \;.
\end{equation*}
\begin{lemma}[Auxiliary monotonicity result]
  \label{lem_auxMonotonicity}
  If $h$ is non-increasing with $h(0)=1$ and $h(1) = 0$, then
  \begin{enumerate}
    \item
      The function $H$ is continuous and non-increasing
      with $H(0)>0$ and $H(1) = 0$.
    \item
      If $\zeta \in [0,1]$ is such that $H(\zeta) = H(0)$, then
      $h(\xi) = 1$ for all $ \xi \in [ 0, \zeta]$.
    \item
      If $\zeta \in [0,1]$ is such that $H(\zeta) = H(1)$, then
      $h(\xi) = 0$ for all $ \xi \in [\zeta,1]$.
    \item
      For all $\zeta\in[0,1]$ the inequality $H(\zeta)\leq h(\zeta)$
      holds, where equality holds if and only if
      $H(\zeta) = 0$.
  \end{enumerate}
\end{lemma}
\begin{proof}
  Let $\zeta_1 \leq \zeta_2$ be arbitrary. Then for any $t \in [0,1]$
  we have
  $t\,\zeta_1 + 1-t \leq t\,\zeta_2 + 1-t$, so that the monotonicity of
  $h$ implies
  \begin{align*}
    H(\zeta_1)
    &=
    \int_0^1 h(t\,\zeta_1 + 1-t)\,dt
    \geq
    \int_0^1 h(t\,\zeta_2 + 1-t)\,dt
    =
    H(\zeta_2)
    \;,
  \end{align*}
  i.e. $H$ is non-increasing. And $H(0)>0$ and
  $H(1) = \int_0^1 h(1)\,dt = h(1)= 0$
  follows from the assumption on $h$.

  Furthermore, suppose that $H(\zeta) = H(0)$ for some $\zeta \in [0,1]$.
  Then
  \begin{align*}
    0 &= H(0) - H(\zeta) = \int_0^1 [ h(1-t) - h(t\,\zeta + 1-t)]\,dt
  \end{align*}
  and the monotonicity of $h$ imply that
  $ h(1-t) = h(t\,\zeta + 1-t)$ must hold for all $0 \leq t \leq 1$.
  In particular, for $t=1$ this becomes
  $ h(0) = h(\zeta)$, and therefore
  $ h(\xi) = h(0) = 1$ for all $\xi \in [0,\zeta]$ follows
  from the monotonicity of $h$.

  Similarly, suppose that $H(\zeta) = H(1)$ for some $\zeta \in [0,1]$.
  Then
  \begin{align*}
    0
    &=
    H(\zeta) - H(1)
    =
    \int_0^1 [ h(t\,\zeta + 1-t) - h(1) ]\,dt
  \end{align*}
  and the monotonicity of $h$ imply that
  $ h(t\,\zeta + 1-t) = h(1) = 0$ must hold for all $0 \leq t \leq 1$,
  i.e.  $ h(\xi) = 0$ for all $\xi \in [\zeta, 1]$.

  Finally, the inequality
  \begin{align*}
    H(\zeta)
    &=
    \int_0^1 h(t\,\zeta + 1-t)\,dt
    =
    h(\zeta)
    -
    \int_0^1 [h(\zeta) - h(t\,\zeta + 1-t)]\,dt
    \leq
    h(\zeta)
  \end{align*}
  holds for all $\zeta$, because $h$ is non-increasing.
  Hence the equality $H(\zeta) = h(\zeta)$ holds if and only if
  $h(\xi) = h(1) = 0$ for all $\zeta \leq \xi \leq 1$, i.e.
  if and only if $H(\zeta)=0$.
\end{proof}

\begin{corollary}
  \label{cor_auxSolutions}
  Let $h$ as in Lemma~\ref{lem_auxMonotonicity},
  and let $\beta$ be a real number.
  \begin{enumerate}
    \item
      The equation $\beta\,(1-\zeta) = \int_\zeta^1 h(\xi)\,d\xi$
      always has the solution $\zeta =1$, regardless of the value of $\beta$.
    \item
      The equation $\beta\,(1-\zeta) = \int_\zeta^1 h(\xi)\,d\xi$
      has a solution $\zeta \in [0,1)$ if and only if $\zeta$ is
      a solution to $\beta = H(\zeta)$.
    \item
      The equation
      $\beta = H(\zeta)$ has a solution $\zeta \in [0,1]$
      if and only if $0 \leq \beta \leq \int_0^1 h(\xi)\,d\xi$.
  \end{enumerate}
\end{corollary}
\begin{proof}
  Clearly, the special value $\zeta=1$ is always a solution to
  $\beta\,(1-\zeta) = \int_\zeta^1 h(\xi)\,d\xi$.
  And for $\zeta \neq 1$, the
  change of variables $t \mapsto \xi = t\,\zeta + 1-t $ shows that the equation
  $\beta\,(1-\zeta) = \int_\zeta^1 h(\xi)\,d\xi$
  is equivalent to
  $\beta = \int_0^1 h(t\,\zeta + 1-t)\,dt$, which
  is nothing else but $\beta = H(\zeta)$.

  By Lemma~\ref{lem_auxMonotonicity} we know that $H$ is continuous
  and non-increasing. Therefore $\beta = H(\zeta)$ has a
  solution $\zeta \in [0,1]$ if and only if
  $H(1) = 0 \leq \beta \leq H(0) = \int_0^1 h(\xi)\,d\xi$ holds.
\end{proof}

\begin{corollary}
  \label{cor_auxSolutions_strict}
  Let $h$ be as in Lemma~\ref{lem_auxMonotonicity},
  and suppose that $h$ satisfies the additional assumption
  $h(\zeta) < h(0)$ for all $\zeta>0$.
  If $0 < \beta < \int_0^1 h(\xi)\,d\xi$,
  then the solution $\zeta_\beta$ to the equation $\beta = H(\zeta)$
  satisfies $0 < \zeta_\beta < 1$
  and
  $\beta < h(\zeta_\beta) < 1$.
\end{corollary}
\begin{proof}
  The existence of a solution $\zeta_\beta$ to $\beta = H(\zeta)$
  was already established in Corollary~\ref{cor_auxSolutions}.

  Since
  $H(\zeta_\beta) = \beta \neq 0 = H(1)$,
  the result of Lemma~\ref{lem_auxMonotonicity}
  shows that $\beta = H(\zeta_\beta) < h(\zeta_\beta)$ holds.
  In particular it follows that the monotonicity of $h$ implies
  $\zeta_\beta<1$, because $h(1) = 0$ and $0 < \beta < h(\zeta_\beta)$.

  On the other hand, since $H(\zeta_\beta) = \beta < H(0)$
  the monotonicity of $H$, recall Lemma~\ref{lem_auxMonotonicity},
  implies that $\zeta_\beta>0$. The assumed property on $h$ that
  $h$ is strictly decreasing at $\zeta=0$ then implies that
  $h(\zeta_\beta) < h(0) = 1$, which finishes the proof.
\end{proof}

In order to proceed we need to introduce a certain non-degeneracy
condition for $h$. In the following
Theorem~\ref{thm_shortSeparation_sufficientCond_ellipticOrbit}
we will consider continuous non-increasing functions
$h\colon [0,1] \to {\mathbb R}$
that satisfy the non-degeneracy condition
\begin{equation}
  \tag{C}
  \begin{split}
    &\text{there exists an integer } n_* \geq 0
    \text{ such that}\quad
    \frac{2\,n_* + 1}{2\,n_* + 2} < \int_0^1 h(\xi)\,d\xi
    \\
    &\text{and for all } \zeta \text{ with}\quad
    \frac{2\,n_* + 1}{2\,n_* + 2}
    =
    \int_\zeta^1 h(\xi)\,dx
    \quad\text{we have}\quad
    h(\zeta) \neq
    \frac{2\,n_* + \frac{3}{2}}{2\,n_* + 2}
    \;.
  \end{split}
  \label{eqn_tag_condition_h}
\end{equation}
Using the results of
Lemma~\ref{lem_auxMonotonicity},
Corollary~\ref{cor_auxSolutions}, and
Corollary~\ref{cor_auxSolutions_strict}
we obtain the following sufficient condition
for the existence of elliptic periodic orbits.

\begin{theorem}[Sufficient condition for the existence of
  elliptic periodic orbits]
  \label{thm_shortSeparation_sufficientCond_ellipticOrbit}
  Suppose that $h\colon [0,1]\to{\mathbb R}$ is continuous non-increasing with
  \begin{equation*}
    h(0) = 1
    \;,\quad
    h(1) = 0
    \;,\quad
    \frac{1}{2} < \int_0^1 h(\xi)\,d\xi
    \;,\quad
    h(\zeta) < h(0)
    \quad\text{for all } \zeta >0
    \;.
  \end{equation*}
  Suppose further that $h$
  satisfies the non-degeneracy condition \eqref{eqn_tag_condition_h}.
  Then there exists a linearly stable $4\,(n_*+1)$-periodic orbit
  for all sufficiently small values of $\alpha$.
\end{theorem}
\begin{proof}
  By Corollary~\ref{cor_auxSolutions_strict} we know that under the
  stated assumptions on $h$ there exists a solution $\zeta_*$ to
  $H(\zeta_*) = \frac{2\,n_*+1}{2\,n_*+2}$,
  which satisfies
  \begin{equation*}
    0 < \zeta_* < 1
    \qquad\text{and}\qquad
    \frac{n_* + \frac{1}{2} }{n_*+1} < h(\zeta_*) < 1
    \;.
  \end{equation*}
  Therefore, Theorem~\ref{thm_asymptotic_stability} proves that
  there exists a linearly stable $4\,(n+1)$-periodic orbit
  (as shown in
  Fig.~\ref{fig_four-periodic} with
  one reflection off the smooth boundary component $\Gamma$)
  for all sufficiently small values of $\alpha$.
\end{proof}

\begin{proof}[Proof of Theorem~\ref{thm_shortEllipticOrbits}]
  At this point we would like to make a few comments on
  Theorem~\ref{thm_shortSeparation_sufficientCond_ellipticOrbit}
  and at the same time prove the announced
  Theorem~\ref{thm_shortEllipticOrbits}.

  First of all, Theorem~\ref{thm_asymptotic_stability} provides
  a characterization of the existence and stability of
  certain periodic orbits for small smoothening regions
  for fixed $h$.

  Under the assumption that $h$ is non-increasing, that is to say
  that the curvature of the smoothened part of $\Gamma$ tends
  to zero monotonically, the result of
  Corollary~\ref{cor_auxSolutions_strict} shows the following. If there exists
  a solution to the equation for $\zeta$ in
  Theorem~\ref{thm_asymptotic_stability}, then it automatically satisfies
  the stability criterion of Theorem~\ref{thm_asymptotic_stability}. In
  other words, as long as the curvature of the smoothened part of
  $\Gamma$
  is monotone, the existence a periodic orbit as shown in
  Fig.~\ref{fig_four-periodic}
  automatically implies its ellipticity.
  The condition $\frac{1}{2} < \int_0^1 h(\xi)\,d\xi$ on $h$
  in Theorem~\ref{thm_shortSeparation_sufficientCond_ellipticOrbit}
  then guarantees that there exists such a periodic orbit.

  Secondly, recall that by
  Lemma~\ref{lem_smoothening_characterizationCurvature}
  the arc length of the smoothened part of the
  boundary, denoted by $s_\alpha$, satisfies
  \begin{equation*}
    \frac{\alpha\,{\rho}}{s_\alpha}
    =
    \int_0^1 h(\xi)\,d\xi
  \end{equation*}
  for all $\alpha>0$. Note further that
  $s_\alpha^{\text{circle}} = \alpha\,{\rho}$
  is precisely the arc length of the piece of the circular boundary
  component that got replaced by the smoothened part.

  Therefore the condition
  $\frac{1}{2} < \int_0^1 h(\xi)\,d\xi$
  in
  Theorem~\ref{thm_shortSeparation_sufficientCond_ellipticOrbit}
  can be recast as
  \begin{equation*}
    \int_0^1 h(\xi)\,d\xi
    =
    \frac{ s_\alpha^{\text{circle}} }{s_\alpha}
    >
    \frac{1}{2}
    \qquad\text{i.e.}\qquad
    s_\alpha
    <
    2\,s_\alpha^{\text{circle}}
    \;.
  \end{equation*}
  Therefore, the meaning of the result of
  Theorem~\ref{thm_shortSeparation_sufficientCond_ellipticOrbit}
  is that if the curvature monotonically decreases to zero
  so that the smoothened part of $\Gamma$ is less than twice the
  length of the original circular segment it replaces, then
  there exists an elliptic periodic orbit
  as is shown in Fig.~\ref{fig_four-periodic}.

  Furthermore, if the smoothened segment is larger than twice the
  corresponding circular part (the curvature is still non-increasing),
  then the corresponding periodic orbits still exists (see
  Lemma~\ref{lem_existence_qualitative}), but have no reflection off
  the smoothened part. Hence they are unstable, as they are also
  present in the usual stadium.

  Finally, as was already mentioned at the very beginning of this section,
  since elliptic periodic orbits persist as elliptic
  periodic orbits for small enough separations of the two
  curved boundary components we obtain
  Theorem~\ref{thm_shortEllipticOrbits}.
\end{proof}

\section{Elliptic Periodic Orbits for Large Separations}
\label{sect_large_separations}

In order to find stable periodic orbits for large separations of
the two curved boundary components we exploit the fact that
the curved components are not absolutely focusing. In particular,
the first step will be to find parabolic periodic orbits. This
will be accomplished by constructing parts of billiard trajectories
that correspond to a complete sequence of reflections off the curved
boundary segment such that an infinitesimal parallel beam
falling onto a focusing component will component being again
a parallel beam.

It will be shown that every such part of a billiard trajectory gives rise
to a parabolic periodic orbit on a sequence of stadium-like billiard
tables. Then a perturbation argument is applied to establish existence
of elliptic periodic orbits for large set of separation distances of
the two curved boundary segments.

Finally, the nonlinear stability of
these orbits will be established by computing the corresponding
Birkhoff. To guarantee the $C^4$-smoothness of the billiard map, which
will be used in the nonlinear stability analysis, we make throughout
this section the standing assumption that the curved boundary component
$\Gamma$ is of class $C^5$. However, globally the boundary
of the billiard table is, generally, only of class $C^2$, because we
do not assume that also the derivative of the curvature vanishes at
the endpoints of $\Gamma$.

\subsection{Construction of symmetric periodic orbits and
corresponding return maps}

The first step in the construction of stable periodic orbits
is to use the symmetry of the smoothed out circular boundary component
$\Gamma$
about the horizontal axis to construct a special part of a billiard
trajectory. This serves as a building block in the construction
of parabolic periodic orbits.

The actual construction of these special trajectory pieces is
not particularly complicated. However, we prefer to postpone the
description of this construction to
Section~\ref{sect_largeSeparation_verifyAssumption}
for two reasons. Firstly, the particular details are not relevant for
our results on the existence of nonlinearly stable periodic orbits,
and so we prefer to first provide those general results. In particular,
only after these general results on nonlinear stability we know what
we need to verify in specific examples. And this is the second reason
for postponing the details of the construction of these special
trajectory pieces, so that we can verify at once all conditions we
need for the existence and nonlinear stability results.

Therefore, instead of showing the existence of certain trajectory
segments we make the following assumption on the smoothed-out curved
boundary component $\Gamma$:
\begin{assumption}
  \label{assumption_def_parallelInParallelOutSegment}
  There exists a complete sequence of reflections off of
  $\Gamma$ that is symmetric about the horizontal axis.
  Furthermore, denote the
  initial pre-collisional coordinates by
  $\hat{s}_0, \hat{\omega}_0$
  and final post-collisional coordinates by
  $\hat{s}_1, \hat{\omega}_1$, then
  $ \frac{3}{2} \pi < \hat{\omega}_0 < 2\pi $,
  $
  \hat{\omega}_1
  =
  3\pi - \hat{\omega}_0
  $, and the linearization of the corresponding billiard flow
  (in terms of the usual Jacobi coordinates) is given by
  $
  \begin{pmatrix}
    -1 & \beta \\
    0 & -1
  \end{pmatrix}
  $ with $\beta \neq 0$.
\end{assumption}
In Section~\ref{sect_largeSeparation_verifyAssumption} we show that as long as
the smoothed-out portion of $\Gamma$ is short enough, i.e.
as long as $\Gamma$ is a circle segment except for short
segments near its endpoints, $\Gamma$ satisfies
Assumption~\ref{assumption_def_parallelInParallelOutSegment}.

In the following we adopt the following convention:
Given a complete sequence of reflections off of $\Gamma$,
let $s_0$ denote the arc length parameter
of the first reflection and let $\omega_0$ denote the
angle corresponding to the pre-collisional flow direction
$( \cos\omega_0, \sin\omega_0)$. Similarly, let
$s_1$ denote the arc length parameter corresponding to
last point of reflection and let $\omega_1$ denote the
angle corresponding to the post-collisional flow direction.

\begin{figure}[ht!]
  \centering
  \includegraphics[width=4.5cm]{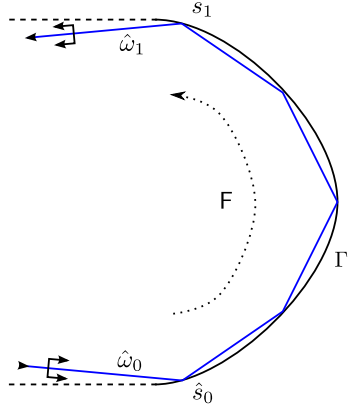}
  \caption{A complete symmetric sequence of reflections off
  the curved boundary component $\Gamma$
  as in Assumption~\ref{assumption_def_parallelInParallelOutSegment}.
  Shown is the case of an odd number of reflections. The dashed lines indicate
  the straight line segments of the boundary of the billiard table.}
  \label{fig_long_orbit_curved_segment}
\end{figure}

The next step in the construction of stable periodic orbits
is to use specialized coordinates
$(\mathsf{x}_0, \mathsf{y}_0)$
and
$(\mathsf{x}_1, \mathsf{y}_1)$
in a neighborhood of
$(\hat{s}_0, \hat{\omega}_0)$
and
$(\hat{s}_1, \hat{\omega}_1)$,
respectively,
\begin{equation}
  \label{def_localFlowCoordinates_01}
  \begin{split}
    \mathsf{x}_0
    &= 
    [
    \Gamma(s_0)
    -
    \Gamma(\hat{s}_0)
    ] \cdot
    \begin{pmatrix}
      \cos\omega_0 \\
      \sin\omega_0
    \end{pmatrix}^\perp
    \;,\quad
    \mathsf{y}_0
    =
    \omega_0
    -
    \hat{\omega}_0
    \;,
    \\
    \mathsf{x}_1
    &= 
    [
    \Gamma(s_1)
    -
    \Gamma(\hat{s}_1)
    ] \cdot
    \begin{pmatrix}
      \cos\omega_1 \\
      \sin\omega_1
    \end{pmatrix}^\perp
    \;,\quad
    \mathsf{y}_1
    =
    \omega_1
    -
    \hat{\omega}_1
    \;,
  \end{split}
\end{equation}
where we make use of the notation $\perp$
\begin{equation*}
  \begin{pmatrix}
    a \\
    b
  \end{pmatrix}^\perp
  =
  \begin{pmatrix}
    -b \\
    a
  \end{pmatrix}
  \qquad\text{for any two } a,b \in {\mathbb R}
  \;,
\end{equation*}
which will be used throughout the rest of the paper.
These coordinates are adapted to the billiard flow, and
more details are given Section~\ref{sect_billiard_basicFacts}.
In terms of fhese coordinates the billiard map
$
(\mathsf{x}_0, \mathsf{y}_0)
\mapsto
(\mathsf{x}_1, \mathsf{y}_1)
$
near the reference orbit
$
(\hat{s}_0, \hat{\omega}_0)
\mapsto
(\hat{s}_1, \hat{\omega}_1)
$
can be expressed as
\begin{equation}
  \label{eqn_def_reducedMonodromyMap}
  (\mathsf{x}_1, \mathsf{y}_1)
  =
  {\mathsf F}(\mathsf{x}_0, \mathsf{y}_0)
  \;,
\end{equation}
which is a well-defined $C^4$ map in a neighborhood of $(0,0)$
and satisfies
\begin{equation}
  \label{eqn_reducedMonodromyMap_basicProps}
  {\mathsf F}(0,0)
  =
  (0,0)
  \;,\quad
  \mathrm{D}{\mathsf F}(0,0)
  =
  \begin{pmatrix}
    -1 & a_{01} \\
    0 & -1
  \end{pmatrix}
  \quad\text{with}\quad
  a_{01} \neq 0
  \;.
\end{equation}
In fact, by construction ${\mathsf F}$
is area and orientation preserving, i.e.
\begin{equation}
  \label{eqn_reducedMonodromyMap_detD}
  \det \mathrm{D} {\mathsf F}(\mathsf{x}_0, \mathsf{y}_0)
  =
  1
\end{equation}
for all $(\mathsf{x}_0, \mathsf{y}_0)$.
Furthermore, the symmetry of the boundary component
$\Gamma$ and the symmetry of the reference orbit
$
(\hat{s}_0, \hat{\omega}_0)
\mapsto
(\hat{s}_1, \hat{\omega}_1)
$
ensure that ${\mathsf F}$ has the
following symmetry property
\begin{equation}
  \label{eqn_reducedMonodromyMap_symmetry}
  J \circ {\mathsf F} \circ J \circ
  {\mathsf F}(\mathsf{x}_0, \mathsf{y}_0)
  =
  (\mathsf{x}_0, \mathsf{y}_0)
  \quad\text{where}\quad
  J(\mathsf{x}_0, \mathsf{y}_0)
  =
  (\mathsf{x}_0, -\mathsf{y}_0)
\end{equation}
for all $(\mathsf{x}_0, \mathsf{y}_0)$.

\begin{lemma}
  \label{lem_funcX_construction}
  In some neighborhood of $0$
  there exists a unique function
  $
  \mathsf{x} = f(\mathsf{y})
  $, which satisfies
  \begin{equation*}
    {\mathsf F}( f(\mathsf{y}), \mathsf{y})
    =
    ( f(\mathsf{y}), - \mathsf{y} )
    \;,\quad
    f(0) = 0
    \;,\quad
    f'(0) = \frac{1}{2}\,a_{01} \neq 0
  \end{equation*}
  for all $\mathsf{y}$ in the domain of $f$.
\end{lemma}
\begin{proof}
  By the implicit function theorem there exists a
  neighborhood $U$ of $0$ and a uniquely determined invertible
  function $f$ defined on $U$ such that
  \begin{equation}
    \label{eqn_def_funcX}
    f(\mathsf{y})
    =
    {\mathsf F}_{\mathsf{x}}( f(\mathsf{y}), \mathsf{y} )
    \quad\text{with}\quad
    f(0) = 0
    \;,\quad
    f'(0) = \frac{1}{2}\,a_{01} \neq 0
  \end{equation}
  for all $\mathsf{y}$ in $U$.
  The symmetry condition \eqref{eqn_reducedMonodromyMap_symmetry}
  of ${\mathsf F}$ then implies
  \begin{equation*}
    ( f(\mathsf{y}), \mathsf{y})
    =
    J \circ {\mathsf F} \circ J \circ
    {\mathsf F}( f(\mathsf{y}), \mathsf{y})
    \;.
  \end{equation*}
  By definition of $J$ we have
  \begin{equation*}
    J \circ {\mathsf F} \circ J \circ
    {\mathsf F}( f(\mathsf{y}), \mathsf{y})
    =
    J \circ {\mathsf F}(
    f(\mathsf{y}) ,
    -{\mathsf F}_{\mathsf{y}}( f(\mathsf{y}), \mathsf{y})
    )
    \;.
  \end{equation*}
  Hence the symmetry property of ${\mathsf F}$ takes on the form
  \begin{equation*}
    f(\mathsf{y})
    =
    {\mathsf F}_{\mathsf{x}}(
    f(\mathsf{y}) ,
    -{\mathsf F}_{\mathsf{y}}( f(\mathsf{y}), \mathsf{y})
    )
    \;,\quad
    \mathsf{y}
    =
    -
    {\mathsf F}_{\mathsf{y}}(
    f(\mathsf{y}) ,
    -{\mathsf F}_{\mathsf{y}}( f(\mathsf{y}), \mathsf{y})
    )
    \;.
  \end{equation*}
  The local uniqueness of
  $\mathsf{y} \mapsto f(\mathsf{y})$, i.e.
  \begin{equation*}
    \mathsf{x}
    =
    {\mathsf F}_{\mathsf{x}}( \mathsf{x} , \mathsf{y})
    \iff
    \mathsf{x}
    =
    f(\mathsf{y})
  \end{equation*}
  implies
  \begin{equation*}
    f(\mathsf{y})
    =
    f(
    -{\mathsf F}_{\mathsf{y}}( f(\mathsf{y}), \mathsf{y})
    )
    \;.
  \end{equation*}
  Since $f$ is invertible we get
  \begin{equation*}
    \mathsf{y}
    =
    -{\mathsf F}_{\mathsf{y}}( f(\mathsf{y}), \mathsf{y})
    \;,
  \end{equation*}
  which means
  \begin{equation*}
    {\mathsf F}( f(\mathsf{y}), \mathsf{y})
    =
    ( f(\mathsf{y}), - \mathsf{y} )
    \;.
  \end{equation*}
\end{proof}

Now we are in the position to complete the construction of families
of certain symmetric periodic orbits. Given
$\bar{\mathsf{y}}_0$ sufficiently close to $0$
we choose $\bar{\mathsf{x}}_0 = f(\bar{\mathsf{y}}_0)$.
By Lemma~\ref{lem_funcX_construction} the image
$
(\bar{\mathsf{x}}_1, \bar{\mathsf{y}}_1)
=
{\mathsf F}(\bar{\mathsf{x}}_0, \bar{\mathsf{y}}_0)
$
satisfies
\begin{equation}
  \label{eqn_relation_barFlow_01}
  (\bar{\mathsf{x}}_1, \bar{\mathsf{y}}_1)
  =
  (\bar{\mathsf{x}}_0, -\bar{\mathsf{y}}_0)
  \;.
\end{equation}
Let
$(\bar{s}_0, \bar{\omega}_0)$
and
$(\bar{s}_1, \bar{\omega}_1)$
denote the arc length and direction angle corresponding to
$ (\bar{\mathsf{x}}_0, \bar{\mathsf{y}}_0) $
and
$ (\bar{\mathsf{x}}_1, \bar{\mathsf{y}}_1) $,
respectively.
In particular, for
$(\bar{\mathsf{x}}_0, \bar{\mathsf{y}}_0) = (0,0)$
we have
$
(\bar{s}_0, \bar{\omega}_0)
=
(\hat{s}_0, \hat{\omega}_0)
$
and
$
(\bar{s}_1, \bar{\omega}_1)
=
(\hat{s}_1, \hat{\omega}_1)
$.

Upon unfolding the billiard table in the vertical direction
(i.e. about the parallel walls) it is easily seen that the additional
reflections inside the parallel channel correspond
to a free flight on the unfolded table.
This construction is illustrated by Fig.~\ref{fig_general_period}.
\begin{figure}[ht!]
  \centering
  \begin{center}
    \subfloat[The ``folded'' table]{\label{fig_general_period-a}
    \includegraphics[width=7cm]{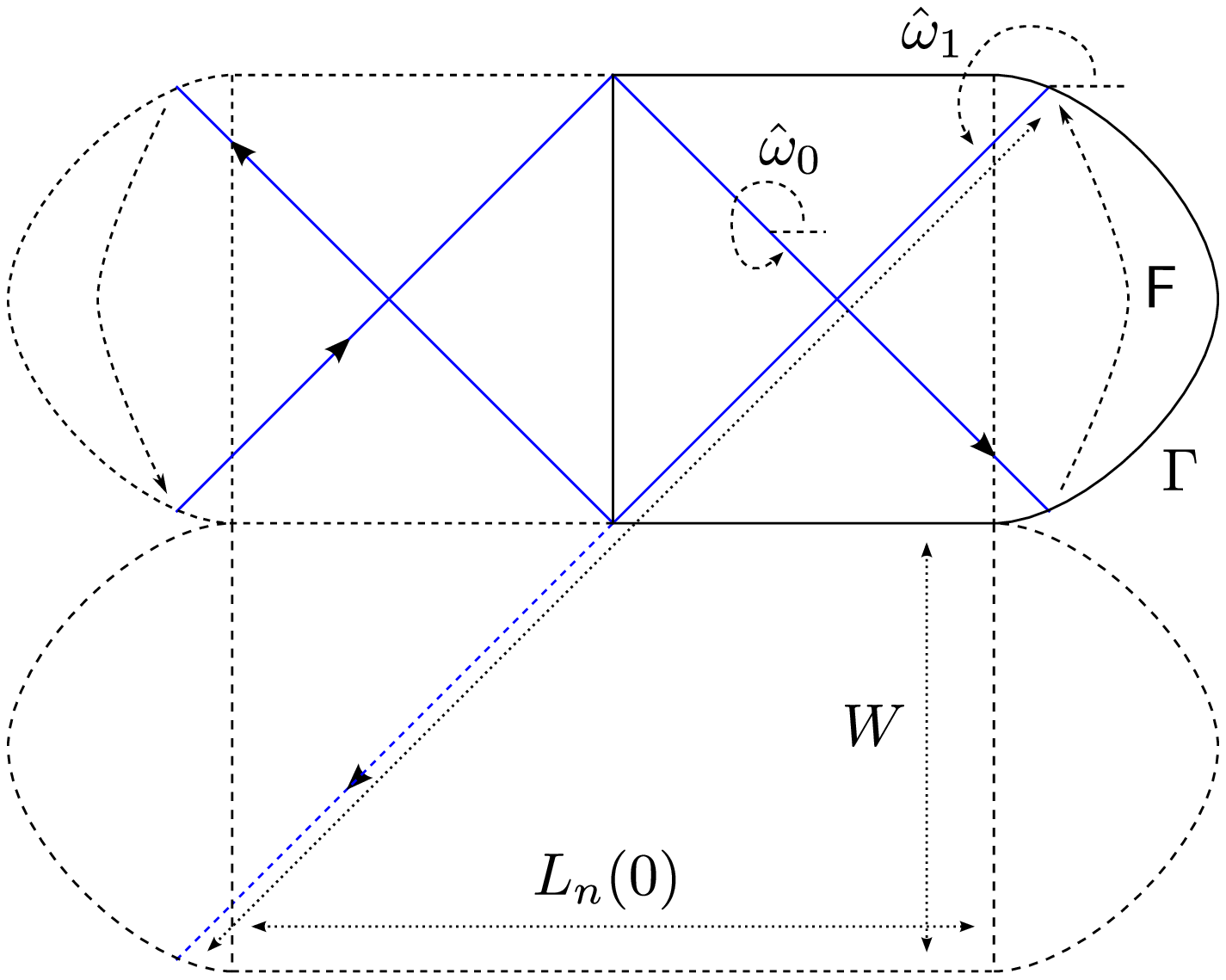}
    }
    \subfloat[Construction of periodic orbits]{\label{fig_general_period-b}
    \includegraphics[width=8cm]{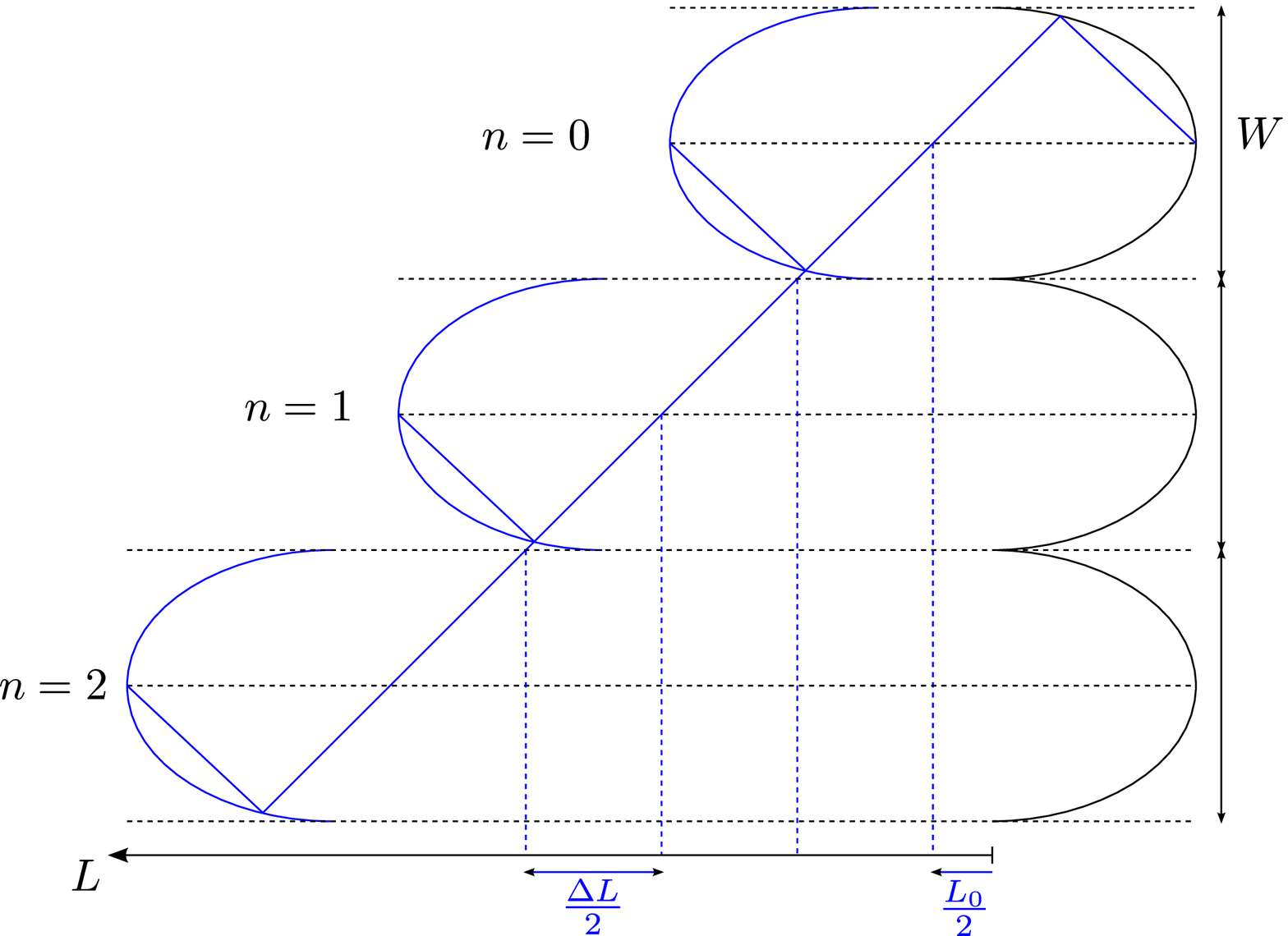}
    }
  \end{center}
  \caption{The solid black lines in Fig.~\ref{fig_general_period-a}
  indicate the boundary of the ``folded'' billiard table.
  The construction of a periodic orbit
  by closing up a complete symmetric sequence of reflections
  off the curved boundary component $\Gamma$ is then
  obtained by unfolding the billiard in the horizontal and vertical direction,
  which is indicated by the dashed lines.}
  \label{fig_general_period}
\end{figure}
In order for
$ (\bar{\mathsf{x}}_0, \bar{\mathsf{y}}_0) $
and
$ (\bar{\mathsf{x}}_1, \bar{\mathsf{y}}_1) $
to belong to a periodic orbit for
any integer $n=1, 2, \ldots$ the separation distance $\bar{L}_n$ of the 
two curved boundary segments must be a specially chosen value.
Before getting to the expression for $\bar{L}_n$ note that due to
the symmetry of the billiard table about the vertical center axis
the free flight
$
(s_1, \omega_1)
\mapsto
(s_2, \omega_2)
$
across the channel of length $L$ connecting the two curved boundary
components can be expressed as a map from the post-collisional state
$(s_1, \omega_1)$ (after the last in a sequence
of reflection off of $\Gamma$)
to the pre-collisional state
$(s_2, \omega_2)$
of the first in a sequence of reflection off of the same boundary
component $\Gamma$
\begin{equation}
  \label{eqn_def_freeFlight}
  \Gamma(s_1)
  +
  \tau
  \begin{pmatrix}
    \cos\omega_1 \\
    \sin\omega_1
  \end{pmatrix}
  =
  \begin{pmatrix}
    -1 & 0 \\
    0 & 1
  \end{pmatrix}
  \Gamma(s_2)
  -
  \begin{pmatrix}
    L \\
    n\,W
  \end{pmatrix}
  \;,\quad
  \omega_2 = 3\pi - \omega_1
  \;,
\end{equation}
where $\tau$ is the length of the free flight.
In terms of the local flow coordinates \eqref{def_localFlowCoordinates_01}
corresponding to
$ (s_1, \omega_1) $,
$ (s_2, \omega_2) $
the expression \eqref{eqn_def_freeFlight}
for the free flight becomes
\begin{equation}
  \label{def_freeFlight_localFlowCoordinates_12}
  -\mathsf{x}_2
  = 
  \mathsf{x}_1
  +
  \Big[
  2\,\Gamma(\hat{s}_1)
  +
  \begin{pmatrix}
    L \\
    n\,W
  \end{pmatrix}
  \Big]
  \cdot
  \begin{pmatrix}
    \cos( \hat{\omega}_1 + \mathsf{y}_1 ) \\
    \sin( \hat{\omega}_1 + \mathsf{y}_1 )
  \end{pmatrix}^\perp
  \;,\quad
  -\mathsf{y}_2
  =
  \mathsf{y}_1
  \;.
\end{equation}

\begin{lemma}
  \label{lem_def_Ln_ybar0}
  The part of a billiard trajectory corresponding to
  $
  (\bar{\mathsf{x}}_0, \bar{\mathsf{y}}_0)
  \mapsto
  (\bar{\mathsf{x}}_1, \bar{\mathsf{y}}_1)
  $
  is part of a periodic orbit if and only if
  the separation distance $L$ of the two curved boundary
  segments has the form
  \begin{equation*}
    L
    \equiv
    L_n(\bar{\mathsf{y}}_0)
    = 
    \frac{
    2\,f(\bar{\mathsf{y}}_0)
    +
    n\,W \,\cos( \hat{\omega}_1 - \bar{\mathsf{y}}_0 )
    +
    2\,\Gamma(\hat{s}_1)
    \cdot
    \begin{pmatrix}
      \cos( \hat{\omega}_1 - \bar{\mathsf{y}}_0 ) \\
      \sin( \hat{\omega}_1 - \bar{\mathsf{y}}_0 )
    \end{pmatrix}^\perp
    }{
    \sin( \hat{\omega}_1 - \bar{\mathsf{y}}_0 )
    }
    \;.
  \end{equation*}
  for any $n=0,1,\ldots$.
\end{lemma}
\begin{proof}
  By \eqref{eqn_relation_barFlow_01} we have
  $
  (\bar{\mathsf{x}}_1, \bar{\mathsf{y}}_1)
  =
  (\bar{\mathsf{x}}_0, -\bar{\mathsf{y}}_0)
  $.
  Therefore, it follows from
  \eqref{def_freeFlight_localFlowCoordinates_12}
  that 
  $
  (\bar{\mathsf{x}}_0, \bar{\mathsf{y}}_0)
  =
  (\bar{\mathsf{x}}_2, \bar{\mathsf{y}}_2)
  $
  holds if and only if
  \begin{equation*}
    -\bar{\mathsf{x}}_0
    = 
    \bar{\mathsf{x}}_0
    +
    \Big[
    2\,\Gamma(\hat{s}_1)
    +
    \begin{pmatrix}
      L \\
      n\,W
    \end{pmatrix}
    \Big]
    \cdot
    \begin{pmatrix}
      \cos( \hat{\omega}_1 - \bar{\mathsf{y}}_0 ) \\
      \sin( \hat{\omega}_1 - \bar{\mathsf{y}}_0 )
    \end{pmatrix}^\perp
    \;.
  \end{equation*}
  Solving for $L$ gives
  \begin{equation*}
    L
    = 
    \frac{
    2\,\bar{\mathsf{x}}_0
    +
    n\,W \,\cos( \hat{\omega}_1 - \bar{\mathsf{y}}_0 )
    +
    2\,\Gamma(\hat{s}_1)
    \cdot
    \begin{pmatrix}
      \cos( \hat{\omega}_1 - \bar{\mathsf{y}}_0 ) \\
      \sin( \hat{\omega}_1 - \bar{\mathsf{y}}_0 )
    \end{pmatrix}^\perp
    }{
    \sin( \hat{\omega}_1 - \bar{\mathsf{y}}_0 )
    }
    \;.
  \end{equation*}
  Finally, recall that
  $ \bar{\mathsf{x}}_0 = f(\bar{\mathsf{y}}_0)$
  which proves the claim.
\end{proof}

For tables with separation distance equal to
$L_n(\bar{\mathsf{y}}_0)$
we thus have a periodic orbit corresponding to
$\bar{\mathsf{y}}_0$, whose stability we will investigate
in the remaining part of this section.
For this purpose we introduce local coordinates in a neighborhood of
$ \bar{\mathsf{x}}_0, \bar{\mathsf{y}}_0 $
and of
$ \bar{\mathsf{x}}_1, \bar{\mathsf{y}}_1 $
by
\begin{equation}
  \label{eqn_def_localCoordinates_aboutPeriodicOrbit}
  \delta\mathsf{x}_0
  =
  \mathsf{x}_0
  -
  \bar{\mathsf{x}}_0
  \;,\quad
  \delta\mathsf{y}_0
  =
  \mathsf{y}_0
  -
  \bar{\mathsf{y}}_0
  \;,\quad
  \delta\mathsf{x}_1
  =
  \mathsf{x}_1
  -
  \bar{\mathsf{x}}_1
  \;,\quad
  \delta\mathsf{y}_1
  =
  \mathsf{y}_1
  -
  \bar{\mathsf{y}}_1
  \;.
\end{equation}

\begin{proposition}
  \label{prop_freeFlight_aboutPeriodic_12}
  For any $n=0,1,\ldots$, on a table with separation distance equal to
  $L_n(\bar{\mathsf{y}}_0)$
  the free flight map has the form
  \begin{align*}
    -\delta\mathsf{x}_2
    &=
    \delta\mathsf{x}_1
    +
    \Big[
    2\,f(\bar{\mathsf{y}}_0)
    \,\Big(
    \frac{ \sin\delta\mathsf{y}_1
    }{ 1 + \cos\delta\mathsf{y}_1 }
    -
    \cot( \hat{\omega}_1 - \bar{\mathsf{y}}_0 )
    \Big)
    -
    \frac{ 2\,\Gamma_y(\hat{s}_1) + n\,W
    }{ \sin( \hat{\omega}_1 - \bar{\mathsf{y}}_0 ) }
    \Big]
    \,\sin\delta\mathsf{y}_1
    \\
    -\delta\mathsf{y}_2
    &=
    \delta\mathsf{y}_1
  \end{align*}
  in a neighborhood of $\bar{\mathsf{y}}_0$.
\end{proposition}
\begin{proof}
  For $L = L_n(\bar{\mathsf{y}}_0)$ the expression
  \eqref{def_freeFlight_localFlowCoordinates_12}
  for the free flight map becomes
  \begin{equation*}
    -\mathsf{x}_2
    = 
    \mathsf{x}_1
    +
    \Big[
    2\,\Gamma(\hat{s}_1)
    +
    \begin{pmatrix}
      L_n(\bar{\mathsf{y}}_0) \\
      n\,W
    \end{pmatrix}
    \Big]
    \cdot
    \begin{pmatrix}
      \cos( \hat{\omega}_1 + \mathsf{y}_1 ) \\
      \sin( \hat{\omega}_1 + \mathsf{y}_1 )
    \end{pmatrix}^\perp
    \;,\quad
    -\mathsf{y}_2
    =
    \mathsf{y}_1
    \;.
  \end{equation*}
  Using the definition of the local coordinates near the periodic
  orbit it then follows that
  \begin{equation*}
    -\delta\mathsf{y}_2
    =
    - [ \mathsf{y}_2 - \bar{\mathsf{y}}_2 ]
    =
    - [ \mathsf{y}_2 - \bar{\mathsf{y}}_0 ]
    =
    \mathsf{y}_1 + \bar{\mathsf{y}}_0
    =
    \mathsf{y}_1 - \bar{\mathsf{y}}_1
    =
    \delta\mathsf{y}_1
  \end{equation*}
  and
  \begin{align*}
    -\delta\mathsf{x}_2
    &=
    -[ \mathsf{x}_2 - \bar{\mathsf{x}}_2 ]
    \\
    &=
    \mathsf{x}_1
    +
    \Big[
    2\,\Gamma(\hat{s}_1)
    +
    \begin{pmatrix}
      L_n(\bar{\mathsf{y}}_0) \\
      n\,W
    \end{pmatrix}
    \Big]
    \cdot
    \begin{pmatrix}
      \cos( \hat{\omega}_1 + \bar{\mathsf{y}}_1
      + \delta\mathsf{y}_1 )
      \\
      \sin( \hat{\omega}_1 + \bar{\mathsf{y}}_1
      + \delta\mathsf{y}_1 )
    \end{pmatrix}^\perp
    \\
    &\quad
    -
    \bar{\mathsf{x}}_1
    -
    \Big[
    2\,\Gamma(\hat{s}_1)
    +
    \begin{pmatrix}
      L_n(\bar{\mathsf{y}}_0) \\
      n\,W
    \end{pmatrix}
    \Big]
    \cdot
    \begin{pmatrix}
      \cos( \hat{\omega}_1 + \bar{\mathsf{y}}_1 )
      \\
      \sin( \hat{\omega}_1 + \bar{\mathsf{y}}_1 )
    \end{pmatrix}^\perp
    \\
    &=
    \delta\mathsf{x}_1
    +
    \Big[
    2\,\Gamma(\hat{s}_1)
    +
    \begin{pmatrix}
      L_n(\bar{\mathsf{y}}_0) \\
      n\,W
    \end{pmatrix}
    \Big]
    \cdot
    \begin{pmatrix}
      \cos( \hat{\omega}_1 + \bar{\mathsf{y}}_1
      + \delta\mathsf{y}_1 )
      -
      \cos( \hat{\omega}_1 + \bar{\mathsf{y}}_1 )
      \\
      \sin( \hat{\omega}_1 + \bar{\mathsf{y}}_1
      + \delta\mathsf{y}_1 )
      -
      \sin( \hat{\omega}_1 + \bar{\mathsf{y}}_1 )
    \end{pmatrix}^\perp
    \;.
  \end{align*}
  By making use of
  \begin{align*}
    \begin{pmatrix}
      \cos( \hat{\omega}_1 + \bar{\mathsf{y}}_1
      + \delta\mathsf{y}_1 )
      -
      \cos( \hat{\omega}_1 + \bar{\mathsf{y}}_1 )
      \\
      \sin( \hat{\omega}_1 + \bar{\mathsf{y}}_1
      + \delta\mathsf{y}_1 )
      -
      \sin( \hat{\omega}_1 + \bar{\mathsf{y}}_1 )
    \end{pmatrix}^\perp
    &=
    -
    \begin{pmatrix}
      \cos( \hat{\omega}_1 + \bar{\mathsf{y}}_1 )
      \\
      \sin( \hat{\omega}_1 + \bar{\mathsf{y}}_1 )
    \end{pmatrix}^\perp
    [ 1 - \cos\delta\mathsf{y}_1 ]
    \\
    &\quad
    -
    \begin{pmatrix}
      \cos( \hat{\omega}_1 + \bar{\mathsf{y}}_1 )
      \\
      \sin( \hat{\omega}_1 + \bar{\mathsf{y}}_1 )
    \end{pmatrix}
    \,\sin\delta\mathsf{y}_1
  \end{align*}
  we get
  \begin{align*}
    -\delta\mathsf{x}_2
    &=
    \delta\mathsf{x}_1
    -
    [ 1 - \cos\delta\mathsf{y}_1 ]
    \,\Big[
    2\,\Gamma(\hat{s}_1)
    +
    \begin{pmatrix}
      L_n(\bar{\mathsf{y}}_0) \\
      n\,W
    \end{pmatrix}
    \Big]
    \cdot
    \begin{pmatrix}
      \cos( \hat{\omega}_1 + \bar{\mathsf{y}}_1 )
      \\
      \sin( \hat{\omega}_1 + \bar{\mathsf{y}}_1 )
    \end{pmatrix}^\perp
    \\
    &\quad
    -
    \sin\delta\mathsf{y}_1
    \,\Big[
    2\,\Gamma(\hat{s}_1)
    +
    \begin{pmatrix}
      L_n(\bar{\mathsf{y}}_0) \\
      n\,W
    \end{pmatrix}
    \Big]
    \cdot
    \begin{pmatrix}
      \cos( \hat{\omega}_1 + \bar{\mathsf{y}}_1 )
      \\
      \sin( \hat{\omega}_1 + \bar{\mathsf{y}}_1 )
    \end{pmatrix}
    \\
    &=
    \delta\mathsf{x}_1
    -
    [ 1 - \cos\delta\mathsf{y}_1 ]
    \,[
    -
    \bar{\mathsf{x}}_2
    -
    \bar{\mathsf{x}}_1
    ]
    -
    \sin\delta\mathsf{y}_1
    \,\Big[
    2\,\Gamma(\hat{s}_1)
    +
    \begin{pmatrix}
      L_n(\bar{\mathsf{y}}_0) \\
      n\,W
    \end{pmatrix}
    \Big]
    \cdot
    \begin{pmatrix}
      \cos( \hat{\omega}_1 + \bar{\mathsf{y}}_1 )
      \\
      \sin( \hat{\omega}_1 + \bar{\mathsf{y}}_1 )
    \end{pmatrix}
    \\
    &=
    \delta\mathsf{x}_1
    +
    2\,[ 1 - \cos\delta\mathsf{y}_1 ]\,\bar{\mathsf{x}}_0
    \\
    &\quad
    -
    \sin\delta\mathsf{y}_1
    \,\Big[
    [
    2\,\Gamma_x(\hat{s}_1)
    +
    L_n(\bar{\mathsf{y}}_0)
    ]
    \,\cos( \hat{\omega}_1 - \bar{\mathsf{y}}_0 )
    +
    [
    2\,\Gamma_y(\hat{s}_1)
    +
    n\,W
    ]
    \,\sin( \hat{\omega}_1 - \bar{\mathsf{y}}_0 )
    \Big]
    \;.
  \end{align*}
  Using the expression for
  $L_n(\bar{\mathsf{y}}_0)$
  given in Lemma~\ref{lem_def_Ln_ybar0}
  we conclude that
  \begin{align*}
    -\delta\mathsf{x}_2
    &=
    \delta\mathsf{x}_1
    +
    2\,f(\bar{\mathsf{y}}_0)
    \,\Big[
    \frac{ 1 - \cos\delta\mathsf{y}_1 }{ \sin\delta\mathsf{y}_1 }
    -
    \frac{ \cos( \hat{\omega}_1 - \bar{\mathsf{y}}_0 )
    }{ \sin( \hat{\omega}_1 - \bar{\mathsf{y}}_0 ) }
    \Big]
    \,\sin\delta\mathsf{y}_1
    -
    \frac{ 2\,\Gamma_y(\hat{s}_1) + n\,W
    }{ \sin( \hat{\omega}_1 - \bar{\mathsf{y}}_0 ) }
    \,\sin\delta\mathsf{y}_1
    \;,
  \end{align*}
  which finishes the proof.
\end{proof}

The result of Proposition~\ref{prop_freeFlight_aboutPeriodic_12}
is an expression for the free flight across the table
with separation distance equal to $L_n(\bar{\mathsf{y}}_0)$
in a neighborhood of the periodic orbit corresponding to
$\bar{\mathsf{y}}_0$. This provides the mapping
$
( \delta\mathsf{x}_1, \delta\mathsf{y}_1 )
\mapsto
( \delta\mathsf{x}_2, \delta\mathsf{y}_2 )
$, which maps a neighborhood of
$ (\bar{\mathsf{x}}_1, \bar{\mathsf{y}}_1) $
to some neighborhood of
$ (\bar{\mathsf{x}}_0, \bar{\mathsf{y}}_0) $.
Introduce the map
$
( \delta\mathsf{x}_1, \delta\mathsf{y}_1 )
=
\bar{{\mathsf F}}_{ \bar{\mathsf{y}}_0 }(
\delta\mathsf{x}_0, \delta\mathsf{y}_0 )
$, defined in some neighborhood of $(0,0)$, as
\begin{equation}
  \label{eqn_def_reducedMonodromyMapYbar}
  \bar{{\mathsf F}}_{ \bar{\mathsf{y}}_0 }(
  \delta\mathsf{x}_0, \delta\mathsf{y}_0 )
  =
  {\mathsf F}(
  f( \bar{\mathsf{y}}_0 ) + \delta\mathsf{x}_0,
  \bar{\mathsf{y}}_0 + \delta\mathsf{y}_0
  )
  -
  {\mathsf F}( f( \bar{\mathsf{y}}_0 ) ,
  \bar{\mathsf{y}}_0 )
  \;.
\end{equation}
Therefore, composing the map
$\bar{{\mathsf F}}_{ \bar{\mathsf{y}}_0 }$,
defined in \eqref{eqn_def_reducedMonodromyMapYbar}, with the
local expression of the free flight provided in 
Proposition~\ref{prop_freeFlight_aboutPeriodic_12}
we obtain the first return map
${\mathsf T}_{ \bar{\mathsf{y}}_0 }$
on the table with separation
distance equal to $L_n(\bar{\mathsf{y}}_0)$ along the periodic orbit
corresponding to $\bar{\mathsf{y}}_0$.
Indeed, the explicit expression for
$
( \delta\mathsf{x}_2, \delta\mathsf{y}_2 )
=
{\mathsf T}_{ \bar{\mathsf{y}}_0 }(
\delta\mathsf{x}_0, \delta\mathsf{y}_0 )
$
readily follows from the above as
\begin{equation}
  \label{eqn_def_monodromyMapYbar}
  \begin{split}
    ( \delta\mathsf{x}_1, \delta\mathsf{y}_1 )
    &=
    \bar{{\mathsf F}}_{ \bar{\mathsf{y}}_0 }(
    \delta\mathsf{x}_0, \delta\mathsf{y}_0 )
    \\
    -\delta\mathsf{x}_2
    &=
    \delta\mathsf{x}_1
    +
    \Big[
    2\,f(\bar{\mathsf{y}}_0)
    \,\Big(
    \frac{ \sin\delta\mathsf{y}_1
    }{ 1 + \cos\delta\mathsf{y}_1 }
    -
    \cot( \hat{\omega}_1 - \bar{\mathsf{y}}_0 )
    \Big)
    -
    \frac{ 2\,\Gamma_y(\hat{s}_1) + n\,W
    }{ \sin( \hat{\omega}_1 - \bar{\mathsf{y}}_0 ) }
    \Big]
    \,\sin\delta\mathsf{y}_1
    \\
    -\delta\mathsf{y}_2
    &=
    \delta\mathsf{y}_1
    \;,
  \end{split}
\end{equation}
which is well defined in a neighborhood of $(0,0)$ (which corresponds
to a neighborhood of $( \mathsf{x}_0, \mathsf{y}_0 )$
in terms of $( \mathsf{x}, \mathsf{y} )$--coordinates).
We finish this section by recording that
the above construction proves:
\begin{lemma}
  \label{lem_monodromyMapYbar_basicProp}
  The first return map
  ${\mathsf T}_{ \bar{\mathsf{y}}_0 }$
  is an orientation and area preserving map satisfying
  ${\mathsf T}_{ \bar{\mathsf{y}}_0 }(0,0) = (0,0)$.
\end{lemma}

\subsection{Linear stability analysis}

In view of Lemma~\ref{lem_monodromyMapYbar_basicProp},
the area preservation property of
${\mathsf T}_{ \bar{\mathsf{y}}_0 }$ implies that
the periodic orbit corresponding to $\bar{\mathsf{y}}_0$
on the table with separation distance equal to
$L_n(\bar{\mathsf{y}}_0)$ is linearly stable if
\begin{equation}
  \label{eqn_linearStabilityCriterion}
  | \operatorname{tr} \mathrm{D} {\mathsf T}_{ \bar{\mathsf{y}}_0 }(0,0) |
  <
  2
  \;.
\end{equation}
It follows immediately from the definition of
${\mathsf T}_{ \bar{\mathsf{y}}_0 }$
that
\begin{equation}
  \label{eqn_traceDreturnMap}
  \begin{split}
    \operatorname{tr} \mathrm{D} {\mathsf T}_{ \bar{\mathsf{y}}_0 }(0,0)
    &=
    -
    \operatorname{tr} \mathrm{D}{\mathsf F}( f( \bar{\mathsf{y}}_0 ),
    \bar{\mathsf{y}}_0 )
    \\
    &\quad
    +
    \frac{
    2\,f(\bar{\mathsf{y}}_0)
    \,\cos( \hat{\omega}_1 - \bar{\mathsf{y}}_0 )
    +
    2\,\Gamma_y(\hat{s}_1)
    +
    n\,W
    }{ \sin( \hat{\omega}_1 - \bar{\mathsf{y}}_0 ) }
    \,\frac{ \partial{\mathsf F}_{ \mathsf{y} }
    }{ \partial \mathsf{x} }( f( \bar{\mathsf{y}}_0 ),
    \bar{\mathsf{y}}_0 )
    \;.
  \end{split}
\end{equation}
In particular, for the periodic orbit to be stable
we need the value of the expressions
\begin{equation*}
  n\,W
  \,\frac{ \partial{\mathsf F}_{ \mathsf{y} }
  }{ \partial \mathsf{x} }( f( \bar{\mathsf{y}}_0 ),
  \bar{\mathsf{y}}_0 )
\end{equation*}
to be not too large in absolute value. Hence, for large values of
$n$ it follows that
\begin{equation*}
  | \bar{\mathsf{y}}_0 | \ll 1
\end{equation*}
must hold as $n \to \infty$ provided that
\begin{equation}
  \label{eqn_sizeOfy_vs_n}
  n\,W
  \,\frac{ \partial{\mathsf F}_{ \mathsf{y} }
  }{ \partial \mathsf{x} }( f( \bar{\mathsf{y}}_0 ),
  \bar{\mathsf{y}}_0 )
  =
  \operatorname{\mathcal{O}}(1)
  \;.
\end{equation}
In other words, the periodic orbit corresponding to 
$\bar{\mathsf{y}}_0$ on a table with parameter $n$
can only be stable if
$\bar{\mathsf{y}}_0 \to 0$ as $n \to \infty$.
Therefore, we will focus our attention to a local description
of $ {\mathsf T}_{ \bar{\mathsf{y}}_0 } $
for small values of $\bar{\mathsf{y}}_0$.

By combining Lemma~\ref{lem_def_Ln_ybar0} and \eqref{eqn_traceDreturnMap}
the following fact relating 
$\operatorname{tr} \mathrm{D} {\mathsf T}_{ \bar{\mathsf{y}}_0 }(0,0)$
and
$L_n( \bar{\mathsf{y}}_0 )$ follows immediately.
\begin{lemma}
  \label{lem_identity_Ln_y}
  For all $\bar{\mathsf{y}}_0$
  \begin{align*}
    L_n(\bar{\mathsf{y}}_0)
    &= 
    L_n(0)
    +
    \frac{
    2\,\Gamma_y(\hat{s}_1)
    +
    n\,W
    }{
    \sin( \hat{\omega}_1 - \bar{\mathsf{y}}_0 )
    }
    \,\frac{ \sin\bar{\mathsf{y}}_0
    }{ \sin\hat{\omega}_1 }
    +
    \frac{
    2\,f(\bar{\mathsf{y}}_0)
    }{
    \sin( \hat{\omega}_1 - \bar{\mathsf{y}}_0 )
    }
    \\
    &= 
    L_n(0)
    \,\frac{ \tan\hat{\omega}_1
    }{ \tan( \hat{\omega}_1 - \bar{\mathsf{y}}_0 ) }
    +
    2\,\frac{
    f(\bar{\mathsf{y}}_0)
    +
    [
    1
    -
    \frac{ \cos( \hat{\omega}_1 - \bar{\mathsf{y}}_0 )
    }{ \cos\hat{\omega}_1 }
    ]
    \,\Gamma(\hat{s}_1)
    \cdot
    \begin{pmatrix}
      \cos\hat{\omega}_1 \\
      \sin\hat{\omega}_1
    \end{pmatrix}^\perp
    }{
    \sin( \hat{\omega}_1 - \bar{\mathsf{y}}_0 )
    }
    \;.
  \end{align*}
  Furthermore, the relation
  \begin{align*}
    L_n(\bar{\mathsf{y}}_0)
    -
    L_n(0)
    &= 
    \frac{
    \operatorname{tr} \mathrm{D} {\mathsf T}_{ \bar{\mathsf{y}}_0 }(0,0)
    +
    \operatorname{tr} \mathrm{D}{\mathsf F}( f( \bar{\mathsf{y}}_0 ),
    \bar{\mathsf{y}}_0 )
    }{
    \sin\hat{\omega}_1
    }
    \,\frac{ \sin\bar{\mathsf{y}}_0
    }{
    \frac{ \partial{\mathsf F}_{ \mathsf{y} }
    }{ \partial \mathsf{x} }( f( \bar{\mathsf{y}}_0 ),
    \bar{\mathsf{y}}_0 )
    }
    \\
    &\quad
    -
    \frac{
    2\,f(\bar{\mathsf{y}}_0)
    \,\cos( \hat{\omega}_1 - \bar{\mathsf{y}}_0 )
    }{
    \sin\hat{\omega}_1
    \,\sin( \hat{\omega}_1 - \bar{\mathsf{y}}_0 )
    }
    \,\sin\bar{\mathsf{y}}_0
    +
    \frac{
    2\,f(\bar{\mathsf{y}}_0)
    }{
    \sin( \hat{\omega}_1 - \bar{\mathsf{y}}_0 )
    }
  \end{align*}
  holds for all $\bar{\mathsf{y}}_0$ for which
  $
  \frac{ \partial{\mathsf F}_{ \mathsf{y} }
  }{ \partial \mathsf{x} }( f( \bar{\mathsf{y}}_0 ),
  \bar{\mathsf{y}}_0 )
  \neq
  0
  $.
\end{lemma}

\subsection{Local analysis of the return map along periodic orbits}

In order to proceed with the stability analysis of the
periodic orbit corresponding to $\bar{\mathsf{y}}_0$
on the table with separation distance equal to
$L_n(\bar{\mathsf{y}}_0)$
we only need a local description of the corresponding first return map
${\mathsf T}_{ \bar{\mathsf{y}}_0 }$.
From its explicit expression given in \eqref{eqn_def_monodromyMapYbar}
we see that the local representation of the return map is
determined by the local form of
$\bar{{\mathsf F}}_{ \bar{\mathsf{y}}_0 }$,
which is in turn determined by a local description of
${\mathsf F}$.

Since the boundary component $\Gamma$ is of
class $C^5$
the map ${\mathsf F}$ is of class $C^4$, and hence
in some neighborhood of $(0,0)$ it has the form
\begin{equation}
  \label{eqn_def_reducedMonodromyMap_powerSeries}
  {\mathsf F}(\mathsf{x}_0, \mathsf{y}_0)
  =
  \begin{pmatrix}
    -\mathsf{x}_0 + a_{01}\,\mathsf{y}_0
    +
    a_{20}\,\mathsf{x}_0^2
    +
    a_{11}\,\mathsf{x}_0\,\mathsf{y}_0
    +
    \ldots
    +
    a_{04}\,\mathsf{y}_0^4
    \\
    -\mathsf{y}_0
    +
    b_{20}\,\mathsf{x}_0^2
    +
    b_{11}\,\mathsf{x}_0\,\mathsf{y}_0
    +
    \ldots
    +
    b_{04}\,\mathsf{y}_0^4
  \end{pmatrix}
  +
  \operatorname{o}_4(\mathsf{x}_0,\mathsf{y}_0)
  ,
\end{equation}
where $a_{10} = -1$, $a_{01} \neq 0$, $b_{10} = 0$, $b_{01} = -1$.
Clearly, this local representation of ${\mathsf F}$
only incorporates
\eqref{eqn_reducedMonodromyMap_basicProps}.
However,
\eqref{eqn_reducedMonodromyMap_detD}
and
\eqref{eqn_reducedMonodromyMap_symmetry}
give additional restrictions on the possible values of the various
coefficients $a_{ij}$, $b_{ij}$ in
\eqref{eqn_def_reducedMonodromyMap_powerSeries}.
Indeed, substituting 
\eqref{eqn_def_reducedMonodromyMap_powerSeries}
in
\eqref{eqn_reducedMonodromyMap_detD}
and
\eqref{eqn_reducedMonodromyMap_symmetry}
yields after some straightforward but rather tedious computation
the following:
\begin{proposition}
  \label{prop_reducedMonodromyMap_coeffsRelations}
  The conditions
  \eqref{eqn_reducedMonodromyMap_basicProps},
  \eqref{eqn_reducedMonodromyMap_detD},
  \eqref{eqn_reducedMonodromyMap_symmetry}
  hold to order $3$
  if and only if
  \begin{align*}
    a_{11} &=0
    \;,\quad
    b_{20} =0
    \;,\quad
    b_{11} =-2\,a_{20}
    \;,\quad
    b_{02} =a_{01}\,a_{20}
    \\
    b_{30}
    &=
    -
    2\,\frac{ a_{20}^2 + a_{30} }{ a_{01} }
    \;,\quad
    b_{21}
    =
    3\,a_{30} + 2\,a_{20}^2
    \;,\quad
    b_{12}
    =
    \frac{ 2\,a_{12} - 2\,a_{02}\,a_{20} - 3\,a_{01}^2\,a_{30}
    }{ a_{01} }
    \\
    b_{03}
    &=
    a_{01}^2\,a_{30} + 2\,a_{02}\,a_{20} - a_{12}
    \;,\quad
    a_{21}
    =
    2\,\frac{ a_{02}\,a_{20} - a_{12} }{ a_{01} }
  \end{align*}
  for the coefficients in \eqref{eqn_def_reducedMonodromyMap_powerSeries},
  where it is assumed that $a_{01} \neq 0$.
\end{proposition}

As an immediate consequence of
Lemma~\ref{lem_funcX_construction},
\eqref{eqn_def_reducedMonodromyMap_powerSeries}, and
Proposition~\ref{prop_reducedMonodromyMap_coeffsRelations}
we obtain the local representation of $f$ as
\begin{equation}
  \label{eqn_funcX_localExpansion}
  \begin{split}
    f(\mathsf{y})
    &=
    \frac{ a_{01} }{2}
    \,\mathsf{y}
    +
    \Big(
    \frac{ a_{02} }{2}
    +
    \frac{ a_{01}^2\,a_{20} }{8}
    \Big)
    \,\mathsf{y}^2
    \\
    &\quad
    +
    \Big(
    \frac{ a_{03} }{2}
    +
    \frac{ a_{01}\,a_{02}\,a_{20} }{2}
    +
    \frac{ a_{01}^3\,a_{20}^2 }{16}
    +
    \frac{ a_{01}^3\,a_{30} }{16}
    \Big)
    \,\mathsf{y}^3
    +
    \operatorname{\mathcal{O}}(\mathsf{y}^4)
    \;,
  \end{split}
\end{equation}
which holds for all $\mathsf{y}$ in a neighborhood of $0$.

\begin{lemma}
  \label{lem_funcOmegaLocalPeriodicNTr}
  Suppose that $a_{20} \neq 0$. Then
  there exists an integer $n_0 \in {\mathbb N}$
  and
  $\bar{\Upsilon}_{n} \in C^4[-1,5]$
  such that for all $n \geq n_0$ the equality
  $\operatorname{tr} \mathrm{D} {\mathsf T}_{ \bar{\mathsf{y}}_0 }(0,0) = 2-t$
  holds if and only if
  $
  \bar{\mathsf{y}}_0 = \bar{\Upsilon}_{n}(t)
  $, and all all derivatives of
  $\bar{\Upsilon}_{n}(t)$
  are uniformly bounded in $n \geq n_0$.
\end{lemma}
\begin{proof}
  Let
  $2-t = \operatorname{tr} \mathrm{D} {\mathsf T}_{ \bar{\mathsf{y}}_0 }(0,0)$.
  Then
  we can re-write \eqref{eqn_traceDreturnMap}
  as
  \begin{equation*}
    - t = H_n( \bar{\mathsf{y}}_0 )
  \end{equation*}
  where
  \begin{equation*}
    H_n( \bar{\mathsf{y}}_0 )
    =
    -2
    -
    \operatorname{tr} \mathrm{D}{\mathsf F}( f( \bar{\mathsf{y}}_0 ),
    \bar{\mathsf{y}}_0 )
    +
    \frac{
    2\,f(\bar{\mathsf{y}}_0)
    \,\cos( \hat{\omega}_1 - \bar{\mathsf{y}}_0 )
    +
    2\,\Gamma_y(\hat{s}_1)
    +
    n\,W
    }{ \sin( \hat{\omega}_1 - \bar{\mathsf{y}}_0 ) }
    \,\frac{ \partial{\mathsf F}_{ \mathsf{y} }
    }{ \partial \mathsf{x} }( f( \bar{\mathsf{y}}_0 ),
    \bar{\mathsf{y}}_0 )
    \;.
  \end{equation*}
  Since
  $
  \frac{ \partial{\mathsf F}_{ \mathsf{y} }
  }{ \partial \mathsf{x} }( f( \bar{\mathsf{y}}_0 ),
  \bar{\mathsf{y}}_0 )
  |_{\bar{\mathsf{y}}_0 = 0 }
  =
  0
  $
  and
  $
  \operatorname{tr} \mathrm{D}{\mathsf F}( f( \bar{\mathsf{y}}_0 ),
  \bar{\mathsf{y}}_0 )
  |_{\bar{\mathsf{y}}_0 = 0 }
  =
  -2
  $
  it follows that
  $t=0$, $\bar{\mathsf{y}}_0 = 0$ satisfy the above equation
  $-t=H_n(\bar{\mathsf{y}}_0)$ for every $n$.

  Notice that
  \begin{align*}
    H_n'( \bar{\mathsf{y}}_0 )
    |_{\bar{\mathsf{y}}_0 = 0 }
    &=
    -
    \frac{d}{d \bar{\mathsf{y}}_0 }
    \operatorname{tr} \mathrm{D}{\mathsf F}( f( \bar{\mathsf{y}}_0 ),
    \bar{\mathsf{y}}_0 )
    \Big|_{\bar{\mathsf{y}}_0 = 0 }
    +
    \frac{
    2\,\Gamma_y(\hat{s}_1)
    +
    n\,W
    }{ \sin \hat{\omega}_1 }
    \,
    \frac{d}{d \bar{\mathsf{y}}_0 }
    \frac{ \partial{\mathsf F}_{ \mathsf{y} }
    }{ \partial \mathsf{x} }( f( \bar{\mathsf{y}}_0 ),
    \bar{\mathsf{y}}_0 )
    \Big|_{\bar{\mathsf{y}}_0 = 0 }
    \\
    &=
    -
    2\,a_{20}
    \,\Big[
    a_{01}
    +
    \frac{
    2\,\Gamma_y(\hat{s}_1)
    +
    n\,W
    }{ \sin \hat{\omega}_1 }
    \,\Big]
  \end{align*}
  where we used
  \begin{equation}
    \label{eqn_expansion_trD_dFydx }
    \begin{split}
      \operatorname{tr} \mathrm{D}{\mathsf F}(\mathsf{x}, \mathsf{y})
      &=
      -2
      +
      2\,a_{01}\,a_{20} \,\mathsf{y}
      +
      \operatorname{\mathcal{O}}_2(\mathsf{x}, \mathsf{y})
      \\
      \frac{ \partial{\mathsf F}_{ \mathsf{y} }
      }{ \partial \mathsf{x} }(\mathsf{x}, \mathsf{y})
      &=
      -2\,a_{20} \,\mathsf{y}
      +
      \operatorname{\mathcal{O}}_2(\mathsf{x}, \mathsf{y})
      \;,
    \end{split}
  \end{equation}
  which follow immediately from
  \eqref{eqn_def_reducedMonodromyMap_powerSeries},
  Proposition~\ref{prop_reducedMonodromyMap_coeffsRelations}.

  Since we assume that $a_{20} \neq 0$
  we conclude that
  $H_n'( \bar{\mathsf{y}}_0 )
  |_{\bar{\mathsf{y}}_0 = 0 } \neq 0$
  for but possibly one $n$. Hence, for all but possibly one value of
  $n$, the implicit function theorem guarantees the existence of
  $\bar{\Upsilon}_{n}(t)$, defined in some
  neigborhood of $t=0$, such that
  \begin{equation*}
    -t=H_n( \bar{\Upsilon}_{n}(t) )
    \;,\qquad
    \bar{\Upsilon}_{n}(0) = 0
    \;,
  \end{equation*}
  which provides the relation
  $\bar{\mathsf{y}}_0 = \bar{\Upsilon}_{n}(2-\Theta)$
  between the trace of the monodromy matrix and the choice of
  $\bar{\mathsf{y}}_0$ that realizes the corresponing
  periodic orbit. Furthermore, 
  $\bar{\Upsilon}_{n}(t)$ is defined (at least) on the
  largest interval containing zero on which the initial value problem
  \begin{equation*}
    \frac{d}{dt}\bar{\Upsilon}_{n}(t)
    =
    -\frac{1}{ H_n'( \bar{\Upsilon}_{n}(t) ) }
    \;,\qquad
    \bar{\Upsilon}_{n}(0) = 0
  \end{equation*}
  has a well-defined solution.
  Since we are interested in guaranteeing that
  $\bar{\Upsilon}_{n}(t)$ is well-defined for
  at least all $0 \leq t \leq 2$ we seek to find estimates
  on the domain on which the above initial value problem as
  a well-defined solution. For this purpose we simplify notation
  and introduce
  \begin{align*}
    h_1( \bar{\mathsf{y}}_0 )
    &=
    \frac{1
    }{ \sin( \hat{\omega}_1 - \bar{\mathsf{y}}_0 ) }
    \,\frac{ \partial{\mathsf F}_{ \mathsf{y} }
    }{ \partial \mathsf{x} }( f( \bar{\mathsf{y}}_0 ),
    \bar{\mathsf{y}}_0 )
    \\
    h_2( \bar{\mathsf{y}}_0 )
    &=
    -2
    -
    \operatorname{tr} \mathrm{D}{\mathsf F}( f( \bar{\mathsf{y}}_0 ),
    \bar{\mathsf{y}}_0 )
    +
    2\,f(\bar{\mathsf{y}}_0)
    \,\cos( \hat{\omega}_1 - \bar{\mathsf{y}}_0 )
    \,h_1( \bar{\mathsf{y}}_0 )
  \end{align*}
  so that the initial value problem for
  $\bar{\Upsilon}_{n}(t)$ takes on the form
  \begin{equation*}
    \frac{d}{dt}\bar{\Upsilon}_{n}(t)
    =
    -
    \frac{1}{
    h_2'( \bar{\Upsilon}_{n}(t) )
    +
    [ 2\,\Gamma_y(\hat{s}_1) + n\,W ]
    \,h_1'( \bar{\Upsilon}_{n}(t) )
    }
    \;,\qquad
    \bar{\Upsilon}_{n}(0) = 0
    \;.
  \end{equation*}
  As we had noted above already, using the assumption $a_{20}\neq 0$,
  \begin{align*}
    h_1( \bar{\mathsf{y}}_0 )
    &=
    -
    \frac{ 2\,a_{20} }{ \sin\hat{\omega}_1 }
    \,[ \bar{\mathsf{y}}_0 + R_1( \bar{\mathsf{y}}_0) ]
    \;,\quad
    R_1( \bar{\mathsf{y}}_0)
    =
    \operatorname{\mathcal{O}}(\bar{\mathsf{y}}_0^2)
    \;,
    \\
    h_2( \bar{\mathsf{y}}_0 )
    &=
    -
    2\,a_{20}
    \,[ a_{01} \,\bar{\mathsf{y}}_0 + R_2( \bar{\mathsf{y}}_0) ]
    \;,\quad
    R_2( \bar{\mathsf{y}}_0)
    =
    \operatorname{\mathcal{O}}(\bar{\mathsf{y}}_0^2)
    \;,
  \end{align*}
  where $R_1$, $R_2$ are $C^4$ functions
  independent of the value of $n$.
  Therefore, the initial value problem for
  $\bar{\Upsilon}_{n}(t)$
  takes on the form
  \begin{equation*}
    \frac{d}{dt}\bar{\Upsilon}_{n}(t)
    =
    \frac{1}{2\,a_{20}}
    \,\frac{1}{
    a_{01} + R_2'( \bar{\Upsilon}_{n}(t) )
    +
    \frac{ 2\,\Gamma_y(\hat{s}_1) + n\,W
    }{ \sin\hat{\omega}_1 }
    \,[ 1 + R_1'( \bar{\Upsilon}_{n}(t) ) ]
    }
    \;,\qquad
    \bar{\Upsilon}_{n}(0) = 0
    \;,
  \end{equation*}
  where
  $
  R_1'( \bar{\mathsf{y}}_0)
  =
  \operatorname{\mathcal{O}}(\bar{\mathsf{y}}_0)
  $,
  $
  R_2'( \bar{\mathsf{y}}_0)
  =
  \operatorname{\mathcal{O}}(\bar{\mathsf{y}}_0)
  $.
  From this it is now clear that there
  exists an integer $n_0 \in {\mathbb N}$ such that
  $\bar{\Upsilon}_{n}(t)$ is well-defined
  for all $-1 \leq t \leq 5$, and is of class $C^4$
  with all derivatives uniformly bounded in $n \geq n_0$.
  This completes the proof.
\end{proof}

For large values of $n \geq n_0$ the initial value problem
for $\bar{\Upsilon}_{n}(t)$ that was considered in
the above proof of Lemma~\ref{lem_funcOmegaLocalPeriodicNTr}
shows that
$\bar{\Upsilon}_{n}(t) = \operatorname{\mathcal{O}}(\frac{1}{n})$
uniformly in $-1 \leq t \leq 5$. This allows for the
asymptotic expansion in $\frac{1}{n}$
\begin{equation*}
  \frac{d}{dt}\bar{\Upsilon}_{n}(t)
  =
  \frac{1}{2\,a_{20}}
  \frac{ \sin\hat{\omega}_1 }{ n\,W }
  \,[ 1 + \operatorname{\mathcal{O}}(n^{-1}) ]
  \;,\qquad
  \bar{\Upsilon}_{n}(0) = 0
  \;,
\end{equation*}
because $R_1'( \bar{\Upsilon}_{n}(t) ) = \operatorname{\mathcal{O}}(n^{-1})$
uniformly for $-1 \leq t \leq 5$. Integrating this equation yields
\begin{equation}
  \label{eqn_expansion_funcOmegaLocalPeriodicNTr}
  \bar{\Upsilon}_{n}(t)
  =
  \frac{t}{2\,a_{20}}
  \frac{ \sin\hat{\omega}_1 }{ n\,W }
  +
  \operatorname{\mathcal{O}}(n^{-2})
  \;,\qquad
  \bar{\Upsilon}_{n}(0) = 0
\end{equation}
uniformly in $-1 \leq t \leq 5$.

\begin{proposition}
  \label{prop_elliptic_general}
  For $n \geq n_0$ and $-3 \leq \Theta \leq 3$
  the periodic orbit corresponding to
  $ \bar{\mathsf{y}}_0 = \bar{\Upsilon}_{n}(2-\Theta) $
  has a monodromy matrix with trace $\Theta$, and
  the separation distance $L_n$ of the two curved boundary
  segments has the asymptotic expression
  \begin{equation*}
    L_n( \bar{\mathsf{y}}_0 )
    -
    L_n(0)
    = 
    \frac{ 2 - \Theta }{ 2\,a_{20} \sin\hat{\omega}_1 }
    +
    \operatorname{\mathcal{O}}(n^{-1})
  \end{equation*}
  as $n \to \infty$, which is uniform in $-3 \leq \Theta \leq 3$.
\end{proposition}
\begin{proof}
  Fix $n \geq n_0$.
  Then for any $-3 \leq \Theta \leq 3$
  it follows from Lemma~\ref{lem_funcOmegaLocalPeriodicNTr}
  that the value
  $ \bar{\mathsf{y}}_0 = \bar{\Upsilon}_{n}(2-\Theta) $
  corresponds to a periodic orbit whose monodromy matrix has trace
  $\Theta$.
  By Lemma~\ref{lem_identity_Ln_y}
  \begin{align*}
    L_n( \bar{\Upsilon}_{n}(2-\Theta) )
    -
    L_n(0)
    &= 
    \frac{
    \Theta
    +
    \operatorname{tr} \mathrm{D}{\mathsf F}( f( \bar{\mathsf{y}}_0 ),
    \bar{\mathsf{y}}_0 )
    }{
    \sin\hat{\omega}_1
    }
    \,\frac{ \sin\bar{\mathsf{y}}_0
    }{
    \frac{ \partial{\mathsf F}_{ \mathsf{y} }
    }{ \partial \mathsf{x} }( f( \bar{\mathsf{y}}_0 ),
    \bar{\mathsf{y}}_0 )
    }
    \Big|_{
    \bar{\mathsf{y}}_0 = \bar{\Upsilon}_{n}(2-\Theta)
    }
    \\
    &\quad
    -
    \frac{
    2\,f(\bar{\mathsf{y}}_0)
    \,\cos( \hat{\omega}_1 - \bar{\mathsf{y}}_0 )
    }{
    \sin\hat{\omega}_1
    \,\sin( \hat{\omega}_1 - \bar{\mathsf{y}}_0 )
    }
    \,\sin\bar{\mathsf{y}}_0
    \Big|_{
    \bar{\mathsf{y}}_0 = \bar{\Upsilon}_{n}(2-\Theta)
    }
    \\
    &\quad
    +
    \frac{
    2\,f(\bar{\mathsf{y}}_0)
    }{
    \sin( \hat{\omega}_1 - \bar{\mathsf{y}}_0 )
    }
    \Big|_{
    \bar{\mathsf{y}}_0 = \bar{\Upsilon}_{n}(2-\Theta)
    }
  \end{align*}
  which yields with
  \eqref{eqn_expansion_trD_dFydx },
  \eqref{eqn_expansion_funcOmegaLocalPeriodicNTr}
  \begin{equation*}
    L_n( \bar{\Upsilon}_{n}(2-\Theta) )
    -
    L_n(0)
    = 
    \frac{ 2 - \Theta }{ 2\,a_{20} \sin\hat{\omega}_1 }
    +
    \operatorname{\mathcal{O}}(n^{-1})
  \end{equation*}
  uniformly in $-3 \leq \Theta \leq 3$ as $n \to \infty$.
\end{proof}

The result of Proposition~\ref{prop_elliptic_general}
shows in particular, that for all $n\geq n_0$
there exists elliptic periodic orbits on the table
with separation distance $L_n$ as long as
$L_n$
is between
$L_n(0)$
and
$
L_n(0)
+
\frac{2}{ a_{20} \sin\hat{\omega}_1 }
+
\operatorname{\mathcal{O}}(n^{-1})
$
as $n \to \infty$.
And since the values $( L_n(0) )_{n}$ are equidistant it follows that
there are elliptic periodic orbits on the smoothed out stadium
for all separation distances from an infinite sequence
of equidistantly spaced intervals of fixed length in $(0,\infty)$.

\subsection{Nonlinear stability analysis}

The results in this section will address the non-linear
stability of the elliptic periodic orbits that were
considered in Proposition~\ref{prop_elliptic_general}.
The main tool we use is the Birkhoff normal form.
We record the following abstract result:
\begin{lemma}[Normal form expansion and nonlinear stability
  -- \cite{MR1239173,MR1992664,MR2172708}]
  \label{lem_normal_form_expansion_NEW}
  Let
  $
  {\mathsf T}(\mathsf{x}, \mathsf{y})
  $ be an area-preserving $C^4$ mapping with an elliptic fixed point
  at the origin
  \begin{align*}
    {\mathsf T}(\mathsf{x}, \mathsf{y}) =
    \begin{pmatrix}
      A_{10}\,\mathsf{x} + A_{01}\,\mathsf{y}
      + A_{20}\,\mathsf{x}^2 + A_{11}\,\mathsf{x}\,\mathsf{y}
      + \ldots + A_{03}\,\mathsf{y}^3
      \\
      B_{10}\,\mathsf{x} + B_{01}\,\mathsf{y}
      + B_{20}\,\mathsf{x}^2 + B_{11}\,\mathsf{x}\,\mathsf{y}
      + \ldots + B_{03}\,\mathsf{y}^3
    \end{pmatrix}
    + \operatorname{\mathcal{O}}_4(\mathsf{x},\mathsf{y})
  \end{align*}
  and let $\lambda = e^{\pm i\phi}$
  denote the complex eigenvalues of
  $d{\mathsf T}(0,0)$, where the sign of $\phi$ is chosen such that
  $A_{01}\,\sin\phi > 0$.
  If $\lambda^2, \lambda^3,\lambda^4 \neq 1$, then there exists a
  real-analytic canonical change of coordinates
  $(\mathsf{x}, \mathsf{y}) \mapsto z \in {\mathbb C}$
  taking ${\mathsf T}$
  into its Birkhoff normal form
  $z \mapsto \lambda\,z\,e^{i\,A\,|z|^2} + \operatorname{\mathcal{O}}(|z|^4)$.
  The expression for the first Birkhoff coefficient $A$ reads
  \begin{align*}
    A &=
    \operatorname{Im} c_{21}
    + \frac{\sin\phi}{\cos\phi -1 }\,\big(
    3\,|c_{20}|^2 + \frac{2\,\cos\phi -1}{2\,\cos\phi + 1}\, |c_{02}|^2
    \big)
  \end{align*}
  where
  \begin{align*}
    8\,\operatorname{Im} c_{21}
    &=
    A_{10}\,\Big[
    - A_{21}
    + 3\,\frac{B_{10}\,A_{03}}{A_{01}}
    - 3\,\frac{A_{01}\,B_{30}}{B_{10}}
    + B_{12}
    \Big]
    \\
    &\qquad
    -B_{10}
    \,\Big[
    A_{12}
    - 3\,\frac{A_{01}\,A_{30}}{B_{10}}
    - \frac{A_{01}\,B_{21}}{B_{10}}
    + 3\,B_{03}
    \Big]
    \\
    16\,|c_{20}|^2
    &=
    \sqrt{-\frac{A_{01}}{B_{10}}}
    \,\Big[ \frac{B_{10}}{A_{01}}\,A_{02} + A_{20} + B_{11} \Big]^2
    +
    \sqrt{-\frac{B_{10}}{A_{01}}}
    \,\Big[ \frac{A_{01}}{B_{10}}\,B_{20} + B_{02} + A_{11} \Big]^2
    \\
    16\,|c_{02}|^2
    &=
    \sqrt{-\frac{A_{01}}{B_{10}}}
    \,\Big[ \frac{B_{10}}{A_{01}}\,A_{02} + A_{20} - B_{11} \Big]^2
    +
    \sqrt{-\frac{B_{10}}{A_{01}}}
    \,\Big[ \frac{A_{01}}{B_{10}}\,B_{20} + B_{02} - A_{11} \Big]^2
  \end{align*}
  are given in terms of the $A_{ij}$ and $B_{ij}$.
  Furthermore, if $A \neq 0$, then
  the elliptic fixed point at the origin is nonlinearly stable.
\end{lemma}

In order to proceed with the nonlinear stability analysis of the
first return map
${\mathsf T}_{ \bar{\mathsf{y}}_0 }$
we note that \eqref{eqn_def_monodromyMapYbar} and
Proposition~\ref{prop_elliptic_general} imply
\begin{subequations}
  \label{eqn_expansion_step1}
  \begin{equation}
    \begin{split}
      ( \delta\mathsf{x}_1, \delta\mathsf{y}_1 )
      &=
      \bar{{\mathsf F}}_{ \bar{\mathsf{y}}_0 }(
      \delta\mathsf{x}_0, \delta\mathsf{y}_0 )
      \;,\quad
      \bar{\mathsf{y}}_0 = \bar{\Upsilon}_{n}(2-\Theta)
      \\
      -\delta\mathsf{x}_2
      &=
      \delta\mathsf{x}_1
      +
      h( \Theta, n)
      \,\delta\mathsf{y}_1^2
      \,\Big[
      1
      +
      \operatorname{\mathcal{O}}( \delta\mathsf{y}_1^2 )
      \Big]
      -
      g( \Theta, n)
      \,\Big[
      \delta\mathsf{y}_1
      -
      \frac{1}{6}\,\delta\mathsf{y}_1^3
      +
      \operatorname{\mathcal{O}}( \delta\mathsf{y}_1^5 )
      \Big]
      \\
      -\delta\mathsf{y}_2
      &=
      \delta\mathsf{y}_1
    \end{split}
  \end{equation}
  where the $\operatorname{\mathcal{O}}$--terms are uniform, and we introduced the
  short-hand notations
  \begin{equation}
    \begin{split}
      h( \Theta, n)
      &=
      f( \bar{\Upsilon}_{n}(2-\Theta) )
      \\
      g( \Theta, n)
      &=
      2 \,h( \Theta, n)
      \cot( \hat{\omega}_1 -
      \bar{\Upsilon}_{n}(2-\Theta)
      )
      +
      \frac{ 2\,\Gamma_y(\hat{s}_1) + n\,W
      }{ \sin( \hat{\omega}_1 -
      \bar{\Upsilon}_{n}(2-\Theta)
      ) }
      \;.
    \end{split}
  \end{equation}
\end{subequations}
The next observation is that we can write
$\bar{{\mathsf F}}_{ \bar{\mathsf{y}}_0 }$
in terms of a power series as
\begin{subequations}
  \label{eqn_expansion_step2}
  \begin{equation}
    \bar{{\mathsf F}}_{ \bar{\mathsf{y}}_0 }(
    \delta\mathsf{x}_0, \delta\mathsf{y}_0 )
    =
    \begin{pmatrix}
      \bar{a}_{10}\,\delta\mathsf{x}_0
      +
      \bar{a}_{01}\,\delta\mathsf{y}_0
      +
      \bar{a}_{20}\,\delta\mathsf{x}_0^2
      +
      \ldots
      +
      \bar{a}_{03}\,\delta\mathsf{y}_0^3
      \\
      \bar{b}_{10}\,\delta\mathsf{x}_0
      +
      \bar{b}_{01}\,\delta\mathsf{y}_0
      +
      \bar{b}_{20}\,\delta\mathsf{x}_0^2
      +
      \ldots
      +
      \bar{b}_{03}\,\delta\mathsf{y}_0^3
    \end{pmatrix}
    +
    \operatorname{\mathcal{O}}_4(\delta\mathsf{x}_0,\delta\mathsf{y}_0)
    ,
  \end{equation}
  where
  \begin{equation}
    \begin{split}
      \bar{a}_{k l}(\Theta, n)
      &=
      \frac{1}{k!\,l!}
      \,\partial_{\mathsf{x}}^k \partial_{\mathsf{y}}^l
      {\mathsf F}_{\mathsf{x}}( f( \bar{\mathsf{y}}_0 ),
      \bar{\mathsf{y}}_0)
      \\
      \bar{b}_{k l}(\Theta, n)
      &=
      \frac{1}{k!\,l!}
      \,\partial_{\mathsf{x}}^k \partial_{\mathsf{y}}^l
      {\mathsf F}_{\mathsf{y}}( f( \bar{\mathsf{y}}_0 ),
      \bar{\mathsf{y}}_0)
    \end{split}
    \qquad\text{and}\quad
    \bar{\mathsf{y}}_0
    =
    \bar{\Upsilon}_{n}(2-\Theta)
    \;.
  \end{equation}
\end{subequations}
Combining
\eqref{eqn_expansion_step1}
and
\eqref{eqn_expansion_step2}
yields
\begin{equation}
  \label{eqn_expansion_step3}
  \begin{split}
    A_{10} &= -\bar{a}_{10}
    + g( \Theta, n)\,\bar{b}_{10}
    \\
    A_{01} &= -\bar{a}_{01}
    + g( \Theta, n)\,\bar{b}_{01}
    \\
    A_{20} &= -\bar{a}_{20}
    + g( \Theta, n)\,\bar{b}_{20}
    -
    h(\Theta, n)
    \,\bar{b}_{10}^2
    \\
    A_{11} &= -\bar{a}_{11}
    + g( \Theta, n)\,\bar{b}_{11}
    -
    2\,h(\Theta, n)
    \,\bar{b}_{10} \,\bar{b}_{01}
    \\
    A_{02} &= -\bar{a}_{02}
    + g( \Theta, n)\,\bar{b}_{02}
    -
    h(\Theta, n)
    \,\bar{b}_{01}^2
    \\
    A_{30} &= -\bar{a}_{30}
    + g( \Theta, n)\,\bar{b}_{30}
    -
    2\,h(\Theta, n)
    \,\bar{b}_{10}\,\bar{b}_{20}
    -
    \frac{1}{6}\, g( \Theta, n)
    \,\bar{b}_{10}^3
    \\
    A_{21} &= -\bar{a}_{21}
    + g( \Theta, n)\,\bar{b}_{21}
    -
    2\,h(\Theta, n)
    \,( \bar{b}_{10}\,\bar{b}_{11} + \bar{b}_{01}\,\bar{b}_{20} )
    -
    \frac{1}{2}\, g( \Theta, n)
    \,\bar{b}_{10}^2 \,\bar{b}_{01}
    \\
    A_{12} &= -\bar{a}_{12}
    + g( \Theta, n)\,\bar{b}_{12}
    -
    2\,h(\Theta, n)
    \,( \bar{b}_{10}\,\bar{b}_{02} + \bar{b}_{01}\,\bar{b}_{11} )
    -
    \frac{1}{2}\, g( \Theta, n)
    \,\bar{b}_{10} \,\bar{b}_{01}^2
    \\
    A_{03} &= -\bar{a}_{03}
    + g( \Theta, n)\,\bar{b}_{03}
    -
    2\,h(\Theta, n)
    \,\bar{b}_{01}\,\bar{b}_{02}
    -
    \frac{1}{6}\, g( \Theta, n)
    \,\bar{b}_{01}^3
    \\
    B_{kl}
    &=
    -\bar{b}_{kl}
  \end{split}
\end{equation}
for the coefficients
in the series expansion of the first return map
${\mathsf T}_{ \bar{\mathsf{y}}_0 }$
(using the same notation as in Lemma~\ref{lem_normal_form_expansion_NEW}).

With
\eqref{eqn_expansion_funcOmegaLocalPeriodicNTr}
and
\eqref{eqn_funcX_localExpansion}
the expressions for
$h( \Theta, n)$, $g( \Theta, n)$
defined in \eqref{eqn_expansion_step1}
take on the form
\begin{equation}
  \label{eqn_expansion_step1_asymptotic}
  \begin{split}
    n\,h( \Theta, n)
    &=
    \frac{ a_{01} }{W\,a_{20}}
    \,\frac{2-\Theta}{4} \,\sin\hat{\omega}_1
    +
    \operatorname{\mathcal{O}}(n^{-1})
    \\
    \frac{1}{n}\, g( \Theta, n)
    &=
    \frac{ W }{ \sin\hat{\omega}_1 }
    +
    \operatorname{\mathcal{O}}(n^{-1})
  \end{split}
\end{equation}
as $n\to\infty$, uniformly in $ -3 \leq \Theta \leq 3 $.
Similarly, the asymptotic forms of
$\bar{a}_{k l}(\Theta, n)$ and $\bar{b}_{k l}(\Theta, n)$,
as defined in \eqref{eqn_expansion_step2}, are given by
\begin{equation}
  \label{eqn_expansion_step2_asymptotic}
  \begin{split}
    \bar{a}_{k,l}(\Theta, n)
    &=
    a_{k,l}
    +
    \frac{ (k+1)\,a_{01} \,a_{k+1,l} + 2\,(l+1)\,a_{k,l+1} }{a_{20}}
    \,\frac{2-\Theta}{4}
    \,\frac{ \sin\hat{\omega}_1 }{ n\,W }
    +
    \operatorname{o}(n^{-1})
    \\
    \bar{b}_{k,l}(\Theta, n)
    &=
    b_{k,l}
    +
    \frac{ (k+1)\,a_{01} \,b_{k+1,l} + 2\,(l+1)\,b_{k,l+1} }{a_{20}}
    \,\frac{2-\Theta}{4}
    \,\frac{ \sin\hat{\omega}_1 }{ n\,W }
    +
    \operatorname{o}(n^{-1})
  \end{split}
\end{equation}
as $n\to\infty$, uniformly in $ -3 \leq \Theta \leq 3 $.

\begin{proposition}
  \label{prop_elliptic_general_BirkhoffCoeff}
  Assuming that $a_{20} \neq 0$ and
  $ a_{30} \neq - a_{20}^2$ (i.e. $b_{30} \neq 0$).
  Then the first Birkhoff coefficient $A$ of the periodic orbit
  correspnding to $\Theta$ and $n$ satisfies
  \begin{equation*}
    \frac{1}{n^2}\,A
    =
    - \frac{ 1 }{ 2-\Theta }
    \,\frac{ 3\,b_{30}\,W^2 }{ 8\,\sin^2\hat{\omega}_1 }
    +
    \operatorname{o}(1)
  \end{equation*}
  as $n\to\infty$, uniformly in $\Theta$ in any compact
  subset of $(-2, 2)\setminus\{ -1, 0 \}$.
\end{proposition}
\begin{proof}
  Using the expressions derived in
  \eqref{eqn_expansion_step3},
  \eqref{eqn_expansion_step1_asymptotic},
  \eqref{eqn_expansion_step2_asymptotic}
  as well as the symmetry relations given in
  Proposition~\ref{prop_reducedMonodromyMap_coeffsRelations}
  we obtain
  \begin{align*}
    B_{k,l}
    &=
    -b_{k,l}
    -
    \frac{ (k+1)\,a_{01} \,b_{k+1,l} + 2\,(l+1)\,b_{k,l+1} }{a_{20}}
    \,\frac{2-\Theta}{4}
    \,\frac{ \sin\hat{\omega}_1 }{ n\,W }
    +
    \operatorname{o}(n^{-1})
    \\
    \frac{1}{n}\,A_{k,l} &= \operatorname{\mathcal{O}}(1)
    \;,\quad
    A_{10}
    =
    \Theta - 1
    +
    \operatorname{o}(1)
    \;,\quad
    A_{20}
    =
    - \frac{ 4 -\Theta }{2} \,a_{20} + \operatorname{o}(1)
    \\
    \frac{1}{n^2}\,A_{01}\,A_{30}
    &=
    -b_{30} \,\frac{ W^2 }{ \sin^2\hat{\omega}_1 }
    +
    \operatorname{o}(1)
  \end{align*}
  uniformly in $-3 \leq \Theta \leq 3$, and also
  \begin{equation*}
    -\frac{A_{01}}{B_{10}}
    =
    n^2
    \,\Big[
    \frac{ W^2 }{ (2-\Theta) \,\sin^2\hat{\omega}_1 }
    +
    \operatorname{o}(1)
    \Big]
    \;.
  \end{equation*}
  Using the same notation as in Lemma~\ref{lem_normal_form_expansion_NEW}
  we therefore obtain
  \begin{equation*}
    \frac{1}{n^2}\,8\,\operatorname{Im} c_{21}
    =
    - \frac{ 1 }{ 2-\Theta }
    \,\frac{ 3\,b_{30}\,W^2 }{ \sin^2\hat{\omega}_1 }
    +
    \operatorname{o}(1)
    \;,\quad
    \frac{1}{n}\,16\,|c_{20}|^2
    =
    \operatorname{\mathcal{O}}(1)
    \;,\quad
    \frac{1}{n}\,16\,|c_{02}|^2
    =
    \operatorname{\mathcal{O}}(1)
    \;.
  \end{equation*}
  In particular, this readily implies the claimed expression for
  the Birkhoff coefficient, which finishes the proof.
\end{proof}

With the results of Proposition~\ref{prop_elliptic_general} and
Proposition~\ref{prop_elliptic_general_BirkhoffCoeff}
and the general stability result Lemma~\ref{lem_normal_form_expansion_NEW}
we immediately obtain the existence of nonlinearly stabile
periodic orbits for large set of separation distances of the two
curved boundary components. A summary of our results so far is
given in the following Theorem~\ref{thm_elliptic_general}:
\begin{theorem}[Range of existence of stable symmetric periodic orbits]
  \label{thm_elliptic_general}
  Suppose that $\Gamma$ satisfies
  Assumption~\ref{assumption_def_parallelInParallelOutSegment}, and
  suppose that the corresponding part of a billiard trajectory
  satisfies $a_{20} \neq 0$ and
  $ a_{30} \neq - a_{20}^2$ (i.e. $b_{30} \neq 0$).
  For every $0 < \epsilon < \frac{1}{2}$
  there exists $N_\epsilon \geq n_0$ such that
  on any billiard table with separation distance $L$ in the set
  \begin{equation*}
    \bigcup_{n \geq N_\epsilon}
    \Big\{
    L_0(0)
    +
    \frac{ n\,W }{ \tan\hat{\omega}_1 }
    +
    \frac{1}{ 2\,a_{20} \sin\hat{\omega}_1 }\,I_\epsilon
    \Big\}
    \;,\quad
    I_\epsilon
    =
    [\epsilon, 2-\epsilon]
    \cup
    [ 2+\epsilon, 3-\epsilon]
    \cup
    [ 3+\epsilon, 4 - \epsilon ]
  \end{equation*}
  there exist nonlinearly stable periodic orbits
  in the sense that their first Birkhoff coefficient is
  nonzero.
\end{theorem}
\begin{proof}
  Proposition~\ref{prop_elliptic_general} provides an asymptotic expression
  for the range of separation distances on which ellptic periodic orbits
  exit. In particular, for every $0 < \epsilon < \frac{1}{2}$
  there exists $N_\epsilon \geq n_0$ such that
  there exist elliptic periodic orbits on the billiard table
  with separation distance $L$ in the set
  \begin{equation*}
    \bigcup_{n \geq N_\epsilon}
    \Big\{
    L_0(0)
    +
    \frac{ n\,W }{ \tan\hat{\omega}_1 }
    +
    \frac{2-\Theta}{ 2\,a_{20} \sin\hat{\omega}_1 }
    \operatorname{:}
    -2 +\epsilon \leq \Theta \leq 2-\epsilon
    \Big\}
    \;,
  \end{equation*}
  where we used Lemma~\ref{lem_def_Ln_ybar0}
  to express $L_n(0)$ explicitly in terms of $n$.

  By Proposition~\ref{prop_elliptic_general_BirkhoffCoeff} it
  then follows that these periodic orbits are also nonlinearly
  stable in the sense that their first Birkhoff coefficient is
  nonzero, provided that $N_\epsilon$ is large enough, and
  that the range of $\Theta$ excludes lower order resonances.
  The latter can be achieved by using
  \begin{equation*}
    \Theta \in
    [-2 + \epsilon, -1-\epsilon] \cup [ -1+\epsilon, -\epsilon]
    \cup [\epsilon, 2-\epsilon]
  \end{equation*}
  instead of
  $ -2 +\epsilon \leq \Theta \leq 2-\epsilon $, again assuming that
  the value of $N_\epsilon$ is large enough.
  This finishes the proof.
\end{proof}

Theorem~\ref{thm_elliptic_general} shows the billiard table
has nonlinearly stable periodic orbits for all separation distances
taken from a set that is formed by a sequence of equidistantly spaced
copies of the three closed intervals forming the set
$\frac{1}{ 2\,a_{20} \sin\hat{\omega}_1 }\,I_\epsilon$.
The restriction on the separation distance imposed by
using $I_\epsilon$ instead of $[\epsilon, 4-\epsilon]$
in Theorem~\ref{thm_elliptic_general} is a technical condition to
guarantee the nonlinear stability in the sense of the nonvanishing
of the first Birkohoff coefficient $A$.
This condition is most likely not optimal, because
it was shown in Proposition~\ref{prop_elliptic_general_BirkhoffCoeff}
that $|A|$ actually grows in $n$ for
fixed values of $\Theta$. Therefore, to guarantee
that, say,  $|A| \geq 1$, we can probably allow for the
range of $\Theta$
to approach the resonance values $-1$,$0$ as $n$ increases.
However, since any condition to guarantee nonlinear stability
is of technical nature we have not tried to optimize this part
of Theorem~\ref{thm_elliptic_general}.

Finally we remind the reader that the assertion of
Theorem~\ref{thm_elliptic_general} is conditional to $\Gamma$
satisfying Assumption~\ref{assumption_def_parallelInParallelOutSegment},
in other words, it is conditional to the existence of particular
billiard trajectory piece. The following
Section~\ref{sect_largeSeparation_verifyAssumption}
addresses the question of validity of
Assumption~\ref{assumption_def_parallelInParallelOutSegment}.

\subsection{Verification of
Assumption~\ref{assumption_def_parallelInParallelOutSegment}
and an explicit stability criterion}
\label{sect_largeSeparation_verifyAssumption}

In this section we address the validity of
Assumption~\ref{assumption_def_parallelInParallelOutSegment}.
In view of Fig.~\ref{fig_long_orbit_curved_segment} the simplest
part of a billiard trajectory required for $\Gamma$
to satisfy Assumption~\ref{assumption_def_parallelInParallelOutSegment}
has $3$ reflections off of $\Gamma$. Namely, one at its
center, and two more near the end points.
In other words, the existence of such
a billiard trajectory is a sufficient criterion for
$\Gamma$ to satisfy
Assumption~\ref{assumption_def_parallelInParallelOutSegment}.

For definiteness, let
$s_0$,
$s_1$,
$s_2$ denote the arc length parameters of the three
points of reflection, and let
$\varphi_0$,
$\varphi_1$,
$\varphi_2$
denote the corresponding angles of reflection.
By symmetry $\varphi_0 = \varphi_1$, and
$s_1 = \frac{1}{2}\,|\Gamma|$.
Furthermore, by the two free paths $\tau_{01}$ and
$\tau_{12}$ corresponding to the three points of reflection
satisfy $\tau_{01} = \tau_{12}$, again by the assumed
symmetry of the trajectory piece.

In particular, the linearization of the billiard flow from the
pre-collisional state at $s_0$ to the post-collisional
state at $s_2$ is given by
\begin{equation*}
  \begin{pmatrix}
    a & b \\
    c & d
  \end{pmatrix}
  =
  -
  \begin{pmatrix}
    1 & 0 \\
    {\mathcal R}_0 & 1
  \end{pmatrix}
  \begin{pmatrix}
    1 & \tau_{01} \\
    0 & 1
  \end{pmatrix}
  \begin{pmatrix}
    1 & 0 \\
    {\mathcal R}_1 & 1
  \end{pmatrix}
  \begin{pmatrix}
    1 & \tau_{01} \\
    0 & 1
  \end{pmatrix}
  \begin{pmatrix}
    1 & 0 \\
    {\mathcal R}_0 & 1
  \end{pmatrix}
  \;.
\end{equation*}
Evaluating yields
\begin{align*}
  a=d
  &=
  -1
  -
  \tau_{01}
  \,({\mathcal R}_1 + 2\,{\mathcal R}_0 + \tau_{01}\,{\mathcal R}_0\,{\mathcal R}_1)
  \\
  &=
  1
  -
  (1 + \tau_{01}\,{\mathcal R}_0)
  \,(2 + \tau_{01}\,{\mathcal R}_1)
  \\
  b
  &=
  -
  \tau_{01}\,(2 + \tau_{01}\,{\mathcal R}_1)
  \\
  c
  &=
  -
  (1 + \tau_{01}\,{\mathcal R}_0)
  \,({\mathcal R}_1 + 2\,{\mathcal R}_0 + \tau_{01}\,{\mathcal R}_0\,{\mathcal R}_1)
  \;.
\end{align*}
In particular we obtain that
\begin{equation*}
  a=-1 \text{ and } c = 0
  \iff
  {\mathcal R}_0 = -\frac{{\mathcal R}_1}{2 + \tau_{01}\,{\mathcal R}_1}
  \;.
\end{equation*}
That is to say that for
Assumption~\ref{assumption_def_parallelInParallelOutSegment}
to be satisfied for this particular type of trajectory piece
the relation
$ {\mathcal R}_0 = -\frac{{\mathcal R}_1}{2 + \tau_{01}\,{\mathcal R}_1} $
must hold.
\begin{lemma}
  \label{lem_threeLink_step1}
  For any $\Gamma$ there exists
  $s_0 \in (0, s_1) $
  such that
  $ {\mathcal R}_0 = -\frac{{\mathcal R}_1}{2 + \tau_{01}\,{\mathcal R}_1} $.
\end{lemma}
\begin{proof}
  Varying $s_0$ on
  $(0, s_1)$
  makes
  $
  q(s_0)
  =
  \tau_{01}\,{\mathcal R}_0 
  +
  \frac{\tau_{01}\,{\mathcal R}_1}{2 + \tau_{01}\,{\mathcal R}_1}
  $
  a well-defined continuous function.
  Clearly
  $
  \lim_{s_0 \to 0} q(s_0)
  =
  \infty
  $.
  And since for a circle we have that
  $\tau\,{\mathcal R} = -4$ it follows that
  $
  \lim_{s_0 \to s_1} q(s_0)
  =
  -2
  $.
  By continuity of $q$ it thus follows that there must be a
  $s_0 \in (0, s_1) $ for which
  $q(s_0) = 0$, which finishes the proof.
\end{proof}

While the existence of $s_0$ as in
Lemma~\ref{lem_threeLink_step1} is necessary for this special
part of a billiard trajectory to be as required by
Assumption~\ref{assumption_def_parallelInParallelOutSegment},
it is not sufficient. This is due to the fact that
Lemma~\ref{lem_threeLink_step1} does not necessarily guarantee that
there are no other reflections off of $\Gamma$
prior to the reflection at $s_0$.

In order to proceed
we need to recall the notion of absolute focusing, which was introduced
in \cite{MR1133266,MR1179172}.
A finite piece of a billiard trajectory is called absolutely focused, if
an initially parallel beam, when sent along this billiard trajectory, is
focused right after each reflection, and has a conjugate point in between
any two consecutive reflections. Furthermore, a focusing boundary component
is called absolutely focusing, if every complete sequence of reflections
off it is absolutely focused. In particular, arcs of a circle are absolutely
focusing, and so are sufficiently short segments of an arbitrary focusing
boundary component, as well as any subsegments of absolutely focusing
components.

The key observation now is that if the three reflections
$s_0$, $s_1$, $s_2$
as in Lemma~\ref{lem_threeLink_step1} were all on an
absolutely focusing subsegment of $\Gamma$, then
an initially parallel beam entering right before the first reflection
at $s_0$ would have to be focusing right after the
reflection at $s_2$, \cite{MR1133266,MR1179172}.
\begin{figure}[ht!]
  \centering
  \includegraphics[height=4cm]{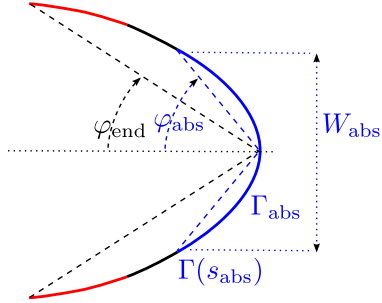}
  \caption{
  An illustration of the notation used in
  Lemma~\ref{lem_boundaryCondition_largeSep}.
  The symmetric absolutely focusing subsegment
  $\Gamma_{\mathrm{abs}}$ of $\Gamma$
  is shown in blue, with the corresponding arc length parameter
  $s_{\mathrm{abs}}$, angle
  $\varphi_{\mathrm{abs}}$, and height
  $W_{\mathrm{abs}}$. The angle corresponding to the
  end point of $\Gamma$ is $\varphi_{\mathrm{end}}$.
  }
  \label{fig_three_orbit}
\end{figure}
However, by Lemma~\ref{lem_threeLink_step1} it is parallel, and hence
$s_0$, $s_1$, $s_2$
cannot be on an absolutely focusing subsegment of $\Gamma$.
This observation now implies our main criterion, given in
Lemma~\ref{lem_boundaryCondition_largeSep} below,
to verify
Assumption~\ref{assumption_def_parallelInParallelOutSegment}.

\begin{lemma}
  \label{lem_boundaryCondition_largeSep}
  Suppose that $\Gamma_{\mathrm{abs}}$ is
  an absolutely focusing subsegment of $\Gamma$, which
  is symmetric about the horizontal axis. With the notation as
  shown in Fig.~\ref{fig_three_orbit}, suppose that
  the angle enclosed by
  $\Gamma_{\mathrm{abs}}$ is at least as large as
  $\pi - \varphi_{\mathrm{end}}$, and
  that
  $
  W_{\mathrm{abs}}
  \geq
  \frac{ |\Gamma| - |\Gamma_{\mathrm{abs}}| }{2}
  \,\tan \varphi_{\mathrm{abs}}
  $.
  Then $\Gamma$ satisfies
  Assumption~\ref{assumption_def_parallelInParallelOutSegment}.
\end{lemma}
\begin{proof}
  There are two types of additional reflections possible. The
  first kind are reflections off of $\Gamma$
  on the segment connecting
  $\Gamma(0)$ and $\Gamma(s_0)$.
  The second kind are reflections off of $\Gamma$ on
  the segment connecting
  $\Gamma(|\Gamma|)$
  and $\Gamma(s_2)$.

  A sufficient condition to avoid additional reflections of the first
  kind is
  $ 0 \leq 2\pi - \omega_0 \leq \frac{\pi}{2} $,
  where $\omega_0$ denotes the direction angle corresponding to the
  pre-collisional velocity at $s_0$.
  Due to the geometry
  \begin{equation*}
    2\pi - \omega_0
    =
    \varphi_1 - 2\,\theta(s_0)
    \,,\quad
    \frac{d\varphi_1}{d s_0} > 0
    \,,\quad
    \frac{d \theta(s_0) }{d s_0} > 0
    \,,\quad
    0 \leq \varphi_1, \theta(s_0) \leq \frac{\pi}{2}
    \;,
  \end{equation*}
  and hence
  \begin{equation*}
    \varphi_{\mathrm{end}}
    -
    2\,\theta(s_{\mathrm{abs}})
    \leq
    2\pi - \omega_0
    \leq
    \varphi_1
    \leq
    \frac{\pi}{2}
    \;.
  \end{equation*}
  And since the angle enclosed by the
  $\Gamma_{\mathrm{abs}}$
  is
  $\pi - 2\,\theta(s_{\mathrm{abs}})$
  it thus follows from our assumption that
  $ 0 \leq 2\pi - \omega_0 \leq \frac{\pi}{2}$
  holds.

  It remains to rule out additional reflections of the second kind.
  Since $\Gamma$ is convex, and we already verified that
  $ 0 \leq 2\pi - \omega_0 \leq \frac{\pi}{2}$
  it follows that
  \begin{equation*}
    \Gamma(s_0)
    =
    \begin{pmatrix}
      1 & 0 \\
      0 & -1
    \end{pmatrix}
    \Gamma(s_{\mathrm{abs}})
    -
    \begin{pmatrix}
      q \\
      0
    \end{pmatrix}
    +
    t
    \begin{pmatrix}
      \cos\omega_0 \\
      \sin\omega_0
    \end{pmatrix}
    \quad\text{for some}\quad
    q \geq s_{\mathrm{abs}}
  \end{equation*}
  is sufficient to rule out additional reflections of the second kind.
  Eliminating $t$ yields
  \begin{equation*}
    q
    =
    [
    \Gamma(s_{\mathrm{abs}})
    -
    \Gamma(s_0)
    ]
    \cdot
    \begin{pmatrix}
      1 \\
      \frac{1}{-\tan\omega_0}
    \end{pmatrix}
    +
    \frac{ W_{\mathrm{abs}} }{-\tan\omega_0}
    \quad\text{and}\quad
    q \geq s_{\mathrm{abs}}
    \;,
  \end{equation*}
  where we used the fact that
  $
  W_{\mathrm{abs}} = -2\, \Gamma_y(s_{\mathrm{abs}})
  $.
  Since $\tan\omega_0 < 0$ the first term of the right-and-side is
  clearly positive, hence
  $q \geq \frac{ W_{\mathrm{abs}} }{-\tan\omega_0}$.
  Furthermore, using again the relation
  $
  2\pi - \omega_0
  =
  \varphi_1 - 2\,\theta(s_0)
  $
  yields
  $
  -\tan\omega_0
  =
  \tan( \varphi_1 - 2\,\theta(s_0) )
  $.
  The monotonicity property implies
  $
  \varphi_1 - 2\,\theta(s_0)
  \leq
  \varphi_1
  \leq
  \varphi_{\mathrm{abs}}
  $, so that
  $q \geq \frac{ W_{\mathrm{abs}} }{ \tan \varphi_{\mathrm{abs}} }$.
  Therefore, a sufficient condition for
  $ q \geq s_{\mathrm{abs}}$ is to require
  $
  W_{\mathrm{abs}}
  \geq
  s_{\mathrm{abs}} \,\tan \varphi_{\mathrm{abs}}
  $.
  Since
  $s_{\mathrm{abs}}
  =
  \frac{ |\Gamma| - |\Gamma_{\mathrm{abs}}| }{2}
  $
  this finishes the proof.
\end{proof}

Because we are particularly interested in smoothening a semi-circle
of some radius $R$ in a small neighborhood of its endpoints,
we consider again the same construction as in
Fig.~\ref{fig_four-periodic}.
The following
Corollary~\ref{cor_boundaryCondition_largeSep_circleSmooth}
shows that if $\Gamma$ is a semi-circle except for
a sufficiently small segment near its endpoints, then
it satisfies Assumption~\ref{assumption_def_parallelInParallelOutSegment}.
\begin{corollary}
  \label{cor_boundaryCondition_largeSep_circleSmooth}
  Every local smoothing $\Gamma$ of a semi-circle
  $\Gamma_{\mathrm{c}}$
  as shown in
  Fig.~\ref{fig_four-periodic}
  with
  \begin{equation*}
    2\,\alpha \leq \varphi_{\mathrm{end}}
    \quad\text{and}\quad
    \frac{ |\Gamma| }{
    |\Gamma_\mathrm{c}|
    }
    \leq
    1
    +
    \frac{4}{\pi}
    \,\frac{\cos\alpha }{ \tan( \frac{\pi}{4} + \frac{\alpha}{2}) }
    -
    \frac{2\,\alpha}{\pi}
  \end{equation*}
  satisfies Assumption~\ref{assumption_def_parallelInParallelOutSegment}.
\end{corollary}
\begin{proof}
  For any $\alpha$ we can take $\Gamma_{\mathrm{abs}}$ to
  be the remaining circular segment of radius $R$. In particular,
  $\Gamma_{\mathrm{abs}}$ encloses an angle of
  $\pi - 2\,\alpha$, and
  $\varphi_{\mathrm{abs}} = \frac{\pi}{4} + \frac{\alpha}{2}$,
  $|\Gamma_{\mathrm{abs}}| = (\pi - 2\,\alpha)\,R$,
  $W_{\mathrm{abs}} = 2\,R\,\cos\alpha$.

  Thus our assumption
  $2\,\alpha \leq \varphi_{\mathrm{end}}$
  implies the assumption of
  Lemma~\ref{lem_boundaryCondition_largeSep}
  that the angle enclosed by
  $\Gamma_{\mathrm{abs}}$ is at least as large as
  $\pi - \varphi_{\mathrm{end}}$.
  Furthermore, the second condition of
  Lemma~\ref{lem_boundaryCondition_largeSep}, i.e.
  $
  W_{\mathrm{abs}}
  \geq
  \frac{ |\Gamma| - |\Gamma_{\mathrm{abs}}| }{2}
  \,\tan \varphi_{\mathrm{abs}}
  $
  takes on the explicit form
  \begin{equation*}
    2\,R\,\cos\alpha
    \geq
    \frac{
    |\Gamma|
    -
    (\pi - 2\,\alpha)\,R
    }{2}
    \,\tan( \frac{\pi}{4} + \frac{\alpha}{2})
    \;.
  \end{equation*}
\end{proof}

So far we have proved a sufficient condition for $\Gamma$
to satisfy
Assumption~\ref{assumption_def_parallelInParallelOutSegment}
by showing the existence of a particular
type of billiard trajectory as required by
Assumption~\ref{assumption_def_parallelInParallelOutSegment}.
In the following
Theorem~\ref{thm_elliptic_threeOrbit}
makes the stability result of
Theorem~\ref{thm_elliptic_general}
explicit for this particular setting, were we use the notation
${\mathcal R}^{(1)}$, ${\mathcal R}^{(2)}$ introduced in
\eqref{eqn_def_Rprime}
in Section~\ref{sect_billiard_basicFacts}.

\begin{theorem}[Range of existence of special stable symmetric periodic orbits]
  \label{thm_elliptic_threeOrbit}
  Suppose that $\Gamma$ is as in
  Lemma~\ref{lem_boundaryCondition_largeSep}, and suppose further
  that the two non-degeneracy conditions
  \begin{equation*}
    {\mathcal R}^{(1)}_0
    \neq
    -
    \frac{1}{ \tau_{01}\,{\mathcal R}_1 }
    \,\Big[
    \frac{ {\mathcal R}_1 }{ 2 + {\mathcal R}_1\,\tau_{01} }
    \Big]^2
    \,[
    3\,\tau_{01}\,{\mathcal R}_1\,\tan\varphi_0
    +
    4\,\tan\varphi_1
    +
    4\,\tan\varphi_0
    ]
  \end{equation*}
  and
  \begin{align*}
    {\mathcal R}^{(2)}_0
    &\neq
    \frac{1}{ 2\,[ 2 + \tau_{01}\,{\mathcal R}_1 ]^4 }
    \,\Big[
    -
    16\,{\mathcal R}^{(2)}_1
    +
    3\,\tau_{01}
    \,[{\mathcal R}^{(1)}_0]^2
    \,[ 2 + \tau_{01}\,{\mathcal R}_1 ]^5
    \\
    &\quad\qquad
    +
    2\,{\mathcal R}^{(1)}_0 \,{\mathcal R}_1
    \,[ 2 + \tau_{01} \,{\mathcal R}_1 ]^3
    \,[
    (14+9\,\tau_{01}\,{\mathcal R}_1) \,\tan\varphi_0
    +
    12 \,\tan\varphi_1
    ]
    \\
    &\quad\qquad
    +
    3\,{\mathcal R}_1^3
    \,[
    (
    3\,\tan\varphi_0\,( 2 + \tau_{01} \,{\mathcal R}_1 )
    +
    4\,\tan\varphi_1
    )^2
    +
    2\,(2 + \tau_{01} \,{\mathcal R}_1 )
    ]
    \Big]
  \end{align*}
  are satisfied.
  Then, for every $0 < \epsilon < \frac{1}{2}$
  there exists $N_\epsilon \geq n_0$ such that
  on any billiard table with separation distance $L$ in the set
  \begin{equation*}
    \bigcup_{n \geq N_\epsilon}
    \Big\{
    L_0(0)
    +
    \frac{ n\,W }{
    \tan( \varphi_1 - 2\,\theta(s_0) )
    }
    +
    \frac{1}{
    2\,\kappa\,\sin( \varphi_1 - 2\,\theta(s_0) )
    }\,I_\epsilon
    \Big\}
    \;,
  \end{equation*}
  where
  \begin{align*}
    I_\epsilon
    &=
    [\epsilon, 2-\epsilon]
    \cup
    [ 2+\epsilon, 3-\epsilon]
    \cup
    [ 3+\epsilon, 4 - \epsilon ]
    \\
    \kappa
    &=
    \frac{1}{4}\,\tau_{01}\,{\mathcal R}^{(1)}_0
    \,[ 2 + {\mathcal R}_1\,\tau_{01} ]
    +
    \frac{1}{4}
    \,\frac{ {\mathcal R}_1 }{ 2 + {\mathcal R}_1\,\tau_{01} }
    \,[
    3\,\tau_{01}\,{\mathcal R}_1\,\tan\varphi_0
    +
    4\,\tan\varphi_1
    +
    4\,\tan\varphi_0
    ]
  \end{align*}
  there exist nonlinearly stable periodic orbits
  in the sense that their first Birkhoff coefficient is
  nonzero.
\end{theorem}
\begin{proof}
  The billiard trajectory considered in
  Lemma~\ref{lem_boundaryCondition_largeSep}
  satisfies
  $ {\mathcal R}_0 = - \frac{ {\mathcal R}_1 }{ 2 + \tau_{01}\,{\mathcal R}_1 } $.
  Furthermore, the symmetry about the horizontal axis
  of both the billiard trajectry and of $\Gamma$
  imply
  \begin{equation*}
    \tau_{12}
    =
    \tau_{01}
    \;,\quad
    \varphi_2 = \varphi_0
    \;,\quad
    {\mathcal R}_2 = {\mathcal R}_0
    \;,\quad
    {\mathcal R}^{(1)}_2 = -{\mathcal R}^{(1)}_0
    \;,\quad
    {\mathcal R}^{(2)}_2 = {\mathcal R}^{(2)}_0
    \;,
  \end{equation*}
  and
  \begin{equation*}
    {\mathcal K}'(s_1) = 0
    \quad\text{and hence}\quad {\mathcal R}^{(1)}_1 = 0
    \;.
  \end{equation*}

  In order to determine the nonlinear stability of the corresponding
  periodic orbits using the result of Theorem~\ref{thm_elliptic_general}
  we need to compute $a_{20}$ and $b_{30}$.
  The expansion of the free flight map and of the reflection map
  are provided in
  Lemma~\ref{lem_billiard_localFlowCoordinates_localExpansion}.
  Thus, composing these expansion along the special billiard trajectory
  segment we consider we obtain after a
  straightforward but lengthy computation
  that
  \begin{align*}
    a_{01}
    &=
    - \tau_{01}\, [ 2 + \tau_{01}\,{\mathcal R}_1 ]
    \;,
    \\
    a_{20}
    &=
    -
    \frac{1}{4}\,\tau_{01}\,{\mathcal R}^{(1)}_0
    \,[ 2 + {\mathcal R}_1\,\tau_{01} ]
    -
    \frac{1}{4}
    \,\frac{ {\mathcal R}_1 }{ 2 + {\mathcal R}_1\,\tau_{01} }
    \,[
    3\,\tau_{01}\,{\mathcal R}_1\,\tan\varphi_0
    +
    4\,\tan\varphi_1
    +
    4\,\tan\varphi_0
    ]
    \;,
    \\
    b_{30}
    &=
    \frac{1}{ 24\,( 2 + \tau_{01}\,{\mathcal R}_1 )^4 }
    \,\Big[
    -
    16\,{\mathcal R}^{(2)}_1
    -
    2\,{\mathcal R}^{(2)}_0
    \,[ 2 + \tau_{01}\,{\mathcal R}_1 ]^4
    +
    3\,\tau_{01}
    \,[{\mathcal R}^{(1)}_0]^2
    \,[ 2 + \tau_{01}\,{\mathcal R}_1 ]^5
    \\
    &\quad\qquad
    +
    2\,{\mathcal R}^{(1)}_0 \,{\mathcal R}_1
    \,[ 2 + \tau_{01} \,{\mathcal R}_1 ]^3
    \,[
    (14+9\,\tau_{01}\,{\mathcal R}_1) \,\tan\varphi_0
    +
    12 \,\tan\varphi_1
    ]
    \\
    &\quad\qquad
    +
    3\,{\mathcal R}_1^3
    \,[
    (
    3\,\tan\varphi_0\,( 2 + \tau_{01} \,{\mathcal R}_1 )
    +
    4\,\tan\varphi_1
    )^2
    +
    2\,(2 + \tau_{01} \,{\mathcal R}_1 )
    ]
    \Big]
    \;,
  \end{align*}
  where we eliminated ${\mathcal R}_0$ using the relation
  $ {\mathcal R}_0 = - \frac{ {\mathcal R}_1 }{ 2 + \tau_{01}\,{\mathcal R}_1 } $.
  In particular, $a_{20} \neq 0$ and $b_{30}\neq 0$
  if and only if the two stated assumptions on
  ${\mathcal R}^{(1)}_0$ and ${\mathcal R}^{(2)}_0$ hold,
  respectively.
  With Theorem~\ref{thm_elliptic_general} the claimed existence
  and stability result follows when noting that
  \begin{equation*}
    \varphi_0
    =
    \frac{\pi}{2} - \varphi_1 + \theta(s_0)
    \;,\quad
    \hat{\omega}_1
    =
    3\pi - \omega_0
    =
    \pi + \varphi_1 - 2\,\theta(s_0)
  \end{equation*}
  due to the specific geometry of the billiard trajectory piece
  we consider, hence
  $
  \sin\hat{\omega}_1
  =
  - \sin( \varphi_1 - 2\,\theta(s_0) )
  $, and
  $
  \tan\hat{\omega}_1
  =
  \tan( \varphi_1 - 2\,\theta(s_0) )
  $.
\end{proof}

\section{Proofs of
Theorem~\ref{thm_large_separations} and
Theorem~\ref{thm_all_separations}}

Now we are ready to proof our final main results
Theorem~\ref{thm_large_separations} and
Theorem~\ref{thm_all_separations}.
\begin{proof}[Theorem~\ref{thm_large_separations}]
  By Lemma~\ref{lem_boundaryCondition_largeSep}
  and specifically its
  Corollary~\ref{cor_boundaryCondition_largeSep_circleSmooth},
  any sufficiently short smoothening of a semi-circle
  satisfies the assumptions of
  of Theorem~\ref{thm_elliptic_threeOrbit}.
  Therefore, claim
  (\ref{thm_large_separations_intervals}) of
  Theorem~\ref{thm_large_separations} follows.

  Clearly, if necessary the smoothening can be (locally) modified such that the
  point $\Gamma(s_0)$,
  the corresponding tangent direction
  $\theta(s_0)$,
  the curvature ${\mathcal K}(s_0)$,
  the length of the smoothed region,
  and the endpoint of $\Gamma$
  remain unaltered, while
  the first and second derivatives of the
  curvature
  ${\mathcal K}'(s_0)$,
  ${\mathcal K}''(s_0)$
  can be adjusted to take on any prescribed value.

  This allows to make sure that the conditions of
  Theorem~\ref{thm_elliptic_threeOrbit}
  for the existence of nonlinearly stable periodic orbits
  are satisfied for arbitrarily small values of $\kappa$.
  For $\kappa$ sufficiently small
  the intervals for $L$ for which existence of nonlinearly stable periodic
  orbits is guaranteed overlap.
  Therefore, there exist nonlinearly stable periodic
  orbits for all sufficiently large separations. Finally, note
  that the nonlinear stability persists for $C^5$ perturbation
  of such a boundary component, because of the nonvanishing of
  the first Birkohoff coefficient is an open condition.
  Thus we obtain a proof of
  item (\ref{thm_large_separations_nonlinearStability}) of
  Theorem~\ref{thm_large_separations}.
\end{proof}

\begin{proof}[Theorem~\ref{thm_all_separations}]
  Move a horizontal trajectory segment from the joint point of the circular
  \begin{figure}[ht!]
    \centering
    \includegraphics[width=8cm]{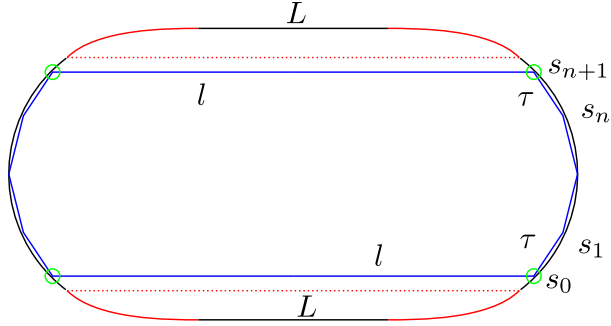}
    \caption{Illustration of the additional perturbation of a circle.}
    \label{fig_circle}
  \end{figure}
  part and the smoothing towards the interior of the circle such that
  a periodic orbit with only reflections off the circular part
  of $\Gamma$ forms, as shown in Fig.~\ref{fig_circle}, 
  This results in a billiard trajectory that falls onto the circular
  boundary component along the horizontal direction and is symmetric
  about the horizontal symmetry axis of the boundary component.
  Let the number of reflections off the boundary component be
  equal to $n+2$, and denote points of reflection by
  $s_0, s_1, \ldots, s_{n+1}$,
  see Fig.~\ref{fig_circle}.

  For reflections off the circle we have
  $\tau\,{\mathcal R} = -4$. Now we perturb the boundary component
  in a symmetric manner in a small neighborhood of $s_0$ and
  $s_{n+1}$ such that the location and the tangent of
  the boundary at this point remains unchanged. Denote the curvature
  at these two points by ${\mathcal K}_0$, so that
  \begin{equation*}
    \frac{{\mathcal R}_0}{{\mathcal R}}
    =
    \frac{{\mathcal K}_0}{{\mathcal K}}
    =
    \frac{{\rho}}{{\rho}_0}
    \;,\qquad
    {\mathcal R}_0
    =
    -\frac{4}{\tau}
    \,\frac{{\rho}}{{\rho}_0}
    \;.
  \end{equation*}

  Since the beam falls horizontally onto the boundary component and
  is symmetric about the horizontal we obtain a periodic orbit
  for any separation distance $L$. Denote the length of the free path
  between the two boundary components by $l$, see Fig.~\ref{fig_circle}.
  Then the linearization along the periodic orbit is given by
  $M^2$, where
  \begin{align*}
    M
    =
    \begin{pmatrix}
      1 & l \\
      0 & 1
    \end{pmatrix}
    \begin{pmatrix}
      -1 & 0 \\
      -{\mathcal R}_0 & -1
    \end{pmatrix}
    \Big[
    \begin{pmatrix}
      1 & \tau \\
      0 & 1
    \end{pmatrix}
    \begin{pmatrix}
      -1 & 0 \\
      -{\mathcal R} & -1
    \end{pmatrix}
    \Big]^n
    \begin{pmatrix}
      1 & \tau \\
      0 & 1
    \end{pmatrix}
    \begin{pmatrix}
      -1 & 0 \\
      -{\mathcal R}_0 & -1
    \end{pmatrix}
    \;.
  \end{align*}
  Simplifying the expression for $M$ yields
  \begin{align*}
    M
    &=
    \begin{pmatrix}
      1 & l \\
      0 & 1
    \end{pmatrix}
    \begin{pmatrix}
      -1 & 0 \\
      -{\mathcal R}_0 & -1
    \end{pmatrix}
    \Big[
    \begin{pmatrix}
      3 & -\tau \\
      \frac{4}{\tau} & -1
    \end{pmatrix}
    \Big]^n
    \begin{pmatrix}
      -1 - \tau\,{\mathcal R}_0 & \tau \\
      -{\mathcal R}_0 & -1
    \end{pmatrix}
    \\
    &=
    \begin{pmatrix}
      -1 - l\,{\mathcal R}_0 & -l \\
      -{\mathcal R}_0 & -1
    \end{pmatrix}
    \begin{pmatrix}
      1 + 2\,n & -n\,\tau \\
      \frac{4\,n}{\tau} & 1 - 2\,n
    \end{pmatrix}
    \begin{pmatrix}
      -1 - \tau\,{\mathcal R}_0 & -\tau \\
      -{\mathcal R}_0 & -1
    \end{pmatrix}
    \\
    &=
    \begin{pmatrix}
      \frac{
      \tau
      +
      [ 2\,n + (n+1)\,{\mathcal R}_0\,\tau
      ]\,[ \tau + (2 + {\mathcal R}_0\,\tau)\,l ]
      }{\tau}
      &
      (2\,n+1)\,l + (n+1)\,(1 + l\,{\mathcal R}_0)\,\tau
      \\
      \frac{
      (2 + {\mathcal R}_0\,\tau)\,[ 2\,n + (n+1)\,{\mathcal R}_0\,\tau ]
      }{\tau}
      &
      2\,n+1 + (n+1)\,{\mathcal R}_0\,\tau
    \end{pmatrix}
    \;.
  \end{align*}
  and hence
  \begin{align*}
    \operatorname{tr} M^2
    =
    (\operatorname{tr} M)^2 - 2
    \;,\qquad
    \operatorname{tr} M
    &=
    2
    +
    [ 2\,n + (n+1)\,{\mathcal R}_0\,\tau ]
    \,[ 2 + (2 + {\mathcal R}_0\,\tau)\,\frac{l}{\tau} ]
    \;.
  \end{align*}
  Using
  $
  {\mathcal R}_0
  =
  -\frac{4}{\tau}
  \,\frac{{\rho}}{{\rho}_0}
  $
  it follows that
  \begin{align*}
    \frac{1}{2}\,\operatorname{tr} M
    &=
    1
    +
    2\,[ n - 2\,(n+1)\,\frac{{\rho}}{{\rho}_0} ]
    \,[ 1 + (1 - 2\,\frac{{\rho}}{{\rho}_0} )
    \,\frac{l}{\tau} ]
    \\
    &=
    -1
    -
    4\,(n+1)\,\Big[
    \frac{{\rho}}{{\rho}_0}
    -
    \frac{1}{2}
    \Big]
    +
    8\,(n+1)
    \,\Big[
    \frac{{\rho}}{{\rho}_0}
    -
    \frac{n }{ 2\,(n+1)}
    \Big]
    \,\Big[
    \frac{{\rho}}{{\rho}_0} - \frac{1}{2}
    \Big]
    \,\frac{l}{\tau}
  \end{align*}
  Therefore,
  \begin{align*}
    \frac{1}{2}\,\operatorname{tr} M
    < 1
    &\iff
    -1
    -
    4\,(n+1)\,\Big[
    \frac{{\rho}}{{\rho}_0}
    -
    \frac{1}{2}
    \Big]
    +
    8\,(n+1)
    \,\Big[
    \frac{{\rho}}{{\rho}_0}
    -
    \frac{n }{ 2\,(n+1)}
    \Big]
    \,\Big[
    \frac{{\rho}}{{\rho}_0} - \frac{1}{2}
    \Big]
    \,\frac{l}{\tau}
    <
    1
    \\
    &\iff
    \Big[
    \frac{{\rho}}{{\rho}_0}
    -
    \frac{n }{ 2\,(n+1)}
    \Big]
    \,\Big[
    \frac{{\rho}}{{\rho}_0} - \frac{1}{2}
    \Big]
    \,\frac{l}{\tau}
    <
    \frac{1}{2}
    \,\Big[
    \frac{{\rho}}{{\rho}_0}
    -
    \frac{n}{2\,(n+1) }
    \Big]
  \end{align*}
  and
  \begin{align*}
    \frac{1}{2}\,\operatorname{tr} M
    > -1
    &\iff
    -1
    -
    4\,(n+1)\,\Big[
    \frac{{\rho}}{{\rho}_0}
    -
    \frac{1}{2}
    \Big]
    +
    8\,(n+1)
    \,\Big[
    \frac{{\rho}}{{\rho}_0}
    -
    \frac{n }{ 2\,(n+1)}
    \Big]
    \,\Big[
    \frac{{\rho}}{{\rho}_0} - \frac{1}{2}
    \Big]
    \,\frac{l}{\tau}
    > -1
    \\
    &\iff
    \Big[
    \frac{{\rho}}{{\rho}_0}
    -
    \frac{n }{ 2\,(n+1)}
    \Big]
    \,\Big[
    \frac{{\rho}}{{\rho}_0} - \frac{1}{2}
    \Big]
    \,\frac{l}{\tau}
    >
    \frac{1}{2}
    \,\Big[
    \frac{{\rho}}{{\rho}_0}
    -
    \frac{1}{2}
    \Big]
  \end{align*}

  Choosing the radius of curvature ${\rho}_0$ at
  $s_0$ such that
  \begin{align*}
    \frac{n }{ 2\,(n+1)}
    <
    \frac{{\rho}}{{\rho}_0}
    <
    \frac{1}{2}
  \end{align*}
  yields
  \begin{align*}
    \frac{1}{2}\,\operatorname{tr} M < 1
    \qquad\text{for all}\qquad l \geq 0
  \end{align*}
  and
  \begin{align*}
    \frac{1}{2}\,\operatorname{tr} M > -1
    &\iff
    \frac{l}{\tau}
    <
    \frac{1}{2}
    \,\frac{1}{
    \frac{{\rho}}{{\rho}_0}
    -
    \frac{n }{ 2\,(n+1)}
    }
    \;.
  \end{align*}
  Therefore, the periodic orbit remains elliptic for all $l\geq 0$ with
  $
  l <
  \frac{\tau}{2}
  \,\frac{1}{
  \frac{{\rho}}{{\rho}_0}
  -
  \frac{n }{ 2\,(n+1)}
  }
  $.
  This upper limit can be made as large as desired by
  choosing ${\rho}_0$ such that
  $
  \frac{{\rho}}{{\rho}_0}
  -
  \frac{n }{ 2\,(n+1)}
  $
  is positive and near zero, i.e. by choosing
  $
  {\rho}_0
  \approx
  2\,(1 + \frac{1}{n}) \,{\rho}
  $.

  Provided that we change the small smoothening of the circular
  boundary component to a slightly larger smoothening region
  we can combine the above construction without changing the
  orbits constructed in
  item (\ref{thm_large_separations_nonlinearStability}) of
  Theorem~\ref{thm_large_separations}. By possibly adjusting
  the above construction we can again guarantee the nonlinear
  stability of these periodic orbits.
  Hence we can find arbitrarily short smoothenings of a
  circular boundary component such that the corresponding
  smooth stadium has a nonlinearly stable periodic orbit
  for all separation distances. And since the nonvanishing of
  the first Birkhoff coefficient is an open condition for
  $C^5$ perturbation of $\Gamma$
  we obtain an open set of smoothenings
  with the property that the corresponding smooth stadium has
  nonlinearly stable periodic orbits for all seperation distances.
  This finishes the proof of Theorem~\ref{thm_all_separations}.
\end{proof}

\section{Conclusions}
\label{sect_conclusions}

Since the appearance of \cite{MR0328219} and \cite{MR0357736,MR530154}
it is known that the dynamics of the billiards depends
on the smoothness of the boundary in an essential way.

More precisely, as long as the boundary of a strictly convex
two-dimensional
billiard table is of class $C^6$ it was shown in \cite{douady}
that a positive measure family of caustics is present near
the boundary of the billiard table. In particular, such billiards
are never ergodic.

The results of \cite{MR0357736,MR530154} show that the
if the boundary of convex billiards is allowed to be only $C^1$,
then the resulting billiards may be hyperbolic and ergodic.
In fact, the $C^1$ smoothness is only imposed at
a few isolated points, and away from these points the boundary
can be $C^\infty$ or even analytic.

A standard example of billiards considered in
\cite{MR0357736,MR530154} is the stadium billiard.
Our result of Theorem~\ref{thm_large_separations}
shows that if at the endpoints of the circular boundary segments
of the stadium the curvature is smoothed out,
then elliptic periodic orbits are present for a large class
of such $C^2$ stadium like billiards.

By our assumption the curvature is continuous on each boundary
component. This implies that the global smoothness of the boundary
is either $C^1$ or $C^2$; no intermediate fractional smoothness $C^{1+\alpha}$
is possible. Either there is a point on the boundary where the curvature
has a jump, or not.
For global $C^1$ or $C^0$ smoothness of the boundary of the billiard table
\cite{MR0357736,MR530154} provides large classes of completely
hyperbolic and ergodic billiards.
Therefore, our results indicate that for
convex billiards with piecewise smooth ($C^3$ is enough) and
globally $C^2$-smooth boundary
elliptic periodic orbits are expected to be present.
Hence global $C^2$ smoothness
represents the critical smoothness for elliptic
dynamics to be present in convex billiards.

\bibliography{smooth_stadium}
\bibliographystyle{plain}

\begin{appendix}
  \numberwithin{equation}{section}
  \numberwithin{theorem}{section}
  \section{Notation and some facts about planar billiards}
  \label{sect_billiard_basicFacts}

  Below we collect some well-known facts about billiards
  \cite{MR2229799}, and derive the higher order
  expansions of the billiard map.
  The latter will be used
  extensively in the study of nonlinear stability of
  periodic orbits, by means of Birkhoff normal form expansion.
  In what follows we will use the notations
  \begin{equation}
    \label{eqn_definition_BoundaryComponent_TangentAngle_Curvature}
    \begin{split}
      \Gamma'(s)
      &=
      {\mathcal T}(s)
      \;,\qquad
      {\mathcal T}(s)
      =
      \begin{pmatrix}
        \cos\theta(s) \\
        \sin\theta(s)
      \end{pmatrix}
      \;,
      \\
      {\mathcal N}(s)
      &=
      \begin{pmatrix}
        -\sin\theta(s) \\
        \cos\theta(s)
      \end{pmatrix}
      \;,\qquad
      \theta'(s)
      =
      - {\mathcal K}(s)
      \;.
    \end{split}
  \end{equation}

  \begin{lemma}[Billiard map]
    \label{lem_billiardMap_local_sPhi}
    Let
    $
    (\bar{s}_{i+1}, \bar{\varphi}_{i+1})
    =
    {\mathcal F}(\bar{s}_i, \bar{\varphi}_i)
    $
    be a given segment of a billiard trajectory.
    Then the billiard map
    $
    (s_{i+1}, \varphi_{i+1})
    =
    {\mathcal F}(s_i, \varphi_i)
    $
    is (locally) determined by the relations
    \begin{align*}
      \int_{\bar{s}_{i+1}}^{s_{i+1}}
      &
      \cos[
      \bar{\varphi}_{i+1}
      + \theta(\bar{s}_{i+1})
      - \theta(\sigma)
      - \varphi_i + \bar{\varphi}_i
      + \theta(s_i)
      - \theta(\bar{s}_i)
      ]
      \,d\sigma
      \\
      &=
      -\int_{\bar{s}_i}^{s_i}
      \cos[
      \varphi_i
      +
      \theta(\sigma) -\theta(s_i)
      ]
      \,d\sigma
      -
      \bar{\tau}_{i,i+1}
      \,\sin[
      \varphi_i - \bar{\varphi}_i
      -
      \theta(s_i)
      +
      \theta(\bar{s}_i)
      ]
    \end{align*}
    and
    \begin{align*}
      \varphi_{i+1} - \bar{\varphi}_{i+1}
      &=
      \bar{\varphi}_i - \varphi_i
      + \theta(s_i)
      - \theta(\bar{s}_i)
      + \theta(\bar{s}_{i+1})
      - \theta(s_{i+1})
      \;,
    \end{align*}
    where
    $\theta'(s) = -{\mathcal K}(s)$
    for any $s$.
    The free path is then given by
    \begin{align*}
      \tau_{i,i+1}
      &=
      \int_{ \bar{s}_{i+1} }^{ s_{i+1} }
      \sin[
      \bar{\varphi}_{i+1}
      + \theta(\bar{s}_{i+1})
      - \theta(\sigma)
      - \varphi_i
      + \bar{\varphi}_i
      + \theta(s_i)
      - \theta(\bar{s}_i)
      ]
      \,d\sigma
      \\
      &\qquad
      -
      \int_{ \bar{s}_i }^{ s_i }
      \sin[
      \varphi_i
      + \theta(\sigma)
      - \theta(s_i)
      ]
      \,d\sigma
      +
      \bar{\tau}_{i,i+1}
      \,\cos[
      \varphi_i
      - \bar{\varphi}_i
      - \theta(s_i)
      + \theta(\bar{s}_i)
      ]
      \;.
    \end{align*}
  \end{lemma}
  \begin{proof}
    Given $s_i$ and $\varphi_i$
    then arc length parameter $s_{i+1}$ of the
    next point of reflection is determined by
    \begin{align*}
      \Gamma(s_{i+1})
      &=
      \Gamma(s_i)
      +
      \tau_{i,i+1}\,\Big[
      {\mathcal T}(s_i)\,\sin\varphi_i
      +
      {\mathcal N}(s_i)\,\cos\varphi_i
      \Big]
    \end{align*}
    for some $\tau_{i,i+1} > 0$.
    Using the parametrization of the boundary in terms of its
    curvature \eqref{eqn_definition_BoundaryComponent_TangentAngle_Curvature}
    we have that
    \begin{align*}
      {\mathcal T}(s_i)\,\sin\varphi_i
      +
      {\mathcal N}(s_i)\,\cos\varphi_i
      &=
      \begin{pmatrix}
        \sin[\varphi_i -\theta(s_i)] \\
        \cos[\varphi_i -\theta(s_i)]
      \end{pmatrix}
    \end{align*}
    and hence
    \begin{align*}
      \Gamma(s_{i+1})
      &=
      \Gamma(s_i)
      +
      \tau_{i,i+1}\,
      \begin{pmatrix}
        \sin[\varphi_i -\theta(s_i)] \\
        \cos[\varphi_i -\theta(s_i)]
      \end{pmatrix}
      \;.
    \end{align*}
    Therefore, the values of $s_{i+1}$ and $\tau_{i,i+1}$
    are (locally) determined by
    \begin{equation}
      \label{eqn_s_tau}
      \begin{split}
        \Gamma(s_{i+1})
        -
        \Gamma(\bar{s}_{i+1})
        &=
        \Gamma(s_i)
        -
        \Gamma(\bar{s}_i)
        +
        \tau_{i,i+1}\,
        \begin{pmatrix}
          \sin[\varphi_i -\theta(s_i)] \\
          \cos[\varphi_i -\theta(s_i)]
        \end{pmatrix}
        \\
        &\qquad
        -
        \bar{\tau}_{i,i+1}\,
        \begin{pmatrix}
          \sin[\bar{\varphi}_i -\theta(\bar{s}_i)] \\
          \cos[\bar{\varphi}_i -\theta(\bar{s}_i)]
        \end{pmatrix}
        \;.
      \end{split}
    \end{equation}

    Eliminating $\tau_{i,i+1}$ yields
    \begin{align*}
      [
      \Gamma(s_{i+1})
      -
      \Gamma(\bar{s}_{i+1})
      ]
      \cdot
      \begin{pmatrix}
        -\cos[\varphi_i -\theta(s_i)] \\
        \sin[\varphi_i -\theta(s_i)]
      \end{pmatrix}
      &=
      [
      \Gamma(s_i)
      -
      \Gamma(\bar{s}_i)
      ]
      \cdot
      \begin{pmatrix}
        -\cos[\varphi_i -\theta(s_i)] \\
        \sin[\varphi_i -\theta(s_i)]
      \end{pmatrix}
      \\
      &\qquad
      -
      \bar{\tau}_{i,i+1}
      \,\sin[
      \varphi_i - \bar{\varphi}_i
      -
      \theta(s_i)
      +
      \theta(\bar{s}_i)
      ]
      \;,
    \end{align*}
    which determines (locally) $s_{i+1}$.
    Writing the differences
    $
    \Gamma(s_i)
    -
    \Gamma(\bar{s}_i)
    $
    and
    $
    \Gamma(s_{i+1})
    -
    \Gamma(\bar{s}_{i+1})
    $
    in terms of the integral over the corresponding tangent vector
    we obtain
    \begin{equation}
      \label{eqn_s}
      \begin{split}
        \int_{\bar{s}_{i+1}}^{s_{i+1}}
        \cos[
        \varphi_i
        +
        \theta(\sigma) - \theta(s_i)
        ]
        \,d\sigma
        &=
        \int_{\bar{s}_i}^{s_i}
        \cos[
        \varphi_i
        +
        \theta(\sigma) -\theta(s_i)
        ]
        \,d\sigma
        \\
        &\qquad
        +
        \bar{\tau}_{i,i+1}
        \,\sin[
        \varphi_i - \bar{\varphi}_i
        -
        \theta(s_i)
        +
        \theta(\bar{s}_i)
        ]
        \;.
      \end{split}
    \end{equation}

    Furthermore, the angle of reflection $\varphi_{i+1}$ is
    determined by
    \begin{equation}
      \label{eqn_phi}
      \varphi_{i+1} + \theta(s_{i+1})
      =
      \pi - \varphi_i + \theta(s_i)
      \;.
    \end{equation}

    Using \eqref{eqn_phi} we obtain
    \begin{align*}
      \varphi_i
      +
      \theta(\sigma) - \theta(s_i)
      &=
      \varphi_i - \bar{\varphi}_i
      + \theta(\sigma)
      - \theta(\bar{s}_{i+1})
      - \theta(s_i)
      + \bar{\varphi}_i
      + \theta(\bar{s}_{i+1})
      \\
      &=
      \pi
      - \bar{\varphi}_{i+1}
      + \varphi_i - \bar{\varphi}_i
      + \theta(\sigma)
      - \theta(\bar{s}_{i+1})
      - \theta(s_i)
      + \theta(\bar{s}_i)
      \;.
    \end{align*}
    Substituting this back into equation \eqref{eqn_s}
    for $s_{i+1}$
    proves the second relation for the billiard map.

    To finish the proof, note that \eqref{eqn_s_tau} also shows that
    \begin{align*}
      [
      \Gamma(s_{i+1})
      -
      \Gamma(\bar{s}_{i+1})
      ]
      &
      \cdot
      \begin{pmatrix}
        \sin[\varphi_i -\theta(s_i)] \\
        \cos[\varphi_i -\theta(s_i)]
      \end{pmatrix}
      =
      [
      \Gamma(s_i)
      -
      \Gamma(\bar{s}_i)
      ]
      \cdot
      \begin{pmatrix}
        \sin[\varphi_i -\theta(s_i)] \\
        \cos[\varphi_i -\theta(s_i)]
      \end{pmatrix}
      \\
      &
      +
      \tau_{i,i+1}
      -
      \bar{\tau}_{i,i+1}
      \,\cos[
      \varphi_i
      - \bar{\varphi}_i
      - \theta(s_i)
      + \theta(\bar{s}_i)
      ]
    \end{align*}
    which can be simplified to
    \begin{align*}
      \int_{ \bar{s}_{i+1} }^{ s_{i+1} }
      &
      \sin[
      \varphi_i
      - \theta(s_i)
      + \theta(\sigma)
      ]
      \,d\sigma
      =
      \int_{ \bar{s}_i }^{ s_i }
      \sin[
      \varphi_i
      - \theta(s_i)
      + \theta(\sigma)
      ]
      \,d\sigma
      \\
      &
      +
      \tau_{i,i+1}
      -
      \bar{\tau}_{i,i+1}
      \,\cos[
      \varphi_i
      - \bar{\varphi}_i
      - \theta(s_i)
      + \theta(\bar{s}_i)
      ]
      \;.
    \end{align*}
    From
    \eqref{eqn_phi}
    we obtain
    $
    \bar{\varphi}_{i+1} + \theta(\bar{s}_{i+1})
    =
    \pi - \bar{\varphi}_i + \theta(\bar{s}_i)
    $, and hence can rewrite the previous relation as
    \begin{align*}
      \int_{ \bar{s}_{i+1} }^{ s_{i+1} }
      &
      \sin[
      \pi - \bar{\varphi}_{i+1}
      + \varphi_i
      - \bar{\varphi}_i
      - \theta(s_i)
      + \theta(\bar{s}_i)
      + \theta(\sigma)
      - \theta(\bar{s}_{i+1})
      ]
      \,d\sigma
      \\
      &=
      \int_{ \bar{s}_i }^{ s_i }
      \sin[
      \varphi_i
      - \theta(s_i)
      + \theta(\sigma)
      ]
      \,d\sigma
      \\
      &\qquad
      +
      \tau_{i,i+1}
      -
      \bar{\tau}_{i,i+1}
      \,\cos[
      \varphi_i
      - \bar{\varphi}_i
      - \theta(s_i)
      + \theta(\bar{s}_i)
      ]
      \;.
    \end{align*}
    Simplifying yields the claimed expression, and finishes the proof.
  \end{proof}

  The billiard flow in Cartesian coordinates
  $x$, $y$ on the billiard table $Q$
  and velocity vector $v_x = \cos\omega$, $v_y = \sin\omega$
  preserves the form $dx \wedge dy \wedge d\omega$.
  Near such a point $(\bar{x}, \bar{y}, \bar{\omega})$
  one can then locally define the so-called Jacobi coordinates
  $\eta$, $\xi$, $\omega$
  (e.g. \cite{MR2229799})
  \begin{equation*}
    \begin{pmatrix}
      x \\
      y
    \end{pmatrix}
    =
    \begin{pmatrix}
      \bar{x} \\
      \bar{y}
    \end{pmatrix}
    +
    \eta\,v
    +
    \xi\,v^\perp
    \;,\qquad
    v
    =
    \begin{pmatrix}
      \cos\omega \\
      \sin\omega
    \end{pmatrix}
  \end{equation*}
  or simply
  \begin{equation}
    \label{eqn_definition_localJacobiCoordinates_Flow}
    \begin{split}
      \eta
      &=
      [ x - \bar{x} ]\,\cos\omega
      +
      [ y - \bar{y} ]\,\sin\omega
      \\
      \xi
      &=
      -[ x - \bar{x} ]\,\sin\omega
      +
      [ y - \bar{y} ]\,\cos\omega
      \;.
    \end{split}
  \end{equation}
  It is straightforward to check that the billiard flow
  preserves the form $d\eta \wedge d\xi \wedge d\omega$.
  Restricting the points
  $(\bar{x}, \bar{y})$
  and
  $(x,y)$
  to some boundary component $\Gamma$, i.e.
  $\Gamma(\bar{s}) = (\bar{x}, \bar{y})$
  and $\Gamma(s) = (x, y)$,
  the (local) Jacobi
  coordinates induce the local coordinates
  $(\mathsf{x},\mathsf{y})$ on
  $\Gamma$ in place of the arc length parameter
  $s$ and the angle $\omega$
  \begin{align*}
    \mathsf{x}
    =
    \mathsf{x}(s, \omega)
    &=
    [
    \Gamma(s)
    -
    \Gamma(\bar{s})
    ]
    \cdot
    \begin{pmatrix}
      -\sin\omega \\
      \cos\omega
    \end{pmatrix}
    \;,\qquad
    \mathsf{y}
    =
    \mathsf{y}(s, \omega)
    =
    \omega - \bar{\omega}
  \end{align*}
  and
  \begin{align*}
    \mathsf{z}
    =
    \mathsf{z}(s, \omega)
    &=
    [
    \Gamma(s)
    -
    \Gamma(\bar{s})
    ]
    \cdot
    \begin{pmatrix}
      \cos\omega \\
      \sin\omega
    \end{pmatrix}
    \;,
  \end{align*}
  which corresponds to $\eta$.
  Clearly, these coordinates are well defined away from tangencies,
  and can be expressed as
  \begin{equation}
    \label{eqn_flowCoordinates_boundaryComponent}
    \begin{split}
      \mathsf{x}
      =
      \mathsf{x}(s, \omega)
      &=
      \int_{\bar{s}}^{s}
      \sin[\theta(\sigma) - \omega ]
      \,d\sigma
      \;,\qquad
      \theta'(s)
      =
      - {\mathcal K}(s)
      \\
      \mathsf{y}
      =
      \mathsf{y}(s, \omega)
      &=
      \omega - \bar{\omega}
      \\
      \mathsf{z}
      =
      \mathsf{z}(s, \omega)
      &=
      \int_{\bar{s}}^{s}
      \cos[ \theta(\sigma) -  \omega ]
      \,d\sigma
    \end{split}
  \end{equation}
  where we used the (local) parametrization
  \eqref{eqn_definition_BoundaryComponent_TangentAngle_Curvature}
  of $\Gamma$
  in terms of the curvature ${\mathcal K}$.

  Note also that
  \eqref{eqn_flowCoordinates_boundaryComponent}
  implies
  \begin{equation*}
    d\mathsf{x}
    =
    \sin[\theta(s) - \omega ]
    \,ds
    -
    \mathsf{z}
    \,d\omega
    \;,\qquad
    d\mathsf{y}
    =
    d\omega
  \end{equation*}
  and hence
  \begin{equation}
    \label{eqn_definition_symplecticForm}
    d\mathsf{x} \wedge d\mathsf{y}
    =
    \sin[\theta(s) - \omega ]
    \,ds \wedge d\omega
    =
    \cos\varphi
    \,ds \wedge d\varphi
    \;,
  \end{equation}
  where
  \begin{equation}
    \label{eqn_definition_bMapAngle}
    \varphi
    \equiv
    \varphi(s, \omega)
    =
    \theta(s) + \frac{\pi}{2} - \omega
  \end{equation}
  denotes the angle of reflection corresponding
  at the point $\Gamma(s)$, i.e.
  the angle between the flow direction given by $\omega$
  and the normal line spanned by ${\mathcal N}(s)$,
  which is counted positively in direction of tangent vector
  $
  \Gamma'(s)
  =
  {\mathcal T}(s)
  =
  \begin{pmatrix}
    \cos\theta(s) \\
    \sin\theta(s)
  \end{pmatrix}
  $.

  \begin{lemma}[Billiard dynamics in local Jacobi coordinates]
    \label{lem_billiard_localFlowCoordinates}
    $ $
    \begin{enumerate}[(i)]
      \item (Free flight)
        Suppose that
        $
        (\bar{x}_1, \bar{y}_1)
        =
        (\bar{x}_0, \bar{y}_0)
        +
        \bar{\tau}_{0,1}
        \,(\cos\bar{\omega}_0 ,\sin\bar{\omega}_0)
        $,
        $
        (\bar{x}_1, \bar{y}_1)
        =
        \Gamma_1(\bar{s}_1)
        $.
        Then the billiard flow at the moment right before
        the next reflection off of $\Gamma_1$
        is locally given by
        \begin{align*}
          \mathsf{x}_1
          &=
          \mathsf{x}_0
          +
          \bar{\tau}_{0,1}
          \,\sin\mathsf{y}_0
          \;,\qquad
          \mathsf{y}_1
          =
          \mathsf{y}_0
          \;,
        \end{align*}
        where $\bar{\omega}_1 = \bar{\omega}_0$ for the
        flow direction at $(\bar{x}_1, \bar{y}_1)$.
        The corresponding expressions for $\tau_{0,1}$
        and $\mathsf{z}_1$ read
        \begin{align*}
          \tau_{0,1}
          &=
          \mathsf{z}_1
          -
          \mathsf{z}_0
          +
          \bar{\tau}_{0,1}
          \,\cos\mathsf{y}_0
          \;,\qquad
          \mathsf{z}_1
          =
          \int_{\bar{s}_1}^{s_1}
          \cos[
          \theta(\sigma) - \bar{\omega}_0 - \mathsf{y}_0
          ]
          \,d\sigma
          \;.
        \end{align*}

        If in addition
        $
        (\bar{x}_0, \bar{y}_0)
        =
        \Gamma_0(\bar{s}_0)
        $,
        and the initial points are restricted to
        $\Gamma_0$, then
        locally
        \begin{align*}
          \mathsf{x}_1
          &=
          \mathsf{x}_0
          +
          \bar{\tau}_{0,1}
          \,\sin\mathsf{y}_0
          \;,\qquad
          \mathsf{y}_1
          =
          \mathsf{y}_0
          \;,
        \end{align*}
        and
        \begin{align*}
          \tau_{0,1}
          =
          \mathsf{z}_1
          -
          \mathsf{z}_0
          +
          \bar{\tau}_{0,1}
          \,\cos\mathsf{y}_0
          \;,\qquad
          \mathsf{z}_i
          &=
          \int_{\bar{s}_i}^{s_i}
          \cos[
          \theta(\sigma) - \bar{\omega}_0 - \mathsf{y}_0
          ]
          \,d\sigma
          \quad (i=0,1)
        \end{align*}
        is the expression for the corresponding free path.

      \item (Reflection)
        Let
        $
        (\bar{x}, \bar{y})
        =
        \Gamma(\bar{s})
        $ and $\bar{\omega}$ be given.
        Then the reflection off of $\Gamma$ is locally
        given by
        \begin{align*}
          \mathsf{x}^+
          &=
          \int_{\bar{s}}^{s}
          \sin[
          \theta(\sigma)
          -
          2\,\theta(s)
          +
          \mathsf{y}^- + \bar{\omega}^-
          ]
          \,d\sigma
          \\
          \mathsf{y}^+
          &=
          2\,[
          \theta(s)
          -
          \theta(\bar{s})
          ]
          - \mathsf{y}^-
          \;,
        \end{align*}
        where
        $
        s
        \equiv
        s(\mathsf{x}^-, \mathsf{y}^-)
        $ is determined by
        \begin{align*}
          \mathsf{x}^-
          &=
          \int_{\bar{s}}^{s}
          \sin[
          \theta(\sigma)
          -
          \mathsf{y}^-
          -
          \bar{\omega}^-
          ]
          \,d\sigma
        \end{align*}
        and
        $
        \bar{\omega}^+
        =
        2\,\theta(\bar{s}) - \bar{\omega}^-
        $.

      \item (Invariant form)
        Both, a free flight and a reflection preserve the form
        $d\mathsf{x} \wedge d\mathsf{y}$,
        and hence so does the billiard map.
    \end{enumerate}
  \end{lemma}
  \begin{proof}
    Locally the billiard flow is determined by
    \begin{align*}
      \Gamma_1(s_1)
      =
      \begin{pmatrix}
        x_0 \\
        y_0
      \end{pmatrix}
      +
      \tau_{0,1}
      \begin{pmatrix}
        \cos\omega_0 \\
        \sin\omega_0
      \end{pmatrix}
      \;,\qquad
      \omega_1 = \omega_0
    \end{align*}
    where $(x_0, y_0, \omega_0)$ is assumed to be close to
    $(\bar{x}_0, \bar{y}_0, \bar{\omega}_0)$. Therefore,
    this is equivalent to
    \begin{align*}
      \Gamma_1(s_1)
      -
      \Gamma_1(\bar{s}_1)
      &=
      \begin{pmatrix}
        x_0 - \bar{x}_0 \\
        y_0 - \bar{y}_0
      \end{pmatrix}
      +
      \tau_{0,1}
      \begin{pmatrix}
        \cos\omega_0 \\
        \sin\omega_0
      \end{pmatrix}
      -
      \bar{\tau}_{0,1}
      \begin{pmatrix}
        \cos\bar{\omega}_0 \\
        \sin\bar{\omega}_0
      \end{pmatrix}
      \\
      \omega_1
      &=
      \omega_0
      \;,\qquad
      \bar{\omega}_1
      =
      \bar{\omega}_0
      \;.
    \end{align*}
    Hence
    \begin{align*}
      [
      \Gamma_1(s_1)
      -
      \Gamma_1(\bar{s}_1)
      ]
      \cdot
      \begin{pmatrix}
        -\sin\omega_0 \\
        \cos\omega_0
      \end{pmatrix}
      &=
      \begin{pmatrix}
        x_0 - \bar{x}_0 \\
        y_0 - \bar{y}_0
      \end{pmatrix}
      \cdot
      \begin{pmatrix}
        -\sin\omega_0 \\
        \cos\omega_0
      \end{pmatrix}
      +
      \bar{\tau}_{0,1}
      \,\sin[ \omega_0 - \bar{\omega}_0 ]
      \\
      \omega_1
      -
      \bar{\omega}_1
      &=
      \omega_0
      -
      \bar{\omega}_0
    \end{align*}
    and
    \begin{align*}
      [
      \Gamma_1(s_1)
      -
      \Gamma_1(\bar{s}_1)
      ]
      \cdot
      \begin{pmatrix}
        \cos\omega_0 \\
        \sin\omega_0
      \end{pmatrix}
      &=
      \begin{pmatrix}
        x_0 - \bar{x}_0 \\
        y_0 - \bar{y}_0
      \end{pmatrix}
      \cdot
      \begin{pmatrix}
        \cos\omega_0 \\
        \sin\omega_0
      \end{pmatrix}
      +
      \tau_{0,1}
      -
      \bar{\tau}_{0,1}
      \,\cos[ \omega_0 - \bar{\omega}_0 ]
      \;.
    \end{align*}
    In terms of the (local) Jacobi coordinates
    \eqref{eqn_definition_localJacobiCoordinates_Flow}
    and
    \eqref{eqn_flowCoordinates_boundaryComponent}
    this takes on the form
    \begin{align*}
      \mathsf{x}_1
      &=
      \xi_0
      +
      \bar{\tau}_{0,1}
      \,\sin\mathsf{y}_0
      \;,\qquad
      \mathsf{y}_1
      =
      \mathsf{y}_0
      \\
      \mathsf{z}_1
      &=
      \eta_0
      +
      \tau_{0,1}
      -
      \bar{\tau}_{0,1}
      \,\cos\mathsf{y}_0
      \;,\qquad
      \mathsf{z}_1
      =
      \int_{\bar{s}_1}^{s_1}
      \cos[ \theta(\sigma) - \bar{\omega}_0 - \mathsf{y}_0 ]
      \,d\sigma
      \;,
    \end{align*}
    and since restricting the initial points to $\Gamma_0$
    makes $\xi_0 = \mathsf{x}_0$
    and
    \begin{equation*}
      \eta_0 = \mathsf{z}_0
      =
      \int_{\bar{s}_0}^{s_0}
      \cos[ \theta(\sigma) - \bar{\omega}_0 - \mathsf{y}_0 ]
      \,d\sigma
      \;,
    \end{equation*}
    the first claim follows.

    Let $(x^-,y^-) = \Gamma(s^-)$ and $\omega^-$
    be any point (near the reference point). Then the reflection
    about $\Gamma$ is given by
    $s^+ = s^-$,
    $\omega^+ = 2\,\theta(s) - \omega^-$.
    Since the corresponding Jacobi coordinates
    \eqref{eqn_flowCoordinates_boundaryComponent}
    are given by
    \begin{align*}
      \mathsf{x}^-
      &=
      \int_{\bar{s}}^{s}
      \sin[\theta(\sigma) - \omega^- ]
      \,d\sigma
      \;,\qquad
      \mathsf{y}^-
      =
      \omega^- - \bar{\omega}^-
    \end{align*}
    and
    \begin{align*}
      \mathsf{x}^+
      &=
      \int_{\bar{s}}^{s}
      \sin[\theta(\sigma) - \omega^+ ]
      \,d\sigma
      =
      \int_{\bar{s}}^{s}
      \sin[
      \theta(\sigma)
      -
      2\,\theta(s)
      +
      \omega^-
      ]
      \,d\sigma
      \\
      \mathsf{y}^+
      &=
      \omega^+ - \bar{\omega}^+
      =
      2\,\theta(s)
      - 2\,\theta(\bar{s})
      - \omega^-
      + \bar{\omega}^-
      =
      2\,[
      \theta(s)
      -
      \theta(\bar{s})
      ]
      - \mathsf{y}^-
      \;.
    \end{align*}
    This proves the second claim.

    The invariance of $d\mathsf{x} \wedge d\mathsf{y}$
    under the free flight is readily verified. To see the invariance
    of this form under the reflection map observe that it follows
    from the above that
    \begin{align*}
      d\mathsf{x}^-
      &=
      \sin[
      \theta(s)
      -
      \mathsf{y}^-
      -
      \bar{\omega}^-
      ]
      \,ds
      \\
      d\mathsf{y}^+
      &=
      2\,\theta'(s)
      \,ds
      -
      d\mathsf{y}^-
      \\
      d\mathsf{x}^+
      &=
      \Big[
      \int_{\bar{s}}^{s}
      \cos[
      \theta(\sigma)
      -
      2\,\theta(s)
      +
      \mathsf{y}^- + \bar{\omega}^-
      ]
      \,d\sigma
      \Big]
      \,
      [
      d\mathsf{y}^-
      -
      2\,\theta'(s)
      \,ds
      ]
      \\
      &\qquad
      +
      \sin[
      -\theta(s)
      +
      \mathsf{y}^- + \bar{\omega}^-
      ]
      \,ds
      \\
      &=
      -
      \Big[
      \int_{\bar{s}}^{s}
      \cos[
      \theta(\sigma)
      -
      2\,\theta(s)
      +
      \mathsf{y}^- + \bar{\omega}^-
      ]
      \,d\sigma
      \Big]
      \,d\mathsf{y}^+
      -
      d\mathsf{x}^-
      \;.
    \end{align*}
    Therefore,
    \begin{align*}
      d\mathsf{x}^+
      \wedge
      d\mathsf{y}^+
      &=
      -d\mathsf{x}^-
      \wedge
      d\mathsf{y}^+
      =
      -d\mathsf{x}^-
      \wedge
      [
      2\,\theta'(s)
      \,ds
      -
      d\mathsf{y}^-
      ]
      =
      d\mathsf{x}^-
      \wedge
      d\mathsf{y}^-
    \end{align*}
    follows, which finishes the proof of the third claim.
  \end{proof}

  For any given $\bar{s}$ and $\bar{\varphi}$
  introduce the following notation
  \begin{equation}
    \label{eqn_def_Rprime}
    {\mathcal R}
    =
    \frac{ 2\,{\mathcal K}(\bar{s})
    }{ \cos\bar{\varphi} }
    \;,\quad
    {\mathcal R}^{(1)}
    =
    \frac{ 4\,{\mathcal K}'(\bar{s})
    }{ \cos^2\bar{\varphi} }
    \;,\quad
    {\mathcal R}^{(2)}
    =
    \frac{ 8\,{\mathcal K}''(\bar{s})
    }{ \cos^3\bar{\varphi} }
    \;,
  \end{equation}
  which will be used in expressing the local expansion of the billiard
  dynamics.

  \begin{lemma}[Asymptotic expansions of the change of coordinates]
    \label{lem_expansion_JacobiCoordinates}
    Fix
    $\bar{s}$ and $\bar{\varphi}$, where
    $
    \theta(\bar{s})
    =
    \bar{\varphi} - \frac{\pi}{2} + \bar{\omega}
    $.

    \begin{enumerate}[(i)]
      \item
        The change of coordinates
        \eqref{eqn_flowCoordinates_boundaryComponent}
        has the asymptotic expansion
        \begin{align*}
          \mathsf{x}(s, \omega)
          &=
          -
          \cos\bar{\varphi}\,(s - \bar{s})
          -
          \tan\bar{\varphi}
          \,[\cos\bar{\varphi}\,(s - \bar{s})]
          \,( \omega - \bar{\omega} )
          -
          \frac{1}{4}
          \,\tan\bar{\varphi}
          \,{\mathcal R}
          \,[\cos\bar{\varphi}\,(s - \bar{s})]^2
          \\
          &\qquad
          +
          \frac{1}{2}
          \,[\cos\bar{\varphi}\,(s - \bar{s})]
          \,( \omega - \bar{\omega} )^2
          +
          \frac{1}{4}
          \,{\mathcal R}
          \,[\cos\bar{\varphi}\,(s - \bar{s})]^2
          \,( \omega - \bar{\omega} )
          \\
          &\qquad
          +
          \frac{1}{24}
          \,[
          {\mathcal R}^2
          -
          \tan\bar{\varphi}
          \,{\mathcal R}^{(1)}
          ]
          \,[\cos\bar{\varphi}\,(s - \bar{s})]^3
          \\
          &\qquad
          +
          \operatorname{\mathcal{O}}_4(
          s - \bar{s},
          \omega - \bar{\omega}
          )
          \\
          \mathsf{y}(s, \omega)
          &=
          \omega - \bar{\omega}
          \;.
        \end{align*}

      \item
        The inverse of the change of coordinates
        \eqref{eqn_flowCoordinates_boundaryComponent}
        has the asymptotic expansion
        \begin{align*}
          [
          s(\mathsf{x}, \mathsf{y})
          -
          \bar{s}
          ]
          \,\cos\bar{\varphi}
          &=
          -
          \mathsf{x}
          +
          \tan\bar{\varphi}
          \,\mathsf{x}
          \,\mathsf{y}
          -
          \frac{1}{4}
          \,\tan\bar{\varphi}
          \,{\mathcal R}
          \,\mathsf{x}^2
          \\
          &\qquad
          -
          \frac{1}{2}
          \,[ 1 + 2\,\tan^2\bar{\varphi} ]
          \,\mathsf{x}
          \,\mathsf{y}^2
          +
          \frac{1}{4}
          \,[ 1 + 3\,\tan^2\bar{\varphi} ]
          \,{\mathcal R}
          \,\mathsf{x}^2
          \,\mathsf{y}
          \\
          &\qquad
          -
          \frac{1}{24}
          \,\Big[
          [ 1 + 3\,\tan^2\bar{\varphi} ]
          \,{\mathcal R}^2
          -
          \tan\bar{\varphi}
          \,{\mathcal R}^{(1)}
          \Big]
          \,\mathsf{x}^3
          +
          \operatorname{\mathcal{O}}_4(\mathsf{x}, \mathsf{y})
          \\
          \omega(\mathsf{x}, \mathsf{y})
          &=
          \bar{\omega}
          +
          \mathsf{y}
          \;.
        \end{align*}

      \item
        The expression for $\mathsf{z}$ in terms of $\mathsf{x}$
        and $\mathsf{y}$ has the asymptotic expansion
        \begin{align*}
          \frac{ \mathsf{z}(\mathsf{x}, \mathsf{y})
          }{ 1 + \tan^2\bar{\varphi} }
          &=
          -
          \frac{ \tan\bar{\varphi} }{ 1 + \tan^2\bar{\varphi} }
          \,\mathsf{x}
          +
          \mathsf{x}
          \,\mathsf{y}
          -
          \frac{1}{4}
          \,{\mathcal R}
          \,\mathsf{x}^2
          \\
          &\qquad
          -
          \tan\bar{\varphi}
          \,\mathsf{x}
          \,\mathsf{y}^2
          +
          \frac{3}{4}
          \,{\mathcal R}
          \,\tan\bar{\varphi}
          \,\mathsf{x}^2
          \,\mathsf{y}
          -
          \frac{1}{24}
          \,[ 3\,\tan\bar{\varphi}\,{\mathcal R}^2 - {\mathcal R}^{(1)} ]
          \,\mathsf{x}^3
          \\
          &\qquad
          +
          \operatorname{\mathcal{O}}_4(\mathsf{x}, \mathsf{y})
          \;.
        \end{align*}
    \end{enumerate}
  \end{lemma}
  \begin{proof}
    Recall the definition of the Jacobi coordinates
    \eqref{eqn_flowCoordinates_boundaryComponent}.
    With \eqref{eqn_definition_bMapAngle} we have
    \begin{align*}
      \mathsf{x}(s, \omega)
      &=
      \sin[\theta(\bar{s}) - \omega ]
      \,( s - \bar{s} )
      -
      \cos[\theta(\bar{s}) - \omega ]
      \,{\mathcal K}(\bar{s})
      \,\frac{1}{2}\,( s - \bar{s} )^2
      \\
      &\qquad
      -
      \Big[
      \sin[\theta(\bar{s}) - \omega ]
      \,{\mathcal K}(\bar{s})^2
      +
      \cos[\theta(\bar{s}) - \omega ]
      \,{\mathcal K}'(\bar{s})
      \Big]
      \,\frac{1}{6}\,( s - \bar{s} )^3
      +
      \operatorname{\mathcal{O}}( s - \bar{s} )^4
      \\
      &=
      -
      \frac{\cos[ \bar{\varphi} + \bar{\omega} - \omega ]
      }{ \cos\bar{\varphi} }
      \,[\cos\bar{\varphi}\,(s - \bar{s})]
      -
      \frac{1}{4}\,{\mathcal R}
      \,\frac{ \sin[ \bar{\varphi} + \bar{\omega} - \omega ]
      }{ \cos\bar{\varphi} }
      \,[\cos\bar{\varphi}\,(s - \bar{s})]^2
      \\
      &\qquad
      +
      \frac{1}{24}
      \,\Big[
      \frac{ \cos[ \bar{\varphi} + \bar{\omega} - \omega ]
      }{ \cos\bar{\varphi} }
      \,{\mathcal R}^2
      -
      \frac{ \sin[ \bar{\varphi} + \bar{\omega} - \omega ]
      }{ \cos\bar{\varphi} }
      \,{\mathcal R}^{(1)}
      \Big]
      \,[\cos\bar{\varphi}\,(s - \bar{s})]^3
      \\
      &\qquad
      +
      \operatorname{\mathcal{O}}( s - \bar{s} )^4
    \end{align*}
    and
    \begin{align*}
      \mathsf{z}(s, \omega)
      &=
      \cos[ \theta(\bar{s}) -  \omega ]
      \,( s - \bar{s} )
      +
      \sin[ \theta(\bar{s}) -  \omega ]
      \,{\mathcal K}(\bar{s})
      \,\frac{1}{2}\,( s - \bar{s} )^2
      \\
      &\qquad
      -
      \Big[
      \cos[ \theta(\bar{s}) -  \omega ]
      \,{\mathcal K}(\bar{s})^2
      -
      \sin[ \theta(\bar{s}) -  \omega ]
      \,{\mathcal K}'(\bar{s})
      \Big]
      \,\frac{1}{6}\,( s - \bar{s} )^3
      +
      \operatorname{\mathcal{O}}( s - \bar{s} )^4
      \\
      &=
      \frac{ \sin[ \bar{\varphi}  + \bar{\omega} -  \omega ]
      }{ \cos\bar{\varphi} }
      \,[\cos\bar{\varphi}\,(s - \bar{s})]
      -
      \frac{1}{4}
      \,{\mathcal R}
      \,\frac{ \cos[ \bar{\varphi} + \bar{\omega} -  \omega ]
      }{ \cos\bar{\varphi} }
      \,[\cos\bar{\varphi}\,(s - \bar{s})]^2
      \\
      &\qquad
      -
      \frac{1}{24}
      \,\Big[
      \frac{ \sin[ \bar{\varphi} + \bar{\omega} -  \omega ]
      }{ \cos\bar{\varphi} }
      \,{\mathcal R}^2
      +
      \frac{ \cos[ \bar{\varphi} + \bar{\omega} -  \omega ]
      }{ \cos\bar{\varphi} }
      \,{\mathcal R}^{(1)}
      \Big]
      \,[\cos\bar{\varphi}\,(s - \bar{s})]^3
      \\
      &\qquad
      +
      \operatorname{\mathcal{O}}( s - \bar{s} )^4
      \;.
    \end{align*}
    Hence it follows that
    $ s = s(\mathsf{x}, \mathsf{y}) $
    is locally given by
    \begin{align*}
      \cos\bar{\varphi}\,(s - \bar{s})
      &=
      -
      \frac{\cos\bar{\varphi}}{\cos[\bar{\varphi} - \mathsf{y}]}
      \,\mathsf{x}
      -
      \frac{1}{4}\,{\mathcal R}
      \,\tan[ \bar{\varphi} - \mathsf{y} ]
      \,\frac{\cos^2\bar{\varphi}
      }{\cos^2[\bar{\varphi} - \mathsf{y}]}
      \,\mathsf{x}^2
      \\
      &\qquad
      -
      \frac{1}{24}
      \,\Big[
      [ 1 + 3\,\tan^2[ \bar{\varphi} - \mathsf{y} ] ]
      \,{\mathcal R}^2
      -
      \tan[ \bar{\varphi} - \mathsf{y} ]
      \,{\mathcal R}^{(1)}
      \Big]
      \,\frac{\cos^3\bar{\varphi}
      }{\cos^3[\bar{\varphi} - \mathsf{y}]}
      \,\mathsf{x}^3
      \\
      &\qquad
      +
      \operatorname{\mathcal{O}}(\mathsf{x}^4)
      \;.
    \end{align*}
    Expanding also with respect to $\mathsf{y}$ yields the
    claimed expression for
    $
    \cos\bar{\varphi}\,(s - \bar{s})
    $.

    Substituting this expression for $s$ back into
    the above expansion for $\mathsf{z}$ yields the claimed
    expression for $\mathsf{z}(\mathsf{x}, \mathsf{y})$.
  \end{proof}

  \begin{lemma}[Asymptotic expansions of the billiard dynamics]
    \label{lem_billiard_localFlowCoordinates_localExpansion}
    Using the notation as in Lemma~\ref{lem_billiard_localFlowCoordinates}
    we have the following asymptotic expansions:
    \begin{enumerate}[(i)]
      \item (Free flight)
        The billiard flow starting at $\Gamma$
        at the moment right before hitting $\Gamma_1$
        is locally given by
        \begin{align*}
          \mathsf{x}_1
          &=
          \mathsf{x}_0
          +
          \bar{\tau}_{0,1}
          \,\mathsf{y}_0
          -
          \frac{1}{6}
          \,\bar{\tau}_{0,1}
          \,\mathsf{y}_0^3
          +
          \operatorname{\mathcal{O}}(\mathsf{y}_0^5)
          \;,\qquad
          \mathsf{y}_1
          =
          \mathsf{y}_0
          \;,
        \end{align*}
        where $\bar{\omega}_1 = \bar{\omega}_0$.

      \item (Reflection)
        Let
        $
        (\bar{x}, \bar{y})
        =
        \Gamma(\bar{s})
        $ and $\bar{\omega}$ be given.
        Then the reflection off of $\Gamma$ is locally
        given by
        \begin{align*}
          -\mathsf{x}_+
          &=
          \mathsf{x}_-
          -
          \frac{1}{2}
          \,\tan\bar{\varphi}
          \,{\mathcal R}
          \,\mathsf{x}_-^2
          -
          \frac{1}{2}
          \,[ 1 + 2\,\tan^2\bar{\varphi} ]
          \,{\mathcal R}
          \,\mathsf{x}_-^2
          \,\mathsf{y}_-
          \\
          &\qquad
          -
          \frac{1}{12}
          \,\Big[
          3\,[ 1 + \tan^2\bar{\varphi} ]\,{\mathcal R}^2
          +
          2\,\tan\bar{\varphi}
          \,{\mathcal R}^{(1)}
          \Big]
          \,\mathsf{x}_-^3
          \\
          &\qquad
          +
          \operatorname{\mathcal{O}}_4(\mathsf{x}_-, \mathsf{y}_-)
          \\
          -\mathsf{y}_+
          &=
          {\mathcal R}
          \,\mathsf{x}_-
          +
          \mathsf{y}_-
          +
          \tan\bar{\varphi}
          \,{\mathcal R}
          \,\mathsf{x}_-
          \,\mathsf{y}_-
          +
          \frac{1}{4}
          \,[ \tan\bar{\varphi} \,{\mathcal R}^2 + {\mathcal R}^{(1)} ]
          \,\mathsf{x}_-^2
          \\
          &\qquad
          +
          \frac{1}{2}
          \,[ 1 + 2\,\tan^2\bar{\varphi} ]
          \,{\mathcal R}
          \,\mathsf{x}_-
          \,\mathsf{y}_-^2
          \\
          &\qquad
          +
          \frac{1}{4}
          \,\Big[
          [ 1 + 3\,\tan^2\bar{\varphi} ]
          \,{\mathcal R}^2
          +
          2\,\tan\bar{\varphi}
          \,{\mathcal R}^{(1)}
          \Big]
          \,\mathsf{x}_-^2
          \,\mathsf{y}_-
          \\
          &\qquad
          +
          \frac{1}{24}
          \,\Big[
          [ 1 + 3\,\tan^2\bar{\varphi} ]
          \,{\mathcal R}^3
          +
          4\,\tan\bar{\varphi}
          \,{\mathcal R}
          \,{\mathcal R}^{(1)}
          +
          {\mathcal R}^{(2)}
          \Big]
          \,\mathsf{x}_-^3
          \\
          &\qquad
          +
          \operatorname{\mathcal{O}}_4(\mathsf{x}_-, \mathsf{y}_-)
        \end{align*}
        where
        ${\mathcal R}$, ${\mathcal R}^{(1)}$, ${\mathcal R}^{(2)}$,
        $\bar{\varphi}$
        correspond to the state right after the reflection.
    \end{enumerate}
  \end{lemma}
  \begin{proof}
    The expansion of the expression for the free flight is trivial.

    To obtain the expansion of the reflection observe first that
    by Lemma~\ref{lem_billiard_localFlowCoordinates} we have
    $
    \bar{\omega}^+
    =
    2\,\theta(\bar{s}) - \bar{\omega}^-
    $. Therefore \eqref{eqn_definition_bMapAngle} shows that
    \begin{align*}
      \bar{\varphi}_+
      =
      \theta(\bar{s})
      + \frac{\pi}{2}
      - \bar{\omega}_+
      =
      \frac{\pi}{2}
      - \theta(\bar{s})
      + \bar{\omega}_-
      =
      \pi
      -
      \bar{\varphi}_-
    \end{align*}
    and thus
    $
    \cos\bar{\varphi}_+
    =
    -\cos\bar{\varphi}_-
    $.
    Consequently
    $
    {\mathcal R}_+ = -{\mathcal R}_-
    $,
    $
    {\mathcal R}^{(1)}_+ = {\mathcal R}^{(1)}_-
    $,
    $
    {\mathcal R}^{(2)}_+ = -{\mathcal R}^{(2)}_-
    $.

    Using the result of
    Lemma~\ref{lem_expansion_JacobiCoordinates}
    we can solve
    \begin{align*}
      \mathsf{x}_-
      &=
      \int_{\bar{s}}^{s}
      \sin[
      \theta(\sigma)
      -
      \mathsf{y}_-
      -
      \bar{\omega}_-
      ]
      \,d\sigma
    \end{align*}
    for $s$
    \begin{align*}
      \cos\bar{\varphi}_-\,(s - \bar{s})
      &=
      -
      \mathsf{x}_-
      +
      \tan\bar{\varphi}_-
      \,\mathsf{x}_-
      \,\mathsf{y}_-
      -
      \frac{1}{4}
      \,\tan\bar{\varphi}_-
      \,{\mathcal R}_-
      \,\mathsf{x}_-^2
      \\
      &\qquad
      -
      \frac{1}{2}
      \,[ 1 + 2\,\tan^2\bar{\varphi}_- ]
      \,\mathsf{x}_-
      \,\mathsf{y}_-^2
      +
      \frac{1}{4}
      \,[ 1 + 3\,\tan^2\bar{\varphi}_- ]
      \,{\mathcal R}_-
      \,\mathsf{x}_-^2
      \,\mathsf{y}_-
      \\
      &\qquad
      -
      \frac{1}{24}
      \,\Big[
      [ 1 + 3\,\tan^2\bar{\varphi}_- ]
      \,{\mathcal R}_-^2
      -
      \tan\bar{\varphi}_-
      \,{\mathcal R}^{(1)}_-
      \Big]
      \,\mathsf{x}_-^3
      +
      \operatorname{\mathcal{O}}_4(\mathsf{x}_-, \mathsf{y}_-)
      \;,
    \end{align*}
    which can be rewritten as
    \begin{align*}
      \cos\bar{\varphi}_+\,(s - \bar{s})
      &=
      \mathsf{x}_-
      +
      \tan\bar{\varphi}_+
      \,\mathsf{x}_-
      \,\mathsf{y}_-
      +
      \frac{1}{4}
      \,\tan\bar{\varphi}_+
      \,{\mathcal R}_+
      \,\mathsf{x}_-^2
      \\
      &\qquad
      +
      \frac{1}{2}
      \,[ 1 + 2\,\tan^2\bar{\varphi}_+ ]
      \,\mathsf{x}_-
      \,\mathsf{y}_-^2
      +
      \frac{1}{4}
      \,[ 1 + 3\,\tan^2\bar{\varphi}_+ ]
      \,{\mathcal R}_+
      \,\mathsf{x}_-^2
      \,\mathsf{y}_-
      \\
      &\qquad
      +
      \frac{1}{24}
      \,\Big[
      [ 1 + 3\,\tan^2\bar{\varphi}_+ ]
      \,{\mathcal R}_+^2
      +
      \tan\bar{\varphi}_+
      \,{\mathcal R}^{(1)}_+
      \Big]
      \,\mathsf{x}_-^3
      +
      \operatorname{\mathcal{O}}_4(\mathsf{x}_-, \mathsf{y}_-)
    \end{align*}

    By
    Lemma~\ref{lem_billiard_localFlowCoordinates}
    we have
    $
    \mathsf{y}_+
    =
    2\,[
    \theta(s)
    -
    \theta(\bar{s})
    ]
    -
    \mathsf{y}_-
    $,
    whose expansion
    \begin{align*}
      -\mathsf{y}_+
      &=
      -2\,[
      \theta'(\bar{s})
      \,( s - \bar{s} )
      +
      \frac{1}{2}
      \,\theta''(\bar{s})
      \,( s - \bar{s} )^2
      +
      \frac{1}{6}
      \,\theta'''(\bar{s})
      \,( s - \bar{s} )^3
      ]
      +
      \mathsf{y}^-
      +
      \operatorname{\mathcal{O}}(s - \bar{s} )^4
      \\
      &=
      {\mathcal R}_+
      \,[\cos\bar{\varphi}_+\,( s - \bar{s})]
      +
      \frac{1}{4}
      \,{\mathcal R}^{(1)}_+
      \,[\cos\bar{\varphi}_+\,( s - \bar{s})]^2
      +
      \frac{1}{24}
      \,{\mathcal R}^{(2)}_+
      \,[\cos\bar{\varphi}_+\,( s - \bar{s})]^3
      \\
      &\qquad
      +
      \mathsf{y}^-
      +
      \operatorname{\mathcal{O}}(s - \bar{s} )^4
    \end{align*}
    therefore takes on the claimed form in terms of
    $\mathsf{x}_-$ and $\mathsf{y}_-$.

    For the corresponding $\mathsf{x}_+$ it follows from
    Lemma~\ref{lem_billiard_localFlowCoordinates}
    (or directly from \eqref{eqn_flowCoordinates_boundaryComponent})
    that
    \begin{align*}
      \mathsf{x}_+
      &=
      \int_{\bar{s}}^{s}
      \sin[
      \theta(\sigma)
      -
      \mathsf{y}_+
      -
      \bar{\omega}_+
      ]
      \,d\sigma
      \;.
    \end{align*}
    Therefore, the above expansions of $s$ and
    $\mathsf{y}^+$, when combined with the result of
    Lemma~\ref{lem_expansion_JacobiCoordinates}
    \begin{align*}
      -\mathsf{x}_+
      &=
      \cos\bar{\varphi}_+\,(s - \bar{s})
      +
      \tan\bar{\varphi}_+
      \,[\cos\bar{\varphi}_+\,(s - \bar{s})]
      \,\mathsf{y}_+
      +
      \frac{1}{4}
      \,\tan\bar{\varphi}_+
      \,{\mathcal R}_+
      \,[\cos\bar{\varphi}_+\,(s - \bar{s})]^2
      \\
      &\qquad
      -
      \frac{1}{2}
      \,[\cos\bar{\varphi}_+\,(s - \bar{s})]
      \,\mathsf{y}_+^2
      -
      \frac{1}{4}
      \,{\mathcal R}_+
      \,[\cos\bar{\varphi}_+\,(s - \bar{s})]^2
      \,\mathsf{y}_+
      \\
      &\qquad
      -
      \frac{1}{24}
      \,[
      {\mathcal R}_+^2
      -
      \tan\bar{\varphi}_+
      \,{\mathcal R}^{(1)}_+
      ]
      \,[\cos\bar{\varphi}_+\,(s - \bar{s})]^3
      \\
      &\qquad
      +
      \operatorname{\mathcal{O}}_4(
      s - \bar{s},
      \mathsf{y}_+
      )
      \;.
    \end{align*}
    Dropping the subscript $+$ for ${\mathcal R}_+$, ${\mathcal R}^{(1)}_+$,
    ${\mathcal R}^{(2)}_+$, $\varphi_+$
    yields the claimed expression for the local expansion of
    $\mathsf{x}_+$ and $-\mathsf{y}_+$.
  \end{proof}

\end{appendix}

\end{document}